\UseRawInputEncoding
\documentclass[a4paper,oneside,11pt]{article}
\usepackage{amsmath}
\usepackage{amssymb}
\usepackage{titlesec}
\usepackage{graphicx}
\usepackage{fancyhdr}
\usepackage{latexsym}
\usepackage{mathrsfs}
\usepackage{booktabs}
\usepackage{mathrsfs}
\usepackage{mathtools}
\usepackage{amsthm}
\usepackage{wasysym}
\usepackage{cite}
\usepackage{color}
\usepackage{amsfonts}
\usepackage{pgf,pgfarrows,pgfnodes,pgfautomata,pgfheaps}
\usepackage{amsmath,amssymb}
\usepackage{graphics}
\usepackage{multimedia}
\usepackage{graphics}
\usepackage{cite}
\usepackage{latexsym}
\usepackage{amsmath}
\usepackage{amssymb}
\usepackage[dvips]{epsfig}
\usepackage{amscd}
\numberwithin{equation}{section}

\newtheorem{theorem}{Theorem}[section]
\newtheorem{lemma}{Lemma}[section]
\newtheorem{remark}{Remark}[section]

\newcommand{\ba}{\begin{aligned}}
\newcommand{\ea}{\end{aligned}}
\newcommand{\be}{\begin{equation}}
\newcommand{\ee}{\end{equation}}

\newcommand{\bnn}{\begin{eqnarray*}}
\newcommand{\enn}{\end{eqnarray*}}

\newcommand{\thatsall}{\hfill$\Box$}
\date{}

\title{The Compressible Navier-Stokes Equations on the Multi-Connected Domains}

\author{Xinyu F{\small AN}$^{a}$, Song J{\small IANG}$^{a}$,  Jing L{\small I}$^{b,c}$ \thanks{Email addresses:  fanxinyu17@mails.ucas.edu.cn (X. Y. Fan), jiang@iapcm.ac.cn (S. Jiang), ajingli@gmail.com  (J. Li)  }  \\  {\normalsize a.  Institute of Applied Physics and Computational Mathematics,}\\
{\normalsize  Beijing 100088, P. R. China;}\\
{\normalsize b. Department of Mathematics, }\\ {\normalsize  Nanchang University, Nanchang 330031, P. R. China;} \\ {\normalsize c. Institute of Applied Mathematics, AMSS,} \\ {\normalsize \&   Hua Loo-Keng Key Laboratory of Mathematics,}\\
{\normalsize  Chinese Academy of Sciences,    Beijing 100190,
P. R. China }}
\begin{document}
\maketitle
\begin{abstract}
This paper investigates the isentropic compressible Navier-Stokes equations on $k$-connected domains $\Omega$ under Navier-slip boundary conditions. We study the multi-solvability of the stationary systems on $\Omega$, which is closely related with the Cauchy-Riemann systems and critical points of harmonic functions on $\Omega$. Then based on the structure of Green's functions, the commutator estimates are obtained on the circular domains and extended to $\Omega$ with the help of conformal mappings. Moreover, we will utilize these assertions to discuss the global well-posedness and large time behaviours of the non-stationary systems on $\Omega$ with large initial values containing vacuum.

\end{abstract}
\section{Introduction and main results}
\quad We assume that $\Omega$ is a $k$-connected $(k\geq 2)$ bounded smooth domain in $\mathbb{R}^2$, and the boundary of the domain $\partial\Omega$ consists of $k$ disjoint smooth simple closed curves $\{\Gamma_i\}_{i=0}^{k-1}$,
\begin{equation}\label{C101}
\partial\Omega=\bigcup_{i=0}^{k-1}\Gamma_i,~~\Gamma_i\cap\Gamma_j=\varnothing\ \  \mathrm{if}\ i\neq j,
\end{equation}
where $\Gamma_0$ is the outermost part of $\partial\Omega$. Formally speaking, there are $k-1$ disjoint holes with smooth boundaries on $\Omega$.
 
Let us investigate the isentropic compressible Navier-Stokes equations in $\Omega$:
\begin{equation}
\begin{cases}
\rho_t+\mathrm{div}(\rho u)=0, \\
(\rho u)_t+\mathrm{div}(\rho u\otimes u)+\nabla P=\mu\Delta u+\nabla\big((\mu+\lambda)\mathrm{div} u\big),   \label{11}
\end{cases}
\end{equation}
where $\rho$ and $u=(u_1, u_2)$ denote the unknown density and velocity field of the fluid. The pressure $P$ and the viscosity coefficients $\mu$ and $\lambda$ are given by 
\begin{equation*}
P=\rho^\gamma,~~0<\mu=\mathrm{constant},~~ \lambda(\rho)=\rho^\beta,
\end{equation*}
with positive constants $\beta>0$ and $\gamma>1.$
The system is imposed the initial conditions,
\begin{equation}\label{12}
\rho(x, 0)=\rho_0(x), \quad \rho u(x, 0)=m_0(x), \quad x\in \Omega,
\end{equation}
and the Navier-slip boundary conditions,
\begin{equation}
u\cdot n=0, \quad \mathrm{curl} u=-Ku\cdot n^{\bot},\quad  \mathrm{on}\ \partial \Omega, \label{15}
\end{equation}
where $n=(n_1,n_2)$ and $ n^\bot\triangleq (n_2,-n_1)$ denote the unit outer normal vector and the unit tangential vector on $\partial \Omega$, while $K$ is a non-negative smooth function on $\partial\Omega$.  
Then the closely related stationary system on $\Omega$ is given by
\begin{equation}\label{C102}
\begin{cases}
\mathrm{div}(\rho u)=0,\\
\mathrm{div}(\rho u\otimes u)+\nabla P=\mu\Delta u+\nabla\big((\mu+\lambda)\mathrm{div}u\big),\\
\end{cases}
\end{equation}
subject to the conditions
\begin{equation}\label{C103}
\begin{split}
\int_\Omega\rho\,dx=\int_\Omega\rho_0\,dx,~~
u\cdot n|_{\partial\Omega}=0,~~\mathrm{curl}u|_{\partial\Omega}=-Ku\cdot n^\bot.
\end{split}
\end{equation}

There is large number of literature about the strong solvability for the multidimensional compressible Navier-Stokes system.  The history of the area may trace back to Nash \cite{NA} and Serrin \cite{ser1},  who established the local existence and uniqueness of classical solutions respectively for the density away from vacuum. The first result of global classical solutions was due to Matsumura-Nishida \cite{1980The}   for initial data close to a non-vacuum equilibrium in $H^{3}.$   Hoff \cite{1995Global,hoff2005} then studied the problem with discontinuous initial data, and introduced a new type of a priori estimates based on the material derivatives  $\dot{u}$. The major breakthrough in the frame of weak solutions is due to Lions \cite{1998Mathematical}, where he successfully obtains the global existence of weak solutions just under the assumption that the energy is finite initially. For technical reasons, he assumed that the exponent  $\gamma\geq  9/5$, which was further released to the critical case $\gamma>3/2$ by Feireisl-Novotn\'{y}-Petzeltov\'{a} \cite{feireisl2004dynamics}. Recently, Huang-Li-Xin \cite{hlx21} and Li-Xin \cite{lx01} established the global existence and uniqueness of classical solutions, merely assumed the initial energy is small enough, where large oscillations are available. Cai-Li \cite{caili01} extended this result to bounded simply connected domains under Navier-slip boundary conditions and derived the exponential growth of classical solutions.

However, the global strong solvability theory for compressible flows upto now mainly deals with simply connected domains (including $\mathbb{R}^n$), which leaves the important case of multi-connected domains almost blank. The corresponding research for incompressible flows on multi-connected domains is fruitful, especially for the stationary system. In the well-known paper \cite{Leray}, Leray investigated the stationary Navier-Stokes equations in multi-connected domains, and proved the existence of solutions under non-homogeneous boundary conditions. Recently, Korobkov-Pileckas-Russo \cite{KPR} also extended the existence assertion to general compatible conditions,  based on the improved version of the Morse-Sard theorem.
In addition, an important multi-connected case is the 2D exterior domain. The classical papers by Gilbarg-Weinberger \cite{GW} and Amick \cite{ACJ} systematically studied the stationary system on 2D exterior domains. One may also refer to Galdi's monograph \cite{Galdi} for detailed discussions on related issues.

In summary, the analysis on the multi-connected domains $\Omega$ is challenging, which requires new observations.
The aim of this paper is to give a more complete discussion on the strong solutions to the stationary and non-stationary compressible Navier-Stokes systems in $\Omega$. We concentrate on two topics: the first one is the multi-solvable phenomenon of the stationary system \eqref{C102}--\eqref{C103} due to the $k$-connectedness of domains; the second one concerns the global well-posedness and large time behaviours of the non-stationary system \eqref{11}--\eqref{15} on $\Omega$.

To begin with, let us introduce the notations and conventions used throughout the paper. For $k\in \mathbb{N}^+$ and $p\in[1,\infty]$, the standard Lebesgue and Sobolev spaces over the domain $A$ are denoted by $L^p(A)$ and $W^{k,p}(A)$ respectively. In particular, $H^k(A)\triangleq W^{k,2}(A) $ and we introduce the notations:
\begin{equation*}
\begin{split}
&L^p\triangleq L^p(\Omega),~~W^{k,p}\triangleq W^{k,p}(\Omega),~~H^k\triangleq H^k(\Omega),\\
\tilde H^1\triangleq&\left\{u\in H^1(\Omega)\big|\,u \cdot n =0,\,{\rm curl}u =-Ku \cdot n^\bot\,\mbox{ on $\partial\Omega$}\right\}. 
\end{split}
\end{equation*}
The transpose gradient, material derivative, and  integral average of $f$ are denoted by
\begin{equation}\label{ou1r}
\nabla^{\bot}f\triangleq(-\partial_{2}f,\,\partial_{1}f),~~\frac{\mathrm{d}}{\mathrm{d}t}f=\dot{f}\triangleq\frac{\partial}{\partial t}f+u\cdot\nabla f,~~\hat{f}\triangleq\frac{1}{|\Omega|}\int_\Omega f\,dx.
\end{equation} 
\begin{remark}
By integrating \eqref{11}$_1$ over $\Omega\times[0,t]$, we apply \eqref{15} to infer that
\begin{equation*}
\int_\Omega\rho(x,t)\,dx=\int_\Omega\rho_0(x)\,dx.
\end{equation*}
Consequently, we will not distinguish $\hat{\rho}$ and $\hat{\rho}_0$ in the rest of the paper.
\end{remark}

In these terminologies, the first result of this paper systematically discusses the uniqueness of non-vacuum smooth solution to the stationary system \eqref{C102}--\eqref{C103}. Although such system is free of external forces, the non-vacuum  solution of it is still not unique in certain multi-connected domains, which is in sharp contrast to the simply connected domains.

\begin{theorem}\label{L11}
Let $\Omega\subset\mathbb{R}^2$ be a $k$-connected $(k\geq 2)$ smooth domain satisfying \eqref{C101}, then there are three cases of the stationary system \eqref{C102}--\eqref{C103}:

$\mathrm{a)}$. If $K$ is not identically $0$ on $\partial\Omega$, the unique smooth solution to the system \eqref{C102}--\eqref{C103} is the trivial one
$$\rho\equiv\hat\rho_0,~~u\equiv 0.$$

$\mathrm{b)}$. If $K\equiv 0$ and $\Omega$ is a 2-connected domain with $\partial\Omega$ consisting of two concentric circles, the non-vacuum smooth solutions to \eqref{C102}--\eqref{C103} are given by
$$\rho^{\gamma-1}(z)=\frac{\gamma-1}{\gamma}\left(C_2-\frac{C_1^2}{2|z|^2}\right),~~u_1(z)-iu_2(z)=\frac{iC_1}{z},~~z=x_1+ix_2,$$
where $C_1$ and $C_2$ are two real numbers.

$\mathrm{c).}$ If $K\equiv 0$ and $\Omega$ is any $k$-connected domain different from $\mathrm{b)}$,   the only non-vacuum smooth solution to  \eqref{C102}--\eqref{C103} is the trivial one.
\end{theorem}
\begin{remark}
According to the Lions' monograph \cite[Chapter 6]{1998Mathematical}, the non-uniqueness of solution to the system \eqref{C102}--\eqref{C103} is usually caused by vacuum. However, the non-trivial solution in Theorem \ref{L11} can be constructed out of vacuum. Such multi-solvable phenomenon is due to the $k$-connectedness of $\Omega$.
\end{remark}
\begin{remark}
Note that the density of the non-trivial solution on the annulus domain given by case $b)$ is not constant, thus it is not a solution to the stationary incompressible Navier-Stokes equations.
\end{remark}

The above classification suggests us considering the global well-posedness and large time behaviour of the non-stationary system \eqref{11}--\eqref{15} satisfying the case a) of Theorem \ref{L11}, in which the corresponding stationary system is uniquely solvable.

\begin{theorem}\label{T1}
Let $\Omega\subset\mathbb{R}^2$ be a $k$-connected $(k\geq 2)$ bounded smooth domain satisfying \eqref{C101}. We assume that $K$ is not identically $0$ on $\partial\Omega$, 
\begin{equation}\label{17}
\beta>4/3,~~\gamma>1, 
\end{equation}
and the initial values $(\rho_0\geq 0,  m_0)$ are provided by:
\begin{equation}\label{18}
\rho_0\in W^{1,q}\,(q>2),~~u_0 \in\tilde H^1 ,~~m_0 = \rho_0u_0. 
\end{equation}

Then the system \eqref{11}--\eqref{15} admits the unique global strong solution $(\rho, u)$ in $\Omega\times(0, \infty)$ with the following regularity assertions: for any $0<T<\infty$,
\begin{equation}\label{19}
\begin{cases}
\rho\in C\big([0, T]; W^{1, q} \big),\  \rho_t\in L^\infty\big(0, T ;L^2 \big),\\
u\in L^\infty\big(0, T;H^1 \big)\cap L^{1+1/q}\big(0, T;W^{2, q} \big), \\
\sqrt tu\in L^\infty\big(0, T;H^{2} \big)\cap L^2\big(0, T;W^{2, q} \big),\ \sqrt tu_t\in L^2\big(0, T;H^1 \big), \\ 
\rho u\in C\big([0, T] ;L^2 \big),\  \sqrt{\rho}u_t\in L^2\big((0, T)\times \Omega\big).
\end{cases}
\end{equation}

Moreover the large time behaviours of this strong solution are given below:

1). $($Uniform bounded$)$ There is a  constant $\mathbf{C}$ depending on $\Omega$, $\mu$, $\beta$, $\gamma$, $\|\rho_0\|_{L^{\infty} }$, and $\|u_0\|_{H^1 }$, such that for any $0<T<\infty$,
\begin{equation}\label{md}
\sup_{0\leq t\leq T}\|\rho(\cdot,t)\|_{L^\infty }\leq\mathbf{C}.
\end{equation}

2). $($Exponential decay$)$ For any $p\in[1,\infty)$, there are constants $\mathbf{C}'$ and $\alpha$ determined by $p$, $\Omega$, $\mu$, $\beta$, $\gamma$, $\|\rho_0\|_{L^{\infty} }$, and $\|u_0\|_{H^1 }$, such that for any $1\leq t<\infty$,
\begin{equation}\label{zs}
\|\rho(\cdot,t)-\hat\rho\|_{L^p }+\|\nabla u(\cdot,t)\|_{L^p }\leq\mathbf{C}'e^{-\alpha t}.
\end{equation}
\end{theorem}
\begin{remark}
For the case of $K\equiv 0$ on $\partial\Omega$, the arguments below also ensure that the system \eqref{11}--\eqref{15} admits the unique global solution satisfying \eqref{19}. However, the large time behaviours \eqref{md} and \eqref{zs} are temporarily not available.
\end{remark}
\begin{remark}
The model in \eqref{11} was introduced by Vaigant-Kazhikhov \cite{vaigant1995}, in which they proved the global well-posedness of strong solutions under the condition that $\beta>3$.
As indicated by \cite{fmd}, the restriction \eqref{17} for $\beta>4/3$ seems still optimal up to now.
\end{remark}
 
We make some comments on the analysis of the paper. This article focuses on the new phenomenons due to $k$-connectedness of the domain $\Omega$, and the key ingredients can be summarized in two aspects.

\textit{1. The Cauchy-Riemann system on multi-connected domains.}

Our arguments start from the homogeneous Cauchy-Riemann system,
\begin{equation}\label{CA101}
\begin{cases}
\mathrm{div}u=0,~~\mbox{in $\Omega$},\\
\mathrm{curl}u=0,~~\mbox{in $\Omega$},\\
u\cdot n=0,~~\mbox{on $\partial\Omega$},
\end{cases}
\end{equation}
which is not uniquely solvable on multi-connected domains.
We deduce that the solution of \eqref{CA101} must be given by $u=\nabla^\bot\omega$, where $\omega$ is a harmonic function which equals to (different) constant on each component of $\partial\Omega$.
Compared with the classical arguments in Temam's monograph \cite{RT}, where $u=\nabla\Phi$ and $\Phi$ is discontinuous along $k-1$ curves, the potential $\omega$ is smooth through out $\overline{\Omega}$ and provides further links between the Cauchy-Riemann systems with critical points of harmonic functions.

Precisely, according to the index theory \cite{AM}, the number of critical points of $\omega$, which is closely related with the multi-connectedness of $\Omega$, can not exceed $2k-4$. Such fact leads to the div-curl type estimates
\begin{equation}\label{CA102}
\|\nabla u\|_{L^p}\leq C\bigg(\|\mathrm{div}u\|_{L^p}+\|\mathrm{curl}u\|_{L^p}+\sum_{j=1}^{2k-3}|u(\xi_j)|\bigg),~~\mathrm{with}\ p>2,
\end{equation}
where $\{\xi_j\}_{j=1}^{2k-3}$ are arbitrary distinct points on $\overline{\Omega}$. In particular, when $\Omega$ is a $2$-connected bounded domain, one additional point is sufficient to give \eqref{CA102}.
 
More inspirationally, the velocity field $u$ of the  solution to the stationary system \eqref{C102}--\eqref{C103} must solve \eqref{CA101}, thus we can substitute $u=\nabla^\bot\omega$ into it and obtain the necessary condition of steady states that
\begin{equation}\label{CA103}
\nabla^\bot\omega\cdot\nabla|\nabla\omega|^2=0.
\end{equation}
Note that $\omega^\bot$ is tangent to the level curves of $\omega$, thus \eqref{CA103} implies that $|\nabla\omega|$ must be constant along each level curve. Once again, the critical points of $\omega$ play the essential role on determining whether $\omega$ (or $u$) is the trivial solution. Moreover, \eqref{CA103} gives non-trivial solutions on certain $2$-connected domains, which is in sharp contrast to simply connected domains.
 
\textit{2. Green's functions on multi-connected domains}
 
As indicated by Lions \cite[Chapter 5]{1998Mathematical}, the commutator estimate lies in the central position of the global well-posedness theory. We will extend it from Cauchy problems to general $k$-connected domains with the help of Green's functions.  
The key issue relies on commuting $\partial_{x_i}$ and $\partial_{y_i}$ of the Green's function  $N(x,y)$, which requires proper controls on the ``commutator",
$$\partial_{x_i}N(x,y)+\partial_{y_i}N(x,y).$$
The standard estimates (\cite{1972The}) on the Green's functions merely provide that
$$\big|\partial_{x_i}N(x,y)+\partial_{y_i}N(x,y)\big|\leq |\nabla N(x,y)|\leq C|x-y|^{-1}.$$ 
However, there is a further cancellation structure of $N(x,y)$ which reduces the singularity of the above estimates in some sense to
\begin{equation}\label{CA104}
\big| \partial_{x_i}N(x,y)+ \partial_{y_i}N(x,y)\big|\leq C.
\end{equation}

According to the theory of conformal mappings, each multi-connected domain is conformally equivalent with some circular domain whose boundary consists of closed circles. While \eqref{CA104} turns out to be an intrinsic property of the Green's functions, thus our strategy is to prove \eqref{CA104} first on circular domains and then extend it to $\Omega$ with the help of conformal mappings. Such method seems applicable for other boundary value problems as well. Moreover, we clearly describe the structure of Green's functions on multi-connected domains: There is a principal part $N_j(x,y)$ with the explicit formula near each component $\Gamma_j$ of $\partial\Omega$, which completely describes the singularity of $N(x,y)$ near $\Gamma_j$ as
$$N(x,y)=N_j(x,y)+R_j(x,y),$$
where $R_j(x,y)$ is a regular function when $x$ approaches $\Gamma_j$. Consequently, it is sufficient to check \eqref{CA104} for $N_j(x,y)$.

The rest of paper is organized as follows: Some basic facts of the complex analysis are collected in Section \ref{S2}, then we utilize these tools to prove Theorem \ref{L11} in Section \ref{S3}, and establish the commutator theory on multi-connected domains in Section \ref{S4}. After these preparations, we carry out a priori estimates in Section \ref{S5}, and finish the proof of Theorem \ref{T1} in Section \ref{S6}.

\section{Elementary complex analysis}\label{S2}
\quad
Let us recall some classical results in the complex analysis, which will play significant roles in our later arguments.
Suppose that $\Omega\subset\mathbb{R}^2$ is a bounded smooth domain, then for each point $(x,y)\in\Omega$, we introduce the complex variable $z=x+iy$ and $\bar{z}=x-iy$.
The change of coordinates provides that
\begin{equation}\label{cc}
\begin{split}
\partial_z=\frac{1}{2}(\partial_{x}-i\partial_{y}),&\quad \partial_{\bar z}=\frac{1}{2}(\partial_{x}+i\partial_{y}),\\
\partial_{x}=\partial_z+\partial_{\bar{z}},&\quad \partial_{y}=i(\partial_z-\partial_{\bar{z}}),
\end{split}
\end{equation}
where the bar denotes the complex conjugate. A   function $f=u+iv$ defined on $\Omega$ is called analytic, if the real part $u$ and the imaginary part $v$ satisfy the Cauchy-Riemann equations
\begin{equation}\label{CR}
\begin{cases}
\partial_{x}u=\partial_{y}v,\\
\partial_{y}u=-\partial_{x}v,\\
\end{cases}
\end{equation} 
In view of \eqref{cc}, the above equation is equivalent with $\partial_{\bar{z}}f=0$, consequently, for $\alpha=\alpha_1+i\alpha_2$, the directional derivative with respect to $\alpha$ of an analytic function $f$ is simply given by
\begin{equation}\label{CA201}
(\alpha_1\cdot\partial_x+\alpha_2\cdot\partial_y)f
=(\alpha_1+i\alpha_2)\cdot\partial_zf+(\alpha_1-i\alpha_2)\cdot\partial_{\bar{z}}f=\alpha\cdot f'. 
\end{equation}

Then we list some elementary properties of analytic functions, whose proof can be found in the text book (see \cite{2007Complex}). 

\begin{lemma}\label{L24}
Suppose that $\Omega\subset\mathbb{C}$ is an open bounded connected domain and $f\in C(\overline{\Omega})$ is an analytic function in $\Omega$, then we deduce that:
 
$\mathrm{a).}$ $\mathrm{(Maximum\ principle)}$ 
\begin{equation}\label{CA202}\max_{z\in\overline{\Omega}}|f(z)|= \max_{z\in\partial\Omega}|f(z)|.
\end{equation}
Thus if $f\equiv 0$ on $\partial\Omega$, then we also have $f\equiv 0$ in $\Omega$. Note that \eqref{CA202} is also valid for harmonic functions.

$\mathrm{b).}$ $\mathrm{(Uniqueness\ theorem)}$ If $f$ is non-constant in $\Omega$, then
the zeros of $f$ in $\Omega$ must be isolated. 

$\mathrm{c).}$ If the real part of $f$ is a constant, $\mathrm{Re}f\equiv c_1$, then the imaginary part of $f$ is a constant as well,
$\mathrm{Im}f\equiv c_2$.
 
\end{lemma}
%When $\Omega$ is the annulus domain $A=\{z\big|a<|z|<1\}$ for some $0<a<1$, any analytic function $f$ has Laurent expansion.
%\begin{lemma}\label{L25}
%Assume that $f$ is analytic in $A$, then for $z\in A$, $f(z)$ can be written as the power series
%$$f(z)=\sum_{k=-\infty}^\infty a_kz^k,$$
%where $a_k\in\mathbb{C}$. Moreover, the series is absolutely convergent on any compact subset of $A$.
%\end{lemma}

The classical arguments on conformal mappings ensure that we can transform the general $k$-connected domains to some canonical domains with simple geometric properties. We state some typical results (see \cite{GOL}) which will be repeatedly used.

\begin{lemma}\label{L27}
Suppose that $\Omega$ is a smooth $k$-connected domain given by \eqref{C101}, and $\mathcal{D}$ is a bounded circular domain with each component of $\partial\mathcal{D}$ a closed complete circle, then there is a conformal mapping $\varphi$ between $\mathcal{D}$ and $\Omega$ which is extended to the closure as 
\begin{equation*}
\varphi:\overline{\mathcal{D}}\rightarrow\overline{\Omega}.
\end{equation*}
In particular, if $\Omega$ is simply connected,  $\mathcal{D}$ can be chosen as the unit disc; if $\Omega$ is a 2-connected domain, $\mathcal{D}$ can be chosen as a concentric annulus $A=\{r<|z|<1\}$.
\end{lemma}

We need some properties of conformal mappings, see \cite{FLL, 1961On} for the detailed proof. %We list some basic properties of conformal mappings, it preserves the distance and angel.
\begin{lemma}\label{L28}
Suppose that $\varphi$ is a conformal mapping between two bounded smooth path-connected domains $\Omega_1$ and $\Omega_2$, which is extends smoothly to the closure:
$$\varphi:\overline{\Omega}_1\rightarrow\overline{\Omega}_2.$$

$\mathrm{a).}$ The derivative of $\varphi$ is non-zero and there is a constant $C$ such that,
$$\varphi'(z)=\partial_z\varphi(z)\neq 0,~~|\varphi'(z)|\leq C,~~\forall z\in\Omega_1.$$
Therefore we can find a constant $C_\varphi$ such that 
$$C_{\varphi}^{-1}\,|z_1-z_2|\leq|\varphi(z_1)-\varphi(z_2)|\leq C_{\varphi}\,|z_1-z_2|.$$

$\mathrm{b).}$ Suppose that $\gamma_1$ and $\gamma_2$ are two smooth curves satisfying for some $z\in\Omega_1$,
\begin{equation*}
\begin{split}
\gamma_i(t):&(-1,1)\rightarrow\Omega_1,\ i=1,2,\\
&\gamma_1(0)=\gamma_2(0)=z.
\end{split}
\end{equation*}
Then the conformal mappings preserve angles in the sense of
$$|\varphi'(z)|\cdot\big\langle \gamma_1'(0), \gamma_2'(0)\big\rangle
=\big\langle (\varphi\circ\gamma_1)'(0), (\varphi\circ\gamma_2)'(0)\big\rangle,$$
where $\langle\cdot,\cdot\rangle$ is the inner product.
In particular, if $z_0\in\partial\Omega_1$ and $\alpha=x_0+iy_0$ is the unit outer normal vector at $z_0$, then $\varphi'(z_0)\cdot\alpha$ is an outer normal vector of $\partial\Omega_2$ at $\varphi(z_0)$.
\end{lemma}

Let $\Omega$ be a $k$-connected domain satisfying \eqref{C101}. For $j\in\{0,1,\cdots,k-1\}$, let us    consider the Dirichlet problems on $\Omega$, 
\begin{equation}\label{hm}
\begin{cases}
\Delta\omega_j=0\, \mathrm{in}\ \Omega,\\
\omega_j=1\ \mathrm{on}\ \Gamma_j,~~ \omega_j=0\ \mathrm{on}\ \partial\Omega\setminus \Gamma_j.
\end{cases}
\end{equation}
Following the monograph of Bergman \cite[Chapter V]{B}, the unique smooth solution $\omega_j$ to \eqref{hm} is called the $j$-th harmonic measure of $\Omega$. We mention that the maximum principle  in Lemma \ref{L24} ensures that $\omega_j$ are non-negative and $\sum_j\omega_j\equiv 1$ in $\Omega$.

The lemma below due to Alessandrini-Magnanini \cite{AM} describes the number of critical points of harmonic measures, which plays an important role in Section \ref{S3}. See also \cite{LY} for related topics.
\begin{lemma}\label{LL24}
Suppose that $\Omega$ is a $k(\geq 2)$-connected domain satisfying \eqref{C101} and $\{\beta_j\}_{j=0}^{k-1}$ are $k-1$ real numbers not coincide. Let us consider the unique smooth solution $u$ to the Dirichlet problem:
\begin{equation*}
\begin{cases}
\Delta u=0\ \mathrm{in}\ \Omega,\\
u\equiv \beta_j\ \mathrm{on}\ \Gamma_j,~~j\in\{0,1,\cdots,k-1\}.
\end{cases}
\end{equation*}
We deduce that $u$ has finite many isolated critical points $\{\xi_i\}$ in $\overline{\Omega}$ with $|\nabla u(\xi_i)|=0$, and it holds that
\begin{equation*}
\sum_{\xi_i\in\Omega}m(\xi_i)+\frac{1}{2}\sum_{\xi_j\in\partial\Omega}m(\xi_j)=k-2,
\end{equation*}
where $m(\xi_i)$ is the multiplicity of the critical point $\xi_i$. Moreover, the level set $\{x\in\Omega\,|\,u(x)=u(\xi_i)\}$ near $\xi_i$ consists of finite many simple curves intersecting at $\xi_i$.
\end{lemma}

\section{Stationary systems on multi-connected domains}\label{S3}

\quad In this section, we discuss the uniqueness of smooth solution to the stationary system \eqref{C102}--\eqref{C103} on general multi-connected domains. In Subsection \ref{SS31}, some crucial technical assertions concerning the Cauchy-Riemann systems are established, then we  utilize them to prove Theorem \ref{L11} in Subsection \ref{SS32}.

\subsection{The Cauchy-Riemann systems on multi-connected domains}\label{SS31}
\quad Let $\Omega$ be a $k$-connected domain satisfying \eqref{C101}:
\begin{equation*}
\partial\Omega=\bigcup_{i=0}^{k-1}\Gamma_i,~~\Gamma_i\cap\Gamma_j=\varnothing\ \  \mathrm{if}\ i\neq j,
\end{equation*}
where $\Gamma_i$ are smooth simple closed curves, and we set $z=x+iy$. Let us consider the following homogeneous Cauchy-Riemann system:
\begin{equation}\label{cr}
\begin{cases}
\mathrm{div}u=0\, &\mathrm{in}\ \Omega,\\
\mathrm{curl}u=0 &\mathrm{in}\ \Omega,\\
u\cdot n=0 &\mathrm{on}\ \partial\Omega.
\end{cases}
\end{equation}
Note that \eqref{cr} is uniquely solved by the trivial solution $u\equiv0$ on simply connected domains (see \cite[Section 6.3]{1998Mathematical} for example). However, such system admits non-trivial solutions on multi-connected domains, and our first task is to give the exact form of these solutions.

In fact, the solutions to \eqref{cr} must be derivatives of the linear combination of harmonic measures.% and can be stated as follows in detail.
\begin{lemma}\label{pr1}
Let $\Omega$ be a $k$ connected domain satisfying \eqref{C101}, and $u$ solve the elliptic system \eqref{cr}, then there exists a sequence of real numbers $\{c_l\}_{l=1}^{n-1}$ such that 
\begin{equation}\notag
 u_2+iu_1= \sum_{l=1}^{n-1}c_l \,\partial_z\omega_l,
\end{equation}
where $\omega_l$ is the $l$-th harmonic measure of $\Omega$. Equivalently,  
\begin{equation*}
F(z)\triangleq u_1-iu_2=i\,\partial_z\omega, 
\end{equation*}
where $\omega$ is a harmonic function on $\Omega$, which is constant on each component of $\partial\Omega$.
\end{lemma}
\begin{proof}
According to \eqref{cr}, the complex-valued function $F(z)=u_1-iu_2$ satisfies the Cauchy-Riemann equations \eqref{CR}, thus $F(z)$ is  analytic in $\Omega$. For $j\in\{1,\cdots,k-1\}$, we check the period of $F$ with respect to $\Gamma_j$ by
\begin{equation*} 
\begin{split}
\int_{\Gamma_j}F(z)\,dz
&=\int_{\Gamma_j}(u_1-iu_2) (dx+idy)\\
&=\int_{\Gamma_j}\left(u_1\,dx+u_2\,dy\right)+i\int_{\Gamma_j}\left(u_1\,dy-u_2\,dx\right).
\end{split}
\end{equation*}
Let $s$ be the arc-length parameter of $\Gamma_j$ and $\Gamma_j$ be parametrized by $\gamma_j(s):[0,l]\rightarrow \Gamma_j$, satisfying
\begin{equation*} 
\gamma_j(0)=\gamma_j(l),\quad \gamma_j([0,l])=\Gamma_j,\quad \gamma'_j(s)=n^{\bot},
\end{equation*}
where $n^{\bot}=(n_2,-n_1)$ is the unit tangential vector on $\Gamma_j$. Consequently, in view of the boundary condition \eqref{cr}$_3$, we derive that
$$\int_{\Gamma_j}\left(u_1\,dy-u_2\,dx\right) =-\int_0^l u\cdot n(\gamma_j(s))\,ds=0,$$
which also guarantees that
\begin{equation}\label{CC301}
b_j\triangleq\int_{\Gamma_j}F(z)\,dz
=\int_{\Gamma_j}\left(u_1\,dx+u_2\,dy\right)\in\mathbb{R}.
\end{equation}

Next, suppose that $\omega_l$ is the harmonic measure given by \eqref{hm}. Then the boundary conditions \eqref{hm}$_2$ ensure that
\begin{equation}\label{00}
n\cdot\nabla^\bot\omega_l=-n^{\bot}\cdot\nabla\omega_l=0\quad \mathrm{on}\  \partial\Omega,
\end{equation}
which together with \eqref{hm}$_1$ shows that $(\partial_{y}\omega_l,-\partial_{x}\omega_l)$ satisfies \eqref{cr} as well.
Therefore, according to \eqref{CC301}, if we define $F_l(z)\triangleq\partial_{y}\omega_l+i\partial_{x}\omega_l$, then $F_l(z)$ is analytic in $\Omega$ and also satisfies that 
\begin{equation}\notag
a_{jl}\triangleq\int_{\Gamma_j}F_l(z)dz
=\int_{\Gamma_j}\left(\partial_{y}\omega_l\,dx-\partial_{x}\omega_l\,dy\right)\in\mathbb{R},~~j,l\in\{1,\cdots,k-1\}.
\end{equation}

The arguments in \cite[Chapter V]{B} provide the key property that the matrix $\{a_{jl}\}$ is real symmetric and invertible. Thus there exists a sequence of real numbers $\{\lambda_l\}_{l=1}^{n-1}$ such that
$$b_j=\sum_{l=1}^{n-1}\lambda_l\cdot a_{jl}\ \Leftrightarrow\ \int_{\Gamma_j}F(z)\,dz=\sum_{l=1}^{n-1} \lambda_{l} \cdot\int_{\Gamma_j}F_l(z)\,dz,$$ 
for any $j\in\{1,\cdots,k-1\}$. Hence, we denote $G(z)\triangleq F(z)-\sum_{l=1}^{n-1}\lambda_{l}\cdot F_l(z)$ and directly derive that 
\begin{equation}\notag
\int_{\Gamma_j}G(z)\,dz=0,\quad \mathrm{for \ any \ } j\in\{1,\cdots,k-1\},
\end{equation}
which means that for any smooth closed curve $\Gamma$ in $\Omega$, one has $\int_{\Gamma}G(z)\,dz=0$. Thus according to \cite[Chapter V]{B}, $G(z)$ is free of period and the single-valued primitive $H(z)$ of $iG(z)$ satisfying  $\partial_z H(z)=iG(z)$ is well defined in $\Omega$. 

In particular, let $h(z)\triangleq\mathrm{Re} H(z)$ be the real part of $H(z)$, then the definition of $G(z)$ together with \eqref{cc} and Cauchy-Riemann equations guarantees that 
\begin{equation}\label{g1}
\begin{split} 
\partial_z H(z)&=\partial_{x}h-i\,\partial_{y}h\\
&=\left(u_2+\sum_{l=1}^{n-1}\lambda_l\cdot\partial_{x}\omega_l\right)+i\left(u_1-\sum_{l=1}^{n-1}\lambda_l\cdot\partial_{y}\omega_l\right).
\end{split}
\end{equation}
In view of the above equation, we argue that
$$\partial_{x}h=u_2+\sum_{l=1}^{n-1}\lambda_l\cdot\partial_{x}\omega_l,~~\partial_{y}h=-u_1+\sum_{l=1}^{n-1}\lambda_l\cdot\partial_{y}\omega_l.$$
Moreover \eqref{cr}$_3$ together with \eqref{00} shows that the tangential derivatives of $h$ vanish on $\partial\Omega$, since 
$$n^{\bot} \cdot \nabla h=u \cdot n+\sum_{l=1}^{n-1} \lambda_{l}\,(n^\bot \cdot \nabla \omega_{l})=0,\ \mathrm{on}\ \partial\Omega.$$ 
Hence there exist real numbers $\{\beta_{l}\}_{l=0}^{k-1}$, such that $h\equiv\beta_{l}$ on $\Gamma_{l}$. Note that $h$ is harmonic in $\Omega$, then the maximum principle (Lemma \ref{L24}) provides that  
\begin{equation}\label{gg}
h=\beta_0\omega_0+\sum_{l=1}^{n-1}\beta_l\,\omega_l=\beta_0+\sum_{l=1}^{n-1}(\beta_l-1)\,\omega_l,~~~~\mathrm{since}\ \sum_{j=0}^{k-1}\omega_j\equiv 1.
\end{equation}

Substituting \eqref{gg} into \eqref{g1} implies that  
\begin{equation}\begin{split}\notag
u_{2}+i u_{1}&=\partial_{z}H(z)-\sum_{l=1}^{n-1} \lambda_{l}\,(\partial_{x}\omega_l-i\,\partial_{y}\omega_l)\\
&=2\partial_{z} h-2 \sum_{l=1}^{n-1} \lambda_{l}\,\partial_{z} \omega_{l} =\sum_{l=1}^{n-1}c_l\,\partial_{z} \omega_{l},
\end{split}\end{equation}
with $c_l\triangleq 2\,(\beta_{l}-\lambda_{l}-1)$, which also gives  
$$F(z)=u_1-iu_2=i\,\partial_z\omega,~~\omega\triangleq \sum_{l=1}^{n-1}c_l\,\partial\omega_l.$$
The proof is therefore completed.
\end{proof}
\begin{remark}
By deleting a set of smooth curves $\{\gamma_j\}_{j=1}^{k-1}$ from $\Omega$, the classical arguments in \cite{RT} indicate that $u=\nabla\Phi$ in the remaining simply connected domain $\Omega\setminus(\bigcup_{j=1}^{k-1}\gamma_j)$. However, $\Phi$ is discontinuous across each curve $\gamma_i$, thus the regularity of it becomes worse as the number of components increasing. In contrast of it, the potential $\omega$ in Lemma \ref{pr1} is always smooth over $\overline{\Omega}$.
\end{remark}

The Cauchy-Riemann systems \eqref{cr} are closely related with the div-curl type estimates in the multi-connected domains:
\begin{equation}\label{C301}
\|\nabla u\|_{L^p}\leq C(\|\mathrm{div}u\|_{L^p}+\|\mathrm{curl} u\|_{L^p})~~~~\mathrm{for}\ u\in W^{1,p}\ \mathrm{with}\ u\cdot n\big|_{\partial\Omega}=0.
\end{equation}
As shown in \cite{VW}, \eqref{C301} is not valid
for multi-connected domains, and some extra terms are required on the right hand side of \eqref{C301}. We begin with a basic version of it.

\begin{lemma} 
Let $\Omega\subset\mathbb{R}^2$ be a bounded smooth domain, and $u\in W^{1,p}$ with $p\in(1,\infty)$ satisfy $u\cdot n=0\ \mathrm{on}\ \partial\Omega$, then there is a constant $C$ depending on $p$ and $\Omega$ such that
\begin{equation}\label{C302}
\|\nabla u\|_{L^p}\leq C\,(\|\mathrm{div}u\|_{L^p}+\|\mathrm{curl} u\|_{L^p}+\|u\|_{L^p}).
\end{equation}
\end{lemma}
\begin{proof}
According to \cite{ADN}, the elliptic system
\begin{equation}\label{hc4}
\begin{cases}
\mathrm{div}u=f\quad &\mathrm{in}\ \Omega,\\
\mathrm{curl}u=g &\mathrm{in}\ \Omega,\\
u\cdot n=0 &\mathrm{on}\ \partial\Omega.
\end{cases}
\end{equation}
is well-posed and  the following elliptic estimate holds,
$$\|\nabla u\|_{L^p}\leq C(\Omega,p)(\|f\|_{L^p}+\|g\|_{L^p}+\|u\|_{L^p}),$$
which implies \eqref{C302} directly.
\end{proof}

\begin{remark}
Note that $($\ref{C302}$)$ is derived from the elliptic estimates of the system \eqref{hc4},
where $u\cdot n|_{\partial\Omega}=0$ provides  proper boundary conditions. 
However, the following system
\begin{equation*}
\begin{cases}
\Delta u=\nabla\mathrm{div}u+\nabla^\bot\mathrm{curl}u\quad &\mathrm{in}\ \Omega,\\
u\cdot n=0\ &\mathrm{on}\ \partial\Omega,
\end{cases}
\end{equation*}
is not well-posed and fails to give $($\ref{C302}$)$ directly. 
Such example partly reflects the stronger properties of analytic functions than harmonic functions.
\end{remark}

Now let us make use of Lemma \ref{pr1} to improve the estimates \eqref{C302}.  

\begin{lemma}\label{le3.2}
Suppose that $\Omega$ is a $k$-connected domain satisfying \eqref{C101} and $\{\xi_j\}_{j=1}^{2k-3}$ are any $2k-3$ distinct points on $\overline\Omega$. We suppose that $u\in W^{1,p} \ (p>2)$ satisfies $u\cdot n=0$ on $\partial\Omega$, then there is a constant $C$ depending on $p$ and $\Omega$, such that
\begin{equation}\label{g4} 
\|\nabla u\|_{L^p}\leq C \bigg(\|\mathrm{div}u\|_{L^p}+\|\mathrm{curl}u\|_{L^p}+\sum_{j=1}^{2k-3}\big|u(\xi_j)\big|\bigg).
\end{equation}
In particular, for the case $p=2$, the function $u\in W^{1,2}(\Omega)$ is not point-wisely defined, thus \eqref{g4} can be modified as 
 \begin{equation}\label{g44} 
\|\nabla u\|_{L^2}\leq C \bigg(\|\mathrm{div}u\|_{L^2}+\|\mathrm{curl}u\|_{L^2}+\sum_{j=1}^{2k-3}\big|\int_{I_j}(u\cdot n^\bot)dS\big|\bigg),
\end{equation}
where $n^\bot$ is the unit tangential vector on $\partial\Omega$ and $\{I_j\}_{j=1}^{2k-3}$ are mutually disjoint intervals on $\partial\Omega$ with arbitrary length.
\end{lemma}
\begin{proof}
Let us first illustrate that $u=0$ is the unique strong solution to the following homogeneous Cauchy-Riemann system:
\begin{equation}\label{g5}
\begin{cases}
\mathrm{div}u=0\ \mathrm{in}\ \Omega,\\
\mathrm{curl}u=0\ \mathrm{in}\ \Omega,\\
 u\cdot n\big|_{\partial\Omega}=0,\ u(\xi_j)=0~~\mathrm{for}\ j\in{1,\cdots,2k-3}.
\end{cases}
\end{equation}
In fact, if $u$ is a non-trivial strong solution to \eqref{g5}, then Lemma \ref{pr1} shows that there exists a non-constant harmonic function $\omega$ which is constant on each component of $\partial\Omega$, such that 
$$F(z)=u_1-iu_2=i\,\partial_z\omega.$$
Thus, in view of \eqref{g5}$_3$, we have $|u(\xi_j)|=|\nabla\omega(\xi_j)|=0$ for $j\in\{1,\cdots,2k-3\}$, which in particular implies that $\{\xi_j\}_{j=1}^{2k-3}$ are critical points of the harmonic function $\omega$ on $\partial\Omega$. However, Lemma \ref{LL24} enforces that
\begin{equation*}
k-2\geq \sum_{\xi_j\in\Omega} m(\xi_j)+\frac{1}{2}\sum_{\xi_j\in\partial\Omega}m(\xi_j)\geq k-3/2,
\end{equation*}
since the multiplicity of $\xi_j$ is at least one. The above inequality yields the contradiction, thus we must have $u\equiv 0$, and the uniqueness assertion of \eqref{g5} is completed.

Now we argue \eqref{g4} by contradictions. If \eqref{g4} fails, then there is a sequence of functions $\{u_m\}_{m\in\mathbb{N}}\subset W^{1,p}$ with $u_m\cdot n=0$ on $\partial\Omega$, such that 
\begin{equation}\notag
\|\nabla u_m\|_{L^p}^2> m\left(\|\mathrm{div}u_m\|_{L^p}^2+\|\mathrm{curl}u_m\|_{L^p}^2+\sum_{j=1}^{2k-3}\big|u_m(\xi_j)\big|^2\right).
\end{equation}
We normalize that $\|\nabla u_m\|_{L^p}=1$ and infer that
\begin{equation}\label{dc}\lim_{m\rightarrow\infty}
\bigg(\|\mathrm{div}u_m\|_{L^p}+\|\mathrm{curl}u_m\|_{L^p}+\sum_{j=1}^{2k-3}\big|u_m(\xi_j)\big|^2\bigg)=0,
\end{equation}
which along with \eqref{C302} leads to
\begin{equation}\label{dc1}\liminf_{m\rightarrow\infty}\|u_m\|_{L^p}\geq C^{-1}.\end{equation}
Note that Sobolev compact embedding theorem and \eqref{dc} ensure us to extract a subsequence   $\{u_{m_j}\}_{j\in\mathbb{N}}$ such that
$$u_{m_j}\rightarrow u\ \mathrm{in} \ W^{1,p} \ \mathrm{weakly},~~u_{m_j}\rightarrow u\  \mathrm{in} \ L^p,$$  
with $u$ satisfying \eqref{g5}. However, \eqref{dc1} implies that $\| u\|_{L^p}\geq C^{-1}$, which contradicts the uniqueness assertion of \eqref{g5}. Thus \eqref{g4} is valid.

To illustrate \eqref{g44}, we observe that if $u$ is continuous and  
$$\int_{I_j}(u\cdot n^\bot)dS=0,$$ 
then there is some $\xi\in I_j$ such that $u(\xi)=0.$ In fact, the intermediate value theorem ensures that we must have $u(\xi)\cdot n^\bot=0$ for some $\xi\in I_j$, however the boundary conditions also give $u(\xi)\cdot n=0$, which indicates $u(\xi)=0$. 
These arguments together with \eqref{g4} imply that $u\equiv0$ is the unique solution to  the Cauchy-Riemann system:
\begin{equation*} 
\begin{cases}
\mathrm{div}u=0\ \mathrm{in}\ \Omega,\\
\mathrm{curl}u=0\ \mathrm{in}\ \Omega,\\
 u\cdot n\big|_{\partial\Omega}=0,\ \int_{I_j}(u\cdot n^\bot)dS~~\mathrm{for}\ j\in{1,\cdots,2k-3}.
\end{cases}
\end{equation*}
The rest of proof is the same as above. We therefore finish the lemma.
\end{proof}

\begin{remark}
The $L^p$ estimates of the multi-solvable elliptic systems usually contains lower order norms over small open \textbf{neighbours} on the right hand side (see \cite[Chapter IV]{ADN} and \cite{Aramaki2014Lp}). However, the estimates \eqref{g4} only require finite distinct \textbf{points}, which is a rather interesting fact.
\end{remark}

Next we provide a series of div-curl estimates as  corollaries of Lemma \ref{le3.2}.  

\begin{lemma}\label{le3.4}
Let $\Omega$ be a $k$-connected domain satisfying \eqref{C101}, and $u\in W^{1,p}(p\geq 2)$ with $u\cdot n=0$ on $\partial\Omega$. The assertions below are true:

$1)$. Let $K$ be a non-negative smooth function on $\partial\Omega$ which is not  identically $0$, then there is a constant $C$ depending only on $\Omega$, $p$, and $K$, such that
\begin{equation}\label{gg1}
\|\nabla u\|_{L^p}^2\leq C \bigg(\|\mathrm{div}u\|_{L^p}^2+\|\mathrm{curl}u\|_{L^p}^2+\int_{\partial\Omega}K|u|^2\,dS\bigg).
\end{equation}

$2)$. Let $\rho$ be a non-negative function on $\Omega$ satisfying
\begin{equation}\label{g66}
\int_\Omega\rho\,dx\geq M_1,~~\int_{\Omega}\rho^r\, dx\leq M_2,
\end{equation}
for some $M_1,M_2>0$ and $r\in(1,\infty)$, then  there is a constant $C$ depending only on $\Omega$, $p$, $r$, $M_1$, and $M_2$, such that 
 \begin{equation}\label{g6} 
\|\nabla u\|_{L^p}\leq C\left(\|\mathrm{div}u\|_{L^p}+\|\mathrm{curl}u\|_{L^p}+\|\rho^{1/2} u\|_{L^2}\right).
\end{equation}
\end{lemma}
\begin{proof}

1). Let $u$ solve the following Cauchy-Riemann system,
\begin{equation}\label{CA302}
\begin{split}
&\begin{cases}
\mathrm{div}u=0\ \mathrm{in}\ \Omega,\\
\mathrm{curl}u=0\ \mathrm{in}\ \Omega,\\
u\cdot n=0\ \mathrm{on}\ \partial\Omega,~~\int_{\partial\Omega}K|u|^2\,dS=0.
\end{cases}\\
\end{split}
\end{equation}
Then $K\geq 0$ and \eqref{CA302} enforces $|u|\equiv 0$ on the open portion of $\partial\Omega$ where $K>0$, thus \eqref{g4} implies $u\equiv 0$ and gives the uniqueness assertion of \eqref{CA302}. The proof of \eqref{gg1} now follows from the same contradiction arguments in Lemma \ref{le3.2}.

2). Suppose that $\rho$ satisfies \eqref{g66} and $u$ solves the next Cauchy-Riemann system,
\begin{equation}\label{CA303} 
\begin{split}
&\begin{cases}
\mathrm{div}u=0\ \mathrm{in}\ \Omega,\\
\mathrm{curl}u=0\ \mathrm{in}\ \Omega,\\
u\cdot n=0\ \mathrm{on}\ \partial\Omega,~~\|\rho^{1/2} u\|_{L^2}=0.
\end{cases}
\end{split}
\end{equation}
Then \eqref{g66} and \eqref{CA303} ensure that $|u|=0$ on the subset $\{\rho>0\}$ which has  positive measure, thus \eqref{g4} yields $u\equiv 0$ and provides the uniqueness assertion of \eqref{CA303}.

Now suppose that \eqref{g6} is false, then there is a sequence $\{(\rho_m,u_m)\}_{m\in\mathbb{N}}$ with $u_m\in W^{1,p}$,  $u_m\cdot n|_{\partial\Omega}=0$, and $\rho_m$ satisfying \eqref{g66}, such that 
\begin{equation}\notag
\|\nabla u_m\|_{L^p}> m\left(\|\mathrm{div}u_m\|_{L^p}+\|\mathrm{curl}u_m\|_{L^p}+\|(\rho_m)^{1/2} u_m\|_{L^2}\right).
\end{equation}
We normalize $\|\nabla u_m\|_{L^p}=1$ and infer that
\begin{equation}\label{cm}
\lim_{m\rightarrow\infty}\left(\|\mathrm{div}u_m\|_{L^p}+\|\mathrm{curl}u_m\|_{L^p}+\|(\rho_m)^{1/2} u_m\|_{L^2}\right)=0.\end{equation}
Moreover, we deduce from \eqref{C302} that
\begin{equation}\label{CCC429}
\liminf_{n\rightarrow\infty}\|u_m\|_{L^p}\geq C^{-1}.
\end{equation}
Sobolev compact embedding theorem along with \eqref{g66} and \eqref{cm} provides a subsequence $\{(\rho_{m_j},u_{m_j})\}_{j\in\mathbb{N}}$ such that,
$$\rho_{m_j}\rightarrow\rho\ \mathrm{in}\ L^{r}\ \mathrm{weakly},~~u_{m_j}\rightarrow u\ \mathrm{in}\ W^{1,p}\ \mathrm{weakly},~~ u_{m_j}\rightarrow u\ \mathrm{in}\ L^{2r'},$$  
with $r^{-1}+{r'}^{-1}=1$, and $(\rho,u)$ satisfying \eqref{g66} and \eqref{CA303}. But \eqref{CCC429} implies that $\|u\|_{L^p}\geq C^{-1},$ which contradicts the uniqueness assertion given above. Thus \eqref{g6} holds and the proof of the lemma is finished.
\end{proof}

To end up this subsection, we state a modified version of  2) in Lemma \ref{le3.4}.
\begin{lemma}\label{a004}
Under the conditions of Lemma \ref{le3.4}, there exist positive constants $\tilde \nu$ and $C_1$ depending only  on  $\Omega$, $r$, $M_1$, and $M_2$,  such that for any $\nu\in (0,\tilde \nu)$,
\begin{equation}\notag
\int_\Omega|u|^{\nu}|\nabla u|^2dx\leq C_1\left(\int_\Omega|u|^{\nu}\big((\mathrm{div}u)^2+(\mathrm{curl}u)^2)\,dx+\int_\Omega\rho|u|^{2+\nu}dx\right).
\end{equation}
\end{lemma}
\begin{proof} 
Direct computations show that
\begin{equation}\label{qa1}
\begin{split}
\left(\mathrm{div}(|u|^{\frac{\nu}{2}} u)\right)^2&
\leq 2|u|^\nu \big(\mathrm{div} u\big)^2+ \nu^2|u|^\nu|\nabla u|^2,\\
\left(\mathrm{curl}(|u|^{\frac{\nu}{2}} u)\right)^2
&\leq 2|u|^\nu \big(\mathrm{curl} u\big)^2+\nu^2|u|^\nu|\nabla u|^2,\\
|\nabla(|u|^{\frac{\nu}{2}} u)|^2&\geq\frac{1}{2}\,|u|^\nu|\nabla u|^2-\nu^2|u|^\nu|\nabla u|^2,
\end{split}
\end{equation}
 
Moreover, note that $ |u|^{\frac{\nu}{2}}u\cdot n\big|_{\partial\Omega}=0$, thus we apply \eqref{g6} to deduce  
\begin{equation*} 
\begin{split}
&\int_\Omega\left|\nabla(|u|^{\frac{\nu}{2}}u)
\right|^2dx\\
&\leq
\tilde{C}\int_\Omega\left(\left|\mathrm{div}(|u|^{\frac{\nu}{2}}u)\right|^2 +\left|\mathrm{curl}(|u|^{\frac{\nu}{2}}u)\right|^2\right)dx+C\int_\Omega\rho|u|^{2+\nu}dx,
\end{split}
\end{equation*} 
which together with \eqref{qa1} proves Lemma \ref{a004} by setting $\nu\leq (40\tilde{C}+10)^{-1/2}$.
\end{proof}

\subsection{The uniqueness of solution to stationary systems}\label{SS32}
\quad Let $\Omega$ be a $k$ connected domain satisfying \eqref{C101}. We  make full use of the assertions in Subsection \ref{SS31} to  investigate the stationary system on $\Omega$:
\begin{equation}\label{C303}
\begin{cases}
\mathrm{div}(\rho u)=0,\\
\mathrm{div}(\rho u\otimes u)+\nabla P=\mu\Delta u+\nabla\big((\mu+\lambda)\mathrm{div}u\big),\\
\end{cases}
\end{equation}
which is subject to the conditions
\begin{equation}\label{C304}
\begin{split}
\int_\Omega\rho\,dx=\int_\Omega\rho_0\,dx,~~
u\cdot n|_{\partial\Omega}=0,~~\mathrm{curl}u|_{\partial\Omega}=-Ku\cdot n^\bot.
\end{split}
\end{equation}
For convenience, we take $P=\rho^2$ and prove Theorem \ref{L11} in different circumstances.
 
\textit{Proof of Theorem \ref{L11}.} We always assume that $(\rho,u)$ is a smooth solution to the system \eqref{C303}--\eqref{C304} with $\rho$ containing no vacuum, then there are three assertions.

a). \textit{If $K$ is not identically $0$ on $\partial\Omega$, then the unique smooth solution to the system is the trivial one.}

Multiplying \eqref{C303}$_2$ by $u$ and integrating over $\Omega$ yield the standard energy estimates,
\begin{equation}\label{C305}
\int_\Omega\left((2\mu+\lambda)(\mathrm{div}u)^2+\mu(\mathrm{curl}u)^2\right)\,dx+\int_{\partial\Omega}K|u|^2\,dS=0.
\end{equation}
According to \eqref{gg1} and \eqref{C305}, we declare that
\begin{equation*}
\|\nabla u\|_{L^2}\leq C\int_\Omega\left((2\mu+\lambda)(\mathrm{div}u)^2+\mu(\mathrm{curl}u)^2\right)\,dx+C\int_{\partial\Omega}K|u|^2\,dS=0,
\end{equation*}
which together with \eqref{C306}$_3$ provides that $u\equiv 0$ in $\Omega$. Thus, \eqref{C303} reduces to
\begin{equation}\label{C319}
\nabla\rho^2=0,~~
\int_\Omega\rho\,dx=\int_\Omega\rho_0\,dx.
\end{equation}
We therefore deduce that the trivial solution
$$\rho\equiv\hat{\rho},~~u\equiv 0,$$
is the unique non-vacuum smooth solution to case a).

b). \textit{If $K\equiv 0$ and $\partial\Omega$ consists of two concentric circles, then the non-vacuum smooth solution to the system is not unique.}

For convenience, we assume that $\Omega=\{1/2<|z|<1\}$. In such case, the standard energy estimates \eqref{C305} reduce to
\begin{equation}\label{C307}
\int_\Omega\left((2\mu+\lambda)(\mathrm{div}u)^2+\mu(\mathrm{curl}u)^2\right)\,dx=0.
\end{equation}
Thus $u$ satisfies the Cauchy-Riemann equations on $\Omega$,
\begin{equation}\label{C306}
\begin{cases}
\mathrm{div}u=\mathrm{curl}u=0\ \mathrm{in}\ \Omega,\\
u\cdot n=0\ \mathrm{on}\ \partial\Omega.
\end{cases}
\end{equation}
According to Lemma \ref{pr1}, there is a harmonic function $\omega$ such that
\begin{equation}\label{CA301}
F(z)=u_1-iu_2=i\,\partial_z\omega,~~\omega\big|_{\{|z|=1/2\}}=C_0,~~\omega\big|_{\{|z|=1\}}=0,
\end{equation}
for some constant $C_0$.
Let us solve \eqref{C306} and \eqref{CA301} more accurately.
Note that for $z\in\partial\Omega$, $z=x+iy$ always lies in the normal direction of $\partial\Omega$, then the boundary conditions \eqref{C306} provide that
$$n\cdot u=x\cdot u_1+y\cdot u_2=\mathrm{Re}(z\cdot F(z))=0\ \mathrm{on}\ \partial\Omega.$$
According to the maximum principle of harmonic functions, we declare that
$$\mathrm{Re}(z\cdot F(z))\equiv 0\ \mathrm{in}\ \Omega.$$
Then Lemma \ref{L24} guarantees that for some constant $C_1$, we also have
$$\mathrm{Im}(z\cdot F(z))\equiv C_1\ \mathrm{in}\ \Omega,$$
which implies that $F(z)=u_1-iu_2=iC_1/z$. Turning back to real variables gives
\begin{equation}\label{C308}
u_1=C_1\cdot\frac{y}{R^2},~~u_2=-C_1\cdot\frac{x}{R^2},~~\omega=C_1\log R,~~R^2=x^2+y^2.
\end{equation}
Hence, the velocity field $u$ revolves around the central disc $\{z\leq 1/2\}$ and lies in the angular directions, while the module $|u|=C_1/R$ is axisymmetric and monotonically decreasing. 

Let us substitute \eqref{C308} into \eqref{C303} and determine $\rho$. Note that $\mathrm{div}u=0$ together with \eqref{C303}$_1$ leads to
\begin{equation}\label{C309}
u\cdot\nabla\rho=0\ \mathrm{in}\ \Omega,
\end{equation}
which means $\rho$ is constant along the angular directions, thus $\rho(z)=\rho(|z|)$ is axisymmetric as well. Moreover, $\mathrm{div}u=\mathrm{curl}u=0$ reduces \eqref{C303}$_2$ to
\begin{equation*}
\rho u\cdot\nabla u+\nabla P=0.
\end{equation*}
Note that $\mathrm{curl}u=0$ yields $\partial_iu_j=\partial_ju_i$ for $i,j=1,2$, thus we further deduce that
\begin{equation}\label{C312}
\begin{split}
\rho u\cdot\nabla u+\nabla P&=\rho u_i\cdot\partial_iu_j+\partial_j\rho^2\\
&=\rho u_i\cdot\partial_ju_i+\partial_j\rho^2\\
&=\rho\,\nabla(|u|^2/2+2\rho)=0.
\end{split}
\end{equation}
Since $\rho$ contains no vacuum, the above equation leads to
\begin{equation}\label{C310}
\rho=C_2-\frac{|u|^2}{4}=C_2-\frac{C_1^2}{4R^2}~~ \mathrm{in}\ \Omega,
\end{equation}
for some constant $C_2$. The representation in \eqref{C310} is axisymmetric thus satisfies \eqref{C309} as well, thus the non-trivial solutions to the system \eqref{C303}--\eqref{C304} are given by
\begin{equation*}
\begin{split}
\rho(z)=C_2-\frac{C_1^2}{4|z|^2},~~u_1(z)-iu_2(z)=\frac{iC_1}{z},~~z=x+iy.
\end{split}
\end{equation*}
In addition, we can choose $C_1$ and $C_2$ such that
\begin{equation*}
\begin{split}
\int_\Omega\rho\,dx=\frac{3\pi C_2}{4}-\frac{(\log2)\pi C_1^2}{2}=\int_\Omega\rho_0\,dx,~~C_2>C_1^2,
\end{split}
\end{equation*}
to ensure that \eqref{C304} holds and $\rho$ is strictly positive. The assertion for case b) is  finished.

c). \textit{If $K\equiv 0$ and $\Omega$ is any $k$-connected domain different from $\mathrm{b)}$, then the unique non-vacuum smooth solution to the system is the trivial one.}

Similar with case b), the energy estimate \eqref{C307} ensures that $u$ solves the Cauchy-Riemann system \eqref{C306} on $\Omega$, which changes \eqref{C303} into the system
\begin{equation}\label{C313}
\begin{cases}
u\cdot\nabla\rho=0,\\
\rho u\cdot\nabla u+\nabla P=0.
\end{cases}
\end{equation}
Note that $\rho$ is out of vacuum, thus the process \eqref{C312} applied to \eqref{C313}$_2$ yields that
$$ \nabla(|u|^2+4\rho)=0,$$
which means $\rho=(C-|u|^2/4)$. This explicit formula combined with \eqref{C313}$_1$ gives
$$u\cdot\nabla|u|^2=0.$$
Consequently, the velocity field $u$ necessarily satisfies the equations
\begin{equation}\label{C314}
\begin{cases}
\mathrm{div}u=\mathrm{curl}u=0\ \mathrm{in}\ \Omega,\\
u\cdot\nabla|u|^2=0\ \mathrm{in}\ \Omega,\\
u\cdot n=0\ \mathrm{on}\ \partial\Omega.
\end{cases}
\end{equation}
We declare that $u\equiv 0$ is the only smooth solution to \eqref{C314}, which along with \eqref{C319} also implies that \eqref{C303}--\eqref{C304} is uniquely solved by $(\hat\rho,0)$.

Let us argue by contradictions. We first consider the case that $\Omega$ is a 2-connected domain different from the concentric annulus, and assume that $u$ is a non-trivial solution to the system \eqref{C314}.  

According to Lemma \ref{L27}, $\Omega$ is conformally equivalent with some annulus domain $A=\{r<|z|<1\}$, and we can find a conformal mapping $\varphi$, which is regular up to the boundary and satisfies
$$\varphi:\overline{A}\rightarrow\overline{\Omega},~~w\mapsto z.$$
Considering the analytic function $G(w)=F\circ\varphi(w)$ defined on $A$,  we denote $v_1=\mathrm{Re}\,G$ and $v_2=-\mathrm{Im}\,G$, then according to Lemma \ref{L28},  the function $v=(v_1,v_2)$ solves the Cauchy-Riemann system
\begin{equation*} 
\begin{cases}
\mathrm{div}v=0\ \mathrm{in}\ A,\\
\mathrm{curl}v=0\ \mathrm{in}\ A,\\
v\cdot n=0\ \mathrm{on}\ \partial A.
\end{cases}
\end{equation*}
Then by virtue of \eqref{C308}, we must have $G(w)=i\,C_0/w $ for some $C_0\in\mathbb{R}\setminus\{0\}$. Setting $w=\psi(z)$ with $\psi=\varphi^{-1}$ implies that
\begin{equation}\label{C320}
F(z)=\frac{i\,C_0}{\psi(z)}.
\end{equation}
According to \eqref{cc}, the equation \eqref{C314}$_2$ is transformed into
\begin{equation*}
\begin{split}
0=u\cdot\nabla|u|^2&=(u_1\partial_x+u_2\partial_y)\,|u|^2\\
&=(u_1-iu_2)\,\partial_{\bar{z}}|u|^2+(u_1+iu_2)\,\partial_z|u|^2\\
&=2\,\mathrm{Re}\,(F\cdot\partial_{\bar{z}}|F|^2)\\
&=2\,\mathrm{Re}\,(F^2\cdot\overline{F'}).
\end{split}
\end{equation*}
Substituting \eqref{C320} into the above equation gives
$$ \mathrm{Re}(F^2\cdot\overline{F'})=\mathrm{Re}\left(\frac{-i\,C_0^3}{|\psi|^4}\cdot\overline{\psi'}\right)=0,$$
which means $\mathrm{Im}\psi'=0$, thus Lemma \ref{L24} and Lemma \ref{L28} imply that $\psi'(z)=\alpha$ for some $\alpha\in\mathbb{R}\setminus\{0\}$. Consequently $\varphi$ and $\psi$ are given by linear functions,
$$w=\psi(z)=\alpha\cdot z+\beta,~~z=\varphi(w)=\alpha^{-1}(w-\beta),$$
which enforces $\Omega$ to be an concentric annulus and contradicts our assumption on the shape of $\Omega$. Thus we must have $u\equiv 0$ and the uniqueness assertion for general 2-connected domains follows at once.

Next, we consider the case that $\Omega$ is a $k$-connected domain with $k\geq 3$. Suppose that $u$ solves \eqref{C314} and is not identically $0$, then according to Lemma \ref{pr1}, there is a harmonic function $\omega$ such that
\begin{equation}\label{C315}
F(z)=u_1-iu_2= 2i\,\partial_z\omega= \partial_y\omega +i\,\partial_x\omega,~~\omega\big|_{\Gamma_j}=\beta_j,~~j\in\{0,1,\cdots,k-1\},
\end{equation}
where $\{\beta_j\}_{j=0}^{k-1}$ are $k$ real constants not coincide.  
Thus we also derive that $u=\nabla^\bot\omega$.

Let us introduce the level sets of $\omega$ by $\mathcal{A}_\alpha\triangleq\{z\in\overline{\Omega}\,|\,\omega(z)=\alpha\}$ and the critical set of $\omega$ by $\mathcal{C}\triangleq\{z\in\overline{\Omega}\,|\,\nabla\omega(z)=0\}$. According to Lemma \ref{LL24}, $\mathcal{C}$ must be non-empty (note that $\mathcal{C}$ is empty on 2-connected domains, see \eqref{C308} for example), and for any $z_0\in\mathcal{C}$, the level set $\mathcal{A}_{\omega(z_0)}$ near $z_0$ consists of finite simple curves intersecting at $z_0$.

Now, we declare that for each $\zeta\in\overline{\Omega}\setminus\mathcal{C}$, the level sets $\mathcal{A}_{\omega(\zeta)}$ near $\zeta$ is exactly a portion of smooth curve. 
In fact, if $\zeta\in{\Omega}\setminus\mathcal{C}$, the declaration follows form the fact $\nabla^\bot\omega(\zeta)\neq 0$ and the implicit function theorem; if $\zeta\in\Gamma_j\setminus\mathcal{C}$, note that $\omega\equiv\beta_j$ on $\Gamma_j$ and the Schwarz reflection principle \cite[Chapter 2]{GT} ensures that $\omega$ can be extended to a harmonic function on the open neighbour of $\Gamma_j$, thus the implicit function theorem guarantees that, in the small neighbour of $\zeta$, the level set $\mathcal{A}_{\beta_j}$   completely lies in $\Gamma_j$, which finishes the declaration.

In particular, for $\zeta\in\overline{\Omega}\setminus\mathcal{C}$, let $\gamma(t):[-1,1]\rightarrow\mathcal{A}_{\omega(\zeta)}$ be a parametrization of the level curve passing through $\zeta$, which satisfies
$$\omega(\gamma(t))=\omega(\zeta),~~\gamma(0)=\zeta.$$
Taking the derivative with respect to $t$ and checking the value at $t=0$ yield 
\begin{equation}\label{C311}
\gamma'(0)\cdot\nabla\omega(\zeta)=0,
\end{equation}
moreover, the definition of $\nabla^\bot$ guarantees that $\nabla^\bot\omega\cdot\nabla\omega=0$, thus in view of \eqref{C311}, we must have $\gamma'(0)$ is parallel to $\nabla^\bot\omega(\zeta)$, which implies that the level sets are in fact the integral curves of $\nabla^\bot\omega$ on $\overline{\Omega}\setminus\mathcal{C}$.
 
Consequently, let us substitute \eqref{C315} into \eqref{C314}$_2$ and declare that
\begin{equation}\label{C316}
\nabla^\bot\omega\cdot\nabla\left(|\nabla\omega|^2\right)=0,
\end{equation}
which in particular means that the module $|u|=|\nabla\omega|$ is constant along path-connected components of each level set $\mathcal{A}_\alpha$ ($\alpha\in\mathbb{R}$) on $\overline{\Omega}\setminus\mathcal{C}$. 

Then, we illustrate that the critical point of $\omega$ can not lie in the interior of $\Omega$.
If it is not true, suppose that $z_0\in\mathcal{C}\cap\Omega$, then by virtue of the structure of the level set $\mathcal{A}_{\omega(z_0)}$ near $z_0$, we can choose a sequence $\{z_n\}_{n=1}^\infty$ lying in the same path-connected component of $\mathcal{A}_{\omega(z_0)}$, such that 
\begin{equation}\label{C317}
\lim_{n\rightarrow\infty}z_n=z_0,~~\lim_{n\rightarrow\infty}|\nabla\omega(z_n)|=|\nabla\omega(z_0)|=0.
\end{equation}
However, the critical points of $\omega$ must be isolated (Lemma \ref{LL24}) , thus we may assume $\{z_n\}_{n=1}^\infty\subset(\overline{\Omega}\setminus\mathcal{C})$, and \eqref{C316} guarantees that $|\nabla\omega(z_n)|$ must be a constant different form $0$, which contradicts \eqref{C317}.
Hence, no critical point lies in the interior.

Similarly we prove that $\partial\Omega$ contains no critical point of $\omega$ as well. Note that $\omega\equiv\beta_j$ on $\Gamma_j$, so we utilize the Schwarz reflection principle (\cite[Chapter 2]{GT}) again to declare that $|\nabla\omega|$ is continuous on $\overline\Omega$. Now suppose that $\nabla\omega(z_0)=0$ for some $z_0\in\Gamma_j\subset\partial\Omega$,   we can select a sequence $\{z_n\}_{n=1}^\infty\subset\Gamma_j$ lying in the same path-connected component of $\mathcal{A}_{\beta_j}$ such that
\begin{equation}\label{C318}
\lim_{n\rightarrow\infty}z_n=z_0,~~\lim_{n\rightarrow\infty}|\nabla\omega(z_n)|=|\nabla\omega(z_0)|=0.
\end{equation}
However, the critical point of $\omega$ is isolated on $\partial\Omega$, thus we may assume $\{z_n\}_{n=1}^\infty\subset(\overline{\Omega}\setminus\mathcal{C})$, and \eqref{C316} ensures that $|\nabla\omega(z_n)|$ is a constant different from $0$, which contradicts \eqref{C318}. Consequently, no critical point is contained in $\partial\Omega$ as well.

The above assertions implies that $\omega$ contains no critical point, which contradicts the fact that $\mathcal{C}$ is non-empty. Thus we must have $\omega\equiv C$ and $u\equiv 0$ on $\Omega$, which finishes the uniqueness assertion of case c).
The proof is therefore completed.
\thatsall

\section{Commutator estimates on multi-connected domains}\label{S4}
\quad In this section, we make detailed discussions on the commutator theory on multi-connected domains, which is the foundation of a priori estimates in Section \ref{S5}. Moreover, our method which is based on the model domains  seems applicable for other boundary value problems, so we carefully conclude it below.

Let us first provide the motivation. The commutator is a powerful tool when we deal with the subtle estimates on the global domains, however it is not directly available on the bounded domains. For example, following the classical method in Lions' monograph \cite[Chapter 5]{1998Mathematical}, we introduce the effective viscous flux, $$F\triangleq(2\mu+\lambda)\mathrm{div}u-(P-\hat P).$$

In the case of the Cauchy problem $\Omega=\mathbb{R}^2$, $F$ is given explicitly by
\begin{equation}\label{CC501}
\begin{split}
F&=\frac{\mathrm{d}}{\mathrm{d}t}\Delta^{-1}\mathrm{div}(\rho u)-
u_i\cdot\partial_i\Delta^{-1}\partial_j(\rho u_j)
+\Delta^{-1}\partial_{i}\partial_{j}(\rho u_i u_j)\\
&=\frac{\mathrm{d}}{\mathrm{d}t}\Delta^{-1}\mathrm{div}(\rho u)-
u_i\cdot\Delta^{-1}\partial_i\partial_j(\rho u_j)
+\Delta^{-1}\partial_{i}\partial_{j}(\rho u_i u_j)\\
&=\frac{\mathrm{d}}{\mathrm{d}t}\Delta^{-1}\mathrm{div}(\rho u)-[u_i,R_{ij}](\rho u_j),\\
\end{split}
\end{equation}
where we suppose $\Gamma(x,y)$ is the fundamental solution of $\Delta$ in $\mathbb{R}^2$ and define 
$$\Delta^{-1}f \triangleq\int_\Omega\Gamma(x,y)f(y)\,dy,~~R_{ij}\triangleq\Delta^{-1}\partial_i\partial_j.$$
The last term in \eqref{CC501} is a typical commutator studied by Coifman-Meyer \cite{CM}. The key step of \eqref{CC501} is the second line, where we have applied
\begin{equation}\label{CC502}
\partial_i\circ\Delta^{-1}=\Delta^{-1}\circ\partial_i~~\mathrm{in}\ \mathbb{R}^2. 
\end{equation} 
From one point of view, \eqref{CC502} is due to the symmetric property
\begin{equation}\label{CC503}
\partial_{x_i}\Gamma(x,y)+\partial_{y_i}\Gamma(x,y)=0.
\end{equation}

However, the above process fails in the bounded domain $\Omega$. In fact, if we still interpret 
$$\Delta^{-1}f\triangleq\int_\Omega N(x,y)f(y)\,dy,$$
where $N(x,y)$ is the Green's function on $\Omega$. Then the key symmetric property \eqref{CC502} is no longer valid since  
\begin{equation}\label{CC504}
\partial_{x_i}N(x,y)+\partial_{y_i}N(x,y)\neq 0, 
\end{equation}
which  violates \eqref{CC501} as well. Accordingly, the term given by \eqref{CC504}, which commutes $\partial_x$ and $\partial_y$ of Green's functions, lies in the central position of the commutator theory on bounded domains. 

Let us first study the structure of Green's functions near $\partial\Omega$ in Subsection \ref{SS41}, and then deduce the cancellation properties of the crucial term \eqref{CC504} in Subsection \ref{SS42}. At last in Subsection \ref{SS43}, we establish a type of commutators arguments on $\Omega$.
 
\subsection{The structure of Green's functions}\label{SS41}
\quad Let $\Omega$ be a $k$-connected domain satisfying:
\begin{equation}\label{CA414}
\partial\Omega=\bigcup_{i=0}^{k-1}\Gamma_i,~~\Gamma_i\cap\Gamma_j=\varnothing\ \  \mathrm{if}\ i\neq j,
\end{equation}
where $\Gamma_j$ are smooth simple closed curves. Setting $w=x_1+ix_2$ and $z=y_1+iy_2$ with $x=(x_1,x_2),\,y=(y_1,y_2)\in \Omega$, we discuss the principal part of the Green's function near each $\Gamma_j$.

Let us first study the model case that $\Omega$ is a $k$-connected bounded circular domain, whose each boundary component $\Gamma_j$ is a completed circle given by
\begin{equation}\label{CA402}
\Gamma_j=\partial B_j,~~B_j\triangleq\left\{z\,\big|\, |z-b_j|< r_j\right\},~~j=\{0,1,\cdots,k-1\}.
\end{equation}
We assume $B_0$ is the unit disc and $\{b_j\}_{j=0}^{k-1}\subset B_0$, then $\Omega=B_0\setminus\bigcup_{j=1}^{k-1}\overline{B_j}$. In addition, we also introduce \begin{equation}\label{CA401}
\mathbf{d}\triangleq\min \left\{\mathrm{dist}(\Gamma_i,\Gamma_j)\,\big|\,i\neq j\ \mathrm{with}\ i,j\in\{0,1,\cdots,k-1\}\right\},
\end{equation}
to describe the smallest distance between each holes.

The Green's function $N(z,w)$ for Neumann problems on $\Omega$ is given by
\begin{equation}\label{CC407}
N(z,w)=\Gamma(z,w)+
H(z,w),~~\Gamma(z,w)=\frac{1}{2\pi}\log|z-w|,
\end{equation}
where $H(z,w)$ is a harmonic function determined by the Neumann problem
\begin{equation}\label{CC408} 
\begin{cases}
\Delta H(\cdot,w)=0\ \mathrm{in}\ \Omega,\\
\partial_nH(\cdot,w)=-\partial_n\Gamma(\cdot,w)+l^{-1}\ \mathrm{on}\ \partial\Omega,
\end{cases}
\end{equation}
where $\partial_n$ is the outer normal derivative and $l=2\pi(\sum_i r_i+1)$ is the total length of $\partial\Omega$.  
Now the first lemma concerns the structure of $N(z,w)$ when $w$ is close to $\partial\Omega$.

\begin{lemma}\label{LL41}
Let $\Omega$ be the circular domain given by \eqref{CA414}--\eqref{CA402} and $N(z,w)$ be the Green's function given by \eqref{CC407}, if for some $j\in\{0,1,\cdots,k-1\}$, we have $\mathrm{dist}(w,\Gamma_j)\leq \mathbf{d}/4$, then $N(z,w)$ can be decomposed as
$$N(z,w)=N_j(z,w)+R_j(z,w).$$

The principal part $N_j(z,w)$ is the Green's function for the exterior domain $\mathbb{R}^2\setminus\overline{B_j}$, for $j\in\{1,\cdots,k-1\}$, and $N_0(z,w)$ is the Green's function for the unit disc $B_0$. The remaining terms $R_j(z,w)$ are regular functions on $\Omega$, in the sense that for any $s\in\mathbb{N}^+$, there is a constant $C$ depending on $s$, $\mathbf{d}$, and $\Omega$, such that
\begin{equation}\label{CA404}
\|\nabla^sR_j(\cdot,w)\|_{L^\infty}\leq C~~\mbox{for any }j\in\{0,1,\cdots,k-1\}.
\end{equation}
\end{lemma}

\begin{proof}
We first consider the case $\mathrm{dist}(w,\Gamma_0)\leq\mathbf{d}/4$, then the definition \eqref{CA401} guarantees that we must have $\mathrm{dist}(w,\Gamma_j)\geq\mathbf{d}/2$ for $j\neq 0$. 

Recalling that the Green's function $N_0(z,w)$ on the unit disc $B_0$ (see \cite[Chapter 2]{GT}) is explicitly given by
\begin{equation}\label{CA403}
N_0(z,w)=\Gamma(z,w)+H_0(z,w),~~H_0(z,w)=\frac{1}{2\pi}\mathrm{Re}\,\log(1-z\bar w) ~~ \ z,w\in B_0,
\end{equation}
where $H_0(z,w)$ solves the Neumann problem on $\overline{B_0}$:
\begin{equation}\label{g7}
\begin{cases}
\Delta H_0(\cdot,w)=0~~\mathrm{in}\ B_0,\\
\partial_nH_0(\cdot,w)=-\partial_n\Gamma(\cdot,w)+(2\pi)^{-1}~~\mathrm{on}\ \Gamma_0.
\end{cases}
\end{equation}
If we denote 
$R_0(z,w)\triangleq H(z,w)-H_0(z,w)$, then \eqref{CC408} together with \eqref{g7} guarantees that the remaining term $R_0(z,w)$ solves the Neumann problem on $\overline{\Omega}$: 
\begin{equation}\label{g8}
\begin{cases}
\Delta R_0(\cdot,w)=0~~\mathrm{in}\ \Omega,\\
\partial_nR_0(\cdot,w)=l^{-1}-(2\pi)^{-1}\ \mathrm{on}\ \Gamma_0, \\ 
\partial_nR_0(\cdot,w)=-\partial_n\Gamma(\cdot,w)-\partial_nH_0(\cdot,w)+l^{-1}~~\mathrm{on}\ \partial\Omega\setminus\Gamma_0.
\end{cases}
\end{equation}
Applying \eqref{CA402} and \eqref{CA403}, we calculate that for $z\in\Gamma_j=\{|z-b_j|=r_j\}$ with $j\in\{1,\cdots,k-1\}$,
$$\partial_nR_0(z,w)=\frac{1}{4\pi r_j}\left(\frac{z-b_j}{z-w}+\frac{\bar z-\bar{b}_j}{\bar z-\bar w}\right)-\frac{1}{2\pi r_j}\mathrm{Re}\left(\frac{(z-b_j)\bar w}{1-z\bar w}\right)+l^{-1},$$
which is smooth and bounded uniformly due to   
$$|z-w|\geq \mathbf{d}/2,~~ |z\cdot\bar w|\leq(1-\mathbf{d}),~~ \mbox{since $\mathrm{dist}(w,\Gamma_0)\leq\mathbf{d}/4$ and $z\in\partial\Omega\setminus\Gamma_0$}.$$
Consequently, the Schauder estimates (\cite[Chapter 3]{ADN}) to  \eqref{g8} ensure that for any $s\in\mathbb{N}^+$, there is a constant $C$ depending on $s$, $\mathbf{d}$, and $\Omega$, such that
\begin{equation}\label{j1}
\|\nabla^s R_0(\cdot,w)\|_{L^{\infty}(\Omega)}\leq C.
\end{equation}
Combining \eqref{CC407}, \eqref{CA403}, and \eqref{j1}, it holds that
$$N(z,w)=N_0(z,w)+R_0(z,w),$$
with $R_0(z,w)$ satisfying \eqref{CA404}, provided $\mathrm{dist}(w,\Gamma_0)\leq\mathbf{d}/4$.

Next, the arguments for cases $\mathrm{dist}(w,\Gamma_j)\leq\mathbf{d}/4$ with $j\in\{1,\cdots,k-1\}$ are similar. Let us consider the Green's function on the exterior domain $\mathbb{R}^2\setminus\overline{B_j}$,
\begin{equation}\label{CA407}
N_j(z,w)=\Gamma(z,w)+H_j(z,w)~~z,w\in \mathbb{R}^2\setminus\overline{B_j},
\end{equation}
where $H_j(z,w)$ solves the Neumann problem on $\mathbb{R}^2\setminus {B_j}$:
\begin{equation}\label{g9}
\begin{cases}
\Delta H_j(\cdot,w)=0~~\mathrm{in}\ \mathbb{R}^2\setminus\overline{B_j},\\
\partial_nH_j(\cdot,w)=-\partial_n\Gamma(\cdot,w) 
~~\mathrm{on}\ \Gamma_j.
\end{cases}
\end{equation}
We mention that the solution of \eqref{g9} is not  unique due to the unboundedness of the exterior domain. However, the Schwarz reflection principle \cite[Chapter 2]{GT} provides an explicit solution of \eqref{g9} written as
\begin{equation}\label{hk}
H_j(z,w)=\frac{1}{2\pi}\mathrm{Re}\,\log\left(1-\frac{r_j}{z-b_j}\cdot\frac{r_j}{\bar w-\bar{b}_j}\right).
\end{equation}
Thus, we denote $R_j(z,w)=H(z,w)-H_j(z,w)$, then \eqref{CC408} together with \eqref{g9} ensures that $R_j(z,w)$ solves the Neumann problem on $\overline{\Omega}$: 
\begin{equation}\label{g10}
\begin{cases}
\Delta R_j(\cdot,w)=0~~\mathrm{in}\ \Omega,\\
\partial_nR_j(\cdot,w)=l^{-1}~~\mathrm{on}\ \Gamma_j,\\ 
\partial_nR_j(\cdot,w)=-\partial_n\Gamma(\cdot,w)-\partial_nH_j(\cdot,w)+l^{-1}~~\mathrm{on}\ \partial\Omega\setminus\Gamma_j.
\end{cases}
\end{equation}
Direct computations yields that when $z\in\Gamma_0$,
\begin{equation}\label{CA405}
\begin{split}\partial_nR_j&=-\frac{1}{4\pi}\left(\frac{z}{z-w}+\frac{\bar z}{\bar z-\bar w}\right)\\
&\quad+\frac{1}{2\pi}\mathrm{Re}\left(\frac{r_j^2\,z\,(z-b_j)^{-1}}{ (z-b_j)(\bar{w}-\bar{b}_j)-r_j^2}\right)+l^{-1},
\end{split}
\end{equation}
while for $z\in\Gamma_p\ (p\neq 0,j)$, we have
\begin{equation}\label{CA406}
\begin{split}\partial_nR_j&=\frac{1}{4\pi r_p}\left(\frac{z-b_p}{z-w}+\frac{\bar z-\bar{b}_p}{\bar z-\bar w}\right)\\
&\quad-\frac{1}{2\pi r_p}\mathrm{Re}\left(\frac{r_j^2\,(z-b_p)\,(z-b_j)^{-1}}{(z-b_j)\,(\bar{w}-\bar{b}_j )-r_j^2}\right)+l^{-1}.
\end{split}
\end{equation}
In particular, when $\mathrm{dist}(w,\Omega_j)\leq\mathbf{d}/4$ and $z\in \partial\Omega\setminus\Gamma_j$, it holds that
$$|z-w|\geq\mathbf{d}/2,~~ |z-b_j|\geq r_j+\mathbf{d},~~ |z-b_j|\cdot|\bar{w}-\bar{b}_j|\geq(r_j+\mathbf{d})\cdot r_j> r_j^2.$$
Thus \eqref{CA405} and \eqref{CA406} ensures that $\partial_nR_j$ is smooth and bounded uniformly on $\partial\Omega\setminus\Gamma_j$. Thereby, applying the Schauder estimates (\cite[Chapter 3]{ADN}) to \eqref{g10} yields that for any $s\in\mathbb{N}^+$, there is a constant $C$ depending on $s$, $\mathbf{d}$, and $\Omega$, such that
\begin{equation*}\ 
\|\nabla^s R_j(\cdot,w)\|_{L^\infty(\Omega)}\leq C.
\end{equation*}
Consequently, when $\mathrm{dist}(w,\Omega_j)\leq\mathbf{d}/4$ with $j\in\{1,\cdots,k-1\}$, we  apply \eqref{CC407} together with \eqref{CA407} to declare that
$$N(z,w)=N_j(z,w)+H_j(z,w),$$
with $R_j(z,w)$ satisfying \eqref{CA404}.
The proof of the lemma is therefore completed.
\end{proof}
\begin{remark}\label{RR42}
Lemma \ref{LL41} also provides the decomposition of the harmonic function $H(z,w)$ given by \eqref{CC408}. If for some $j\in\{0,1,\cdots,k-1\}$, we have $\mathrm{dist}(w,\Gamma_j)\leq \mathbf{d}/4$, then $H(z,w)$ can be decomposed as
$$H(z,w)=H_j(z,w)+R_j(z,w),$$ 
where $H_j(z,w)$ is the principal part explicitly provided by \eqref{CA403} and \eqref{hk}, and $R_j(z,w)$ is the regular remaining term satisfying \eqref{CA404}.
\end{remark}

In the same spirit as Lemma \ref{LL41}, the structure of the Green's function is simpler when $w$ is away form $\partial\Omega$.
\begin{lemma}\label{LL42}
Let $\Omega$ be the circular domain given by \eqref{CA414}--\eqref{CA402} and $N(z,w)$ be the Green's function given by \eqref{CC407}, if $\mathrm{dist}(w,\partial\Omega)>\mathbf{d}/4$, then $N(z,w)$ is written as
$$N(z,w)=\Gamma(z,w)+H(z,w),$$
where $H(z,w)$ is a regular function, and for any $s\in\mathbb{N}^+$, there is a constant $C$ depending on $s$, $\mathbf{d}$, and $\Omega$, such that
\begin{equation}\label{CA408}
\|\nabla^sH(\cdot,w)\|_{L^\infty}\leq C.
\end{equation}
\end{lemma}
\begin{proof}
According to \eqref{CC408}, we check that  
\begin{equation*}
\begin{split}
&\partial_nH(z,w)=-\frac{1}{4\pi}\left(\frac{z}{z-w}+\frac{\bar z}{\bar z-\bar w}\right)~~z\in\Gamma_0,\\
&\partial_nH(z,w)=\frac{1}{4\pi r_j}\left(\frac{z-b_j}{z-w}+\frac{\bar z-\bar{b}_j}{\bar z-\bar w}\right)~~z\in\partial\Omega\setminus\Gamma_0.\\
\end{split}
\end{equation*}
Note that $|z-w|>\mathbf{d}/4$, since $z\in\partial\Omega$ and $\mathrm{dist}(w,\partial\Omega)>\mathbf{d}/4$. Therefore $\partial_n H(z,w)$ is smooth and bounded uniformly on $\partial\Omega$. Thus we apply the Schauder estimates of \eqref{CC408} to deduce \eqref{CA408}. The proof is therefore completed.
\end{proof}

In view of Lemma \ref{LL42}, we need only focusing on the case that $w$ approaches $\partial\Omega$, while Lemma \ref{LL41} completely describe the singularity of $N(z,w)$ and $H(z,w)$ near the boundary. Combining these two lemmas leads to the estimates below.
\begin{lemma}\label{T41}
Let $\Omega$ be the circular domain given by \eqref{CA414}--\eqref{CA402} and $H(z,w)$ be determined by \eqref{CC408}.  
If we have $\mathrm{dist}(w,\Gamma_0)\leq\mathbf{d}/4$, then there is a constant $C$ depending on $\Omega$ such that,
\begin{equation}\label{C408}
\begin{split}
\big|\nabla H(z,w)\big|\leq C|z-w_0|^{-1},~~
\big|\nabla^2 H(z,w)\big|\leq C|z-w_0|^{-2},
\end{split}
\end{equation}
where $w_0$ can be chosen as $w/|w|^2$ or $w/|w|$, regarded as the projection of $w$ with respect to $\Gamma_0$. 
Similarly, if for some $j\in\{1,\cdots,k-1\}$, we have $\mathrm{dist}(w,\Gamma_j)\leq\mathbf{d}/4$, then it holds that
\begin{equation}\label{CA409}
\begin{split}
\big|\nabla H(z,w)\big|\leq C|z -w_j|^{-1},~~
\big|\nabla^2 H(z,w)\big|\leq C|z-w_j|^{-2},
\end{split}
\end{equation}
where $w_j$ can be chosen as $b_j+r_j^2(w-b_j)/|w-b_j|^2$ or $b_j+r_j(w-b_j)/|w-b_j|$, regarded as the projection of $w$ with respect to $\Gamma_j$.

In particular, combining \eqref{CA408}--\eqref{CA409} leads to the following estimates on the Green's function $N(z,w)=(2\pi)^{-1}\log|z-w|+H(z,w)$,
\begin{equation}\label{CA449}
|\nabla N(z,w)|\leq C|z-w|^{-1},~~|\nabla^2 N(z,w)|\leq C|z-w|^{-2}.
\end{equation}
\end{lemma}
\begin{proof}
We prove \eqref{C408} for $\nabla=\partial_{y_1}$ and $\nabla^2=\partial_{y_1}^2$. In view of \eqref{cc}, it holds that
$$
\partial_{y_1}H(z,w)= (\partial_z+\partial_{\bar{z}})H(z,w)~~\mathrm{and}~~
\partial_{y_1}^2H(z,w)=(\partial_z+\partial_{\bar{z}})^2H(z,w).
$$
The estimates for other derivatives and $\eqref{CA409}$ can be established by similar methods, once we apply the fact
\begin{equation}\label{CA410}
H_j(x,y)=H_j(y,x),~~ R_j(x,y)=R_j(y,x),~~j\in\{0,1,\cdots,k-1\}.
\end{equation}

With the help of Lemma \ref{LL41} and Remark \ref{RR42}, when $\mathrm{dist}(w,\Gamma_0)\leq\mathbf{d}/4$, we infer that
$$H(z,w)=H_0(z,w)+R_0(z,w).$$
Therefore, we apply \eqref{CC407} and \eqref{CA403} to calculate directly that 
\begin{equation*}
\begin{split}
\partial_{y_1}H(z,w)
&= (\partial_z+\partial_{\bar{z}})H_0(z,w)+\partial_{y_1}R_0\\
&=\frac{1}{2\pi}\mathrm{Re}\left( \partial_z\log(1-z\bar{w})\right)+\partial_{y_1}R_0\\
&=\frac{1}{2\pi}\mathrm{Re}\left(\frac{-1}{z-\bar{w}^{-1}}\right)+\partial_{y_1}R_0,\\
\end{split}
\end{equation*}
similarly, we also have
\begin{equation*}
\begin{split}
\partial_{y_1}^2H(z,w)
&=(\partial_z+\partial_{\bar{z}})^2H_0(z,w)+\partial_{y_1}^2R_0\\
&=\frac{1}{2\pi}\mathrm{Re}\left( \partial_z^2\log(1-z\bar{w})\right)+\partial_{y_1}^2R_0\\
&=\frac{1}{2\pi}\mathrm{Re}\left( \frac{1}{(z-\bar{w}^{-1})^2}\right)+\partial_{y_1}^2R_0.
\end{split}
\end{equation*}
Then \eqref{C408} follows form \eqref{CA404}, \eqref{CA410}, and the fact that
\begin{equation}\label{CA426}
\left|z-\frac{w}{|w|^2}\right|\geq\left|z-\frac{w}{|w|}\right|,~~\left|z-\frac{w}{|w|^2}\right|\geq C\left|z-w\right|,~~~~\mathrm{for}\ z,w\in B_0.
\end{equation}
The proof for other derivatives in \eqref{C408}--\eqref{CA449} follows in the similar way. We therefore finish the lemma.
\end{proof}

\subsection{Cancellation properties of Green's functions}\label{SS42}
\quad Let us apply the structural Lemma \ref{LL41} of Green's functions to study the crucial term  given by \eqref{CC504},
$$\partial_{x_i}N(x,y)+\partial_{y_i}N(x,y).$$
In view of \eqref{CA449}, the singularity of such term should be order one, since
$$\left|\partial_{x_i}N(x,y)+\partial_{y_i}N(x,y)\right|\leq C|\nabla N(x,y)|\leq C|x-y|^{-1}.$$
However, we show that there is a certain cancellation structure which cuts down the singularity of the above estimates, which in some sense turns out to be
\begin{equation}\label{CA411}
\left|\partial_{x_i}N(x,y)+\partial_{y_i}N(x,y)\right|\leq C.
\end{equation}
Let us first establish \eqref{CA411} on the circular domain, and then extend it to the general $k$-connected domains with the help of conformal mappings.

We still assume that $\Omega$ is the circular domain given by \eqref{CA414}--\eqref{CA402} and $\mathbf{d}$ is defined in \eqref{CA401}.
Let us illustrate the source of our arguments by considering the outermost component $\Gamma_0$, which is a unit circle. Suppose that $w=x_1+ix_2$ is close to $\Gamma_0$ with $\mathrm{dist}(w,\Gamma_0)\leq\mathbf{d}/4$, then the projection $w/|w|$ lies in $\Gamma_0$, moreover $w=x_1+ix_2$ itself gives the outer normal vector of $\Gamma_0$ at $w/|w|$, thus $\tau=i\cdot w=-x_2+ix_1$ provides a tangent vector of $\Gamma_0$ at $w/|w|$. Then the crucial cancellation structure is given by the following term,
\begin{equation}\label{CA437}
\begin{split}
&\sum_{s=1}^2 \tau_s\,\partial_{x_s}N(x,y)+
\tau_s\,\partial_{y_s}N(x,y).\\
\end{split}
\end{equation}
By virtue of \eqref{CC503} and \eqref{CC407}, we obtain 
\begin{equation}\label{CA422}
N(x,y)=\Gamma(x,y)+H(x,y),~~\partial_{x_s}\Gamma(x,y)+\partial_{y_s}\Gamma(x,y)=0~~\mathrm{for}\ s=1,2.
\end{equation}
with $H(x,y)$ given by \eqref{CC408}. Thus \eqref{CA422} further reduces \eqref{CA437} to
\begin{equation}\label{CA420}
\begin{split}
&\sum_{s=1}^2 \tau_s\,\partial_{x_s}N(x,y)+
\tau_s\,\partial_{y_s}N(x,y)=\sum_{s=1}^2 \tau_s\,\partial_{x_s}H(x,y)+
\tau_s\,\partial_{y_s}H(x,y).\\
\end{split}
\end{equation}
If we apply  complex variables and $\tau=-x_2+ix_1=i\cdot w$, then \eqref{CA420} can be transformed into a very neat form as
\begin{equation}\label{CA423}
\begin{split}
\sum_{s=1}^2 \tau_s\,\partial_{x_s}H(x,y)+
\tau_s\,\partial_{y_s}H(x,y)
&= 2\,\mathrm{Re}\big(\tau\,\partial_z H(z,w)+\bar{\tau}\,\partial_{\bar{w}}H(z,w)\big).
\end{split}
\end{equation}
Now, let us illustrate the cancellation properties due to $\left(\tau\,\partial_z+\bar{\tau}\,\partial_{\bar{w}}\right)H(z,w)$. In view of Lemma \ref{LL42}, we still focus on the case that $w$ approaches $\partial\Omega$.
\begin{lemma}\label{L43}
Let $\Omega$ be the circular domain given by \eqref{CA414}--\eqref{CA402}. If $\mathrm{dist}(w,\Gamma_0)\leq\mathbf{d}/4$, then there is a constant $C$ depending only on $\Omega$ such that,
\begin{equation}\label{j3}
\big|\left(\tau\,\partial_z+\bar{\tau}\,\partial_{\bar{w}}\right)H(z,w)\big|\leq C, ~~\mathrm{with}\ \tau=i\cdot w.
\end{equation}
We mention that $\tau$ lies in the tangential direction of $\Gamma_0$ at the projection $w/|w|$. Similarly, if for some $j\in\{1,\cdots,k-1\}$, we have $\mathrm{dist}(w,\Gamma_j)\leq\mathbf{d}/4$, where $\Gamma_j$ is the boundary of the disc $B_j=\{|z-b_j|<r_j\}$, then it holds that
\begin{equation}\label{CA415}
\big| \left(\tau\,\partial_z+\bar{\tau}\,\partial_{\bar{w}}\right)H(z,w)\big|\leq C,~~\mathrm{with}\ \tau=i\cdot(w-b_j).
\end{equation}
Note that $\tau$ lies in the tangential direction of $\Gamma_j$ at $r_j(w-b_j)/|w-b_j|$.
\end{lemma}
\begin{proof}
We first establish \eqref{j3}. By virtue of Lemma \ref{LL41}, when $\mathrm{dist}(w,\Gamma_0)\leq\mathbf{d}/4$, we admit the decomposition due to \eqref{CA403},
$$H(z,w)=H_0(z,w)+R_0(z,w).$$
For $H_0(z,w)=(4\pi)^{-1}\big(\log(1-z\bar{w})+\log(1-\bar{z}w)\big)$, we check that
\begin{equation}\label{CA417}
\begin{split}
&\tau\,\partial_z H_0(z,w)+\bar{\tau}\,\partial_{\bar{w}}H_0(z,w)=\frac{1}{4\pi}\left(\frac{-iw\cdot\bar{w}}{1-z\bar{w}}+\frac{i\bar{w}\cdot z}{1-z\bar{w}}\right)=\frac{-i(z-w)}{4\pi(z-\bar{w}^{-1})}.
\end{split}
\end{equation}
Thus \eqref{CA404}, \eqref{CA410}, and \eqref{CA417} lead to
\begin{equation*}
\begin{split}
\big|\tau\,\partial_z H(z,w)+\bar{\tau}\,\partial_{\bar{w}}H(z,w)\big|\leq\frac{1}{4\pi}\left|\frac{z-w}{z-\bar{w}^{-1}}\right|+|\nabla R_0(z,w)|\leq C,
\end{split}
\end{equation*}
where we have applied $\mathrm{dist}(w,\Gamma_0)\leq\mathbf{d}/4$ to ensure that $|z-w|\leq C|z-\bar{w}^{-1}|$.

Now \eqref{CA415} is proved in the similar methods. Suppose that for some $j\in\{1,\cdots,k-1\}$, we have $\mathrm{dist}(w,\Gamma_j)\leq\mathbf{d}/4$, then Lemma \ref{LL41} guarantees that we can decompose $H(z,w)$ as
\begin{equation*}
H(z,w)=H_j(z,w)+R_j(z,w),
\end{equation*}
with $H_j(z,w)$ provided by \eqref{hk}. Note that $\tau=i(w-b_j)$, thus we have
\begin{equation}\label{CA418}
\begin{split}
&\tau\,\partial_z H_j(z,w)+\bar{\tau}\,\partial_{\bar{w}}H_j(z,w)\\
&=\frac{\,r_j^2}{4\pi}\left(\frac{i(w-b_j)\,(z-b_j)^{-1}}{(z-b_j)(\bar{w}-\bar{b}_j)-r_j^2 }+\frac{-i(\bar{w}-\bar{b}_j)\,(\bar{w}-\bar{b}_j)^{-1}}{(z-b_j)(\bar{w}-\bar{b}_j)-r_j^2 } \right)\\
&=\frac{-i\,r_j^2}{4\pi(z-b_j)(\bar{w}-\bar{b}_j)}\cdot\frac{(z-b_j)-(w-b_j)}{(z-b_j)-r_j^2(\bar{w}-\bar{b}_j)^{-1}}.
\end{split}
\end{equation}
Combining \eqref{CA404}, \eqref{CA410}, and \eqref{CA418} implies that
\begin{equation*}
\begin{split}
&\big|\tau\,\partial_z H(z,w)+\bar{\tau}\,\partial_{\bar{w}}H(z,w)\big|\\
&\leq\frac{1}{4\pi}\left|\frac{(z-b_j)-(w-b_j)}{(z-b_j)-r_j^2(\bar{w}-\bar{b}_j)^{-1}}\right|+|\nabla R_j(z,w)|\leq C.
\end{split}
\end{equation*}
We mention that $r_j^2(\bar{w}-\bar{b}_j)^{-1}$ is the reflection of $(w-b_j)$ with respect to $\Gamma_j$, thus $\mathrm{dist}(w,\Gamma_j)\leq\mathbf{d}/4$ ensures that $|z-w|\leq C|(z-b_j)-r_j^2(\bar{w}-\bar{b}_j)^{-1}|$. Consequently, \eqref{CA415} is valid and the proof is completed.
\end{proof}

As an extension of Lemma \ref{L43}, we deduce that the cancellation properties are valid for higher order derivatives.
\begin{lemma}\label{LL45}
Suppose that $\Omega$ is the circular domain given by \eqref{CA414}--\eqref{CA402}, and for some $j\in\{0,1,\cdots,k-1\}$, we have $\mathrm{dist}(w,\Gamma_j)\leq\mathbf{d}/4$, then there is a constant $C$ depending only on $\Omega$ such that,
\begin{equation}\label{CA419}
\big|\left(\tau\,\partial_z +\bar{\tau}\,\partial_{\bar{w}}\right)\nabla H(z,w)\big|\leq C|z-w_j|^{-1},
\end{equation}
where $w_j= b_j+r_j(w-b_j)/|w-b_j|$ is the projection of $w$ to $\Gamma_j$ given by Lemma \ref{T41}, and $\tau=i\cdot(w-b_j)$ is a tangential vector of $\Gamma_j$ at $w_j$.
\end{lemma}
\begin{proof}
We prove \eqref{CA419} for $\nabla=\partial_z$ under the condition that $\mathrm{dist}(w,\Gamma_0)\leq\mathbf{d}/4$. While the remaining cases follow in the same approaches. 

According to Lemma \ref{LL41}, when $\mathrm{dist}(w,\Gamma_0)\leq\mathbf{d}/4$, we admit the decomposition
$$H(z,w)=H_0(z,w)+R_0(z,w),$$
with $H_0=(4\pi)^{-1}\big(\log(1-z\bar{w})+\log(1-\bar{z}w)\big)$. Thus we calculate directly that
\begin{equation}\label{CA421}
\begin{split}
\big(\tau\,\partial_z +\bar{\tau}\,\partial_{\bar{w}}\big)\partial_z H_0(z,w)
&=\frac{-1}{4\pi}\big(\tau\,\partial_z +\bar{\tau}\,\partial_{\bar{w}}\big)\frac{\bar{w}}{1-z\bar{w}}=\frac{i\,(|w|^2-1)}{4\pi\,(z-\bar{w}^{-1})}\cdot\frac{1}{1-z\bar{w}}.
\end{split}
\end{equation}
In particular, if $\mathrm{dist}(w,\Gamma_0)\leq\mathbf{d}/4$ and $z$ lies in the unit disc, we check that
\begin{equation*}
|z-\bar{w}^{-1}|\geq\left|\frac{w}{|w|}-\bar{w}^{-1}\right|=\frac{1-|w|}{|w|},
\end{equation*}
which together with \eqref{CA404}, \eqref{CA426}, and \eqref{CA421} gives 
\begin{equation*}
\left|\big(\tau\,\partial_z +\bar{\tau}\,\partial_{\bar{w}}\big)\partial_z H(z,w)\right|\leq C|1-z\bar{w}|^{-1}+|\nabla^2R_0(z,w)|\leq C|z-w_0|^{-1}.
\end{equation*}
We therefore finish \eqref{CA419} for $j=0$ and the proof of lemma is completed.
\end{proof}

\begin{remark} 
Let us go back to the real variables. Lemma \ref{LL45} in fact indicates that, if $\mathrm{dist}(x,\Gamma_j)\leq\mathbf{d}/4$ for some $j\in\{0,1,\cdots,k-1\}$, then we can apply \eqref{CA422}, \eqref{CA423}, and \eqref{CA419} to find a constant $C$ depending only on $\Omega$, such that for $p=1,2$,
\begin{equation}\label{C422}
\big|\sum_{s=1}^2(\tau_s\partial_{x_s}+
\tau_s\partial_{y_s})\partial_{y_p}N(x,y)\big|
\leq C\left|y-x_j\right|^{-1}.
\end{equation}
with $x_j\triangleq b_j+ r_j(x-b_j)/|x-b_j|.$ Compared with \eqref{CA449}, the singularity of \eqref{C422} is actually cut down by order one.
\end{remark}
\begin{remark}
The above calculations are horrible if they are carried out in real variables. However, with the help of the complex variables, we clearly identify the principal parts of $N(z,w)$ near $\partial\Omega$,
$$\log\left(1-z\cdot\bar{w}\right)\quad\mathrm{and}\quad\log\left(1-\frac{r_j}{z-b_j}\frac{r_j}{\bar{w}-\bar{b}_j}\right),$$
as well as the key structure cancelling out the singularity,
$$\tau\,\partial_z+\bar{\tau}\,\partial_{\bar{w}}~~\mathrm{with}\ \tau=i\,(w-b_j).$$
These calculations once again illustrate the great usefulness of complex variables.
\end{remark}

Next, let us derive the analogous cancellation properties of the Green's function on the general $k$-connected domain $\Omega$ satisfying 
$$\partial\Omega=\bigcup_{i=0}^{k-1}\Gamma_i,~~\Gamma_i\cap\Gamma_j=\varnothing\ \  \mathrm{if}\ i\neq j,$$
where $\{\Gamma_j\}_{j=0}^{k-1}$ are $k-1$ general simply closed smooth curves, with $\Gamma_0$ the outmost part. Note that \eqref{CA420}--\eqref{CA423} and Lemma \ref{L43}  suggest us focusing on the harmonic function $H(z,w)$ given by \begin{equation}\label{CA430}
\begin{cases}
\Delta H(\cdot,w)=0\ \mathrm{in}\ \Omega,\\
\partial_nH(\cdot,w)=-\partial_n\Gamma(\cdot,w)+l^{-1}\ \mathrm{on}\ \partial\Omega.
\end{cases}
\end{equation}

By virtue of Lemma \ref{L27}, there is a conformal mapping $\varphi$ between $\Omega$ and some circular domain $\mathcal{C}$, which can be extended smoothly to the boundary as 
$\varphi:\overline{\Omega}\rightarrow\overline{\mathcal{C}}.$ Without loss of generality, we assume $B_0$ is the unit disc and $\mathcal{C}$ satisfies 
\begin{equation*}
\begin{split}
\mathcal{C}&=B_0\setminus\big(\bigcup_{j=0}^{k-1}\overline{B_j}\big),~~B_j\triangleq\{\zeta\in B_0\,\big|\,|\zeta-b_j|<r_j\},\\
&\mbox{$\varphi$ maps $\Gamma_j$ onto $\partial B_j$ for $j\in\{0,1,\cdots,k-1\}$}. 
\end{split}
\end{equation*} 
 
In particular, for $\zeta,\eta\in\mathcal{C}$, we consider the corresponding harmonic function  $H_\mathcal{C}(\zeta,\eta)$ on $\mathcal{C}$ solved by 
\begin{equation}\label{CA431}
\begin{cases}
\Delta H_\mathcal{C}(\cdot,\eta)=0\ \mathrm{in}\ \mathcal{C},\\
\partial_nH_\mathcal{C}(\cdot,\eta)=-\partial_n\Gamma(\cdot,\eta)+l^{-1}\ \mathrm{on}\ \partial\mathcal{C}.
\end{cases}
\end{equation} 
Then, let us define the pull-back harmonic function $H^*(z,w)$ on $\Omega$ by
\begin{equation}\label{CA424}
H^*(z,w)=H_\mathcal{C}(\zeta,\eta)~~~~\mathrm{with}\ z,w\in\Omega\  \mathrm{and}\ \zeta=\varphi(z),\,\eta=\varphi(w).
\end{equation}
Now we prove the key assertion that the error  $E(z,w)\triangleq H(z,w)-H^*(z,w)$ between these two functions is regular.
\begin{lemma}\label{LL46}
Suppose that $\Omega$ is a $k$-connected domain satisfying \eqref{C101}, with $E(z,w)$, $H(z,w)$, and $H^*(z,w)$ defined above. Then for any $s\in\mathbb{N}^+$, there is a constant $C$ depending only on $s$, $\Omega$, and $\varphi$, such that
\begin{equation}\label{CA425}
\|\nabla^sE(\cdot,w)\|_{L^\infty(\Omega)}\leq C.
\end{equation}
\end{lemma}
\begin{proof}
According to the definition \eqref{CA424}, the error term is given by
\begin{equation*} 
\begin{split}
E(z,w) 
&=H(z,w)-H_\mathcal{C}(\varphi(z),\varphi(w)).
\end{split}
\end{equation*}
In particular, for $z\in\partial\Omega$, let $\alpha=\alpha_1+i\alpha_2$ be the unit outer normal vector of $\partial\Omega$ at $z$, then we apply \eqref{CC408}$_2$ and the chain rule to check that
\begin{equation*}
\begin{split}
\partial_n(H_\mathcal{C}\circ\varphi)(\cdot,w)&=\alpha_i\cdot\partial_i\varphi_j
\cdot\partial_j H_{\mathcal{C}}\big|_{(\varphi(z),\varphi(w))}\\
&=-\alpha_i\cdot\partial_i\varphi_j
\cdot\partial_j \Gamma\big|_{(\varphi(z),\varphi(w))}+|\varphi'(z)|\cdot l_\mathcal{C}^{-1}\\
&=-\partial_n(\Gamma\circ\varphi)(\cdot,w)+|\varphi'(z)|\cdot l_\mathcal{C}^{-1},
\end{split}
\end{equation*}
where $l_\mathcal{C}$ is the total length of $\partial\mathcal{C}$. Note that we have applied Lemma \ref{L28} in the second line: The conformal mapping $\varphi$ preserves the angle, thus $\alpha_i\cdot\partial_i\varphi_j\cdot\partial_j$ is in fact the normal derivatives at $\varphi(z)\in\partial\mathcal{C}$.

Consequently, \eqref{CA430} and \eqref{CA431} imply that $E(z,w)$ solves the Neumann problem
\begin{equation}\label{CA428}
\begin{cases}
\Delta E(\cdot,w)=0\ \mathrm{in}\ \Omega,\\
\partial_nE(\cdot,w)=\partial_n(\Gamma\circ\varphi-\Gamma)(\cdot,w)+(l^{-1}-|\varphi'(z)|\cdot l_\mathcal{C}^{-1})\ \mathrm{on}\ \partial\Omega,
\end{cases}
\end{equation}
where $l$ is the total length of $\partial\Omega$. Recalling that $\Gamma(z,w)=(2\pi)^{-1}\log|z-w|$, we apply \eqref{CA201} to check that
\begin{equation}\label{CA433}
\begin{split}
\partial_n(\Gamma\circ\varphi-\Gamma)(z,w)
=\frac{1}{2\pi}\,\mathrm{Re}\left(\alpha\cdot\partial_z\log\big(\frac{\varphi(z)-\varphi(w)}{z-w}\big)\right).
\end{split}
\end{equation}
Note that Lemma \ref{L28} guarantees that 
$$\varphi(z)\in C^\infty(\overline{\Omega}),~~
\min_{z,w\in\overline{\Omega}}
\left|\frac{\varphi(z)-\varphi(w)}{z-w}\right|>0,$$ 
which implies that \eqref{CA433} and $\partial_nE(\cdot,w)$ are smooth and bounded uniformly on $\partial\Omega$. 
We therefore apply the Schauder estimates (\cite[Chapter 3]{ADN}) of \eqref{CA428} to yield that for any $s\in\mathbb{N}^+$, there is a constant $C$ depending on $s$, $\mathbf{d}$, and $\Omega$, such that
\begin{equation*} 
\|\nabla^s E(\cdot,w)\|_{L^\infty(\Omega)}\leq C,
\end{equation*}
which gives \eqref{CA425} and finishes the proof.
\end{proof}

From a viewpoint, with the help of conformal mappings, Lemma \ref{LL46} reduces the general $k$-connected domains to the standard circular domains. 
It ensures that $H(z,w)$ and $H_\mathcal{C}(\zeta,\eta)$ are quite similar, thus the previous cancellation structure for $H_\mathcal{C}(\zeta,\eta)$ should be preserved for $H(z,w)$ as well. The next lemma discusses such issue.
\begin{lemma}\label{LL47}
Suppose that $\Omega$ is a $k$-connected domain satisfying \eqref{C101}, and $H(z,w)$ is given by \eqref{CA430}. If $\mathrm{dist}(w,\Gamma_0)\leq\mathbf{d}/4$, with $\mathbf{d}$ given by \eqref{CA401}. Then there is a constant $C$ depending only on $\Omega$ and $\varphi$, such that 
\begin{equation}\label{C429}
\big|(\tau\partial_z+\bar{\tau}\partial_{\bar{w}})\nabla H(z,w)\big|
\leq C\left|z-w_0\right|^{-1},
\end{equation}
where $w_0\triangleq \varphi^{-1}\big(\varphi(w)/|\varphi(w)|\big)$ is the projection of $w$ on $\Gamma_0$ induced by $\varphi$, and $\tau=\tau_1+i\tau_2$ is the unit tangential vector of $\partial\Omega$ at $w_0$.

Similarly, if $\mathrm{dist}(x,\Gamma_j)\leq\mathbf{d}/4$ for some $j\in\{1,\cdots,k-1\}$, then it holds that
\begin{equation*} 
\big|(\tau\partial_z+\bar{\tau}\partial_{\bar{w}})\nabla H(z,w)\big|
\leq C\left|z-w_j\right|^{-1},
\end{equation*}
where $w_j\triangleq \varphi^{-1}\big(b_j+(\varphi(w)-b_j)/|\varphi(w)-b_j|\big)$ is the projection of $w$ on $\Gamma_j$ induced by $\varphi$, and $\tau=\tau_1+i\tau_2$ is the unit tangential vector of $\partial\Omega$ at $w_j$.
\end{lemma}
\begin{proof}
We only prove \eqref{C429} for $\nabla=\partial_z$ and $j=0$, while the general cases follow in the same approaches. We still denote that $\zeta=\varphi(z)$ and $\eta=\varphi(w)$, while in view of Lemma \ref{L28}, it is harmless to assume that $ \mathrm{dist}(\eta,\partial B_0)\leq 1/4$.

According to Lemma \ref{LL46}, it holds that 
\begin{equation}\label{CA435}
\begin{split}
&\left|(\tau\,\partial_z+\bar{\tau}\,\partial_{\bar{w}})\partial_z H(z,w)\right|\\
&\leq\left|(\tau\,\partial_z+\bar{\tau}\,\partial_{\bar{w}})\,\partial_z H^*(z,w)\right|+\left|\nabla^2E(z,w)\right|\\
&\leq\left|(\tau\,\partial_z+\bar{\tau}\,\partial_{\bar{w}})\,\partial_z H_\mathcal{C} (\zeta,\eta) \right|+C.
\end{split}
\end{equation}
By virtue of the chain rule and \eqref{CA201}, for smooth functions $f(\zeta)$ and $g(\eta)$ defined on $\mathcal{C}$, we argue that
\begin{equation}\label{CA434}
\begin{split}
\partial_z f(\zeta) 
&=\partial_z\varphi(z)\cdot\partial_\zeta f(\zeta)+\partial_z\overline{\varphi(z)}\cdot\partial_{\bar{\zeta}} f(\zeta) =\varphi'(z)\cdot\partial_{\zeta}f(\zeta),\\
\partial_{\bar{w}} g(\eta) 
&=\partial_{\bar{w}}\varphi(w)\cdot\partial_\eta g(\eta)+\partial_{\bar{w}}\overline{\varphi(w)}\cdot\partial_{\bar{\eta}} g(\eta) =\overline{\varphi'(w)}\cdot\partial_{\bar{\eta}}g(\eta).
\end{split}
\end{equation}
Thus let us utilize \eqref{CA434} in \eqref{CA435} and carefully calculate that 
\begin{equation}\label{CA440}
\begin{split}
&(\tau\,\partial_z+\bar{\tau}\,\partial_{\bar{w}})\,\partial_z H_\mathcal{C}(\zeta,\eta)\\
&=(\tau\,\partial_z+\bar{\tau}\,\partial_{\bar{w}})\left( \varphi'(z)\cdot \partial_\zeta H_\mathcal{C}(\zeta,\eta)\right)\\
&=\tau\cdot\varphi''(z)\cdot\partial_\zeta H_\mathcal{C}(\zeta,\eta)+\varphi'(z)\cdot(\tau\,\partial_z+\bar{\tau}\,\partial_{\bar{w}}) \left(\partial_\zeta H_\mathcal{C}(\zeta,\eta)\right)\\
&=\tau\cdot\varphi''(z)\cdot\partial_\zeta H_\mathcal{C}(\zeta,\eta)+\varphi'(z)\cdot \left(\tau\cdot\varphi'(z)\cdot\partial_\zeta+\bar{\tau}\cdot\overline{\varphi'(w)}\cdot\partial_{\bar{\eta}}\right) \partial_\zeta H_\mathcal{C}(\zeta,\eta).\\
\end{split}
\end{equation}
In addition, since $\varphi$ maps $\Gamma_0$ to $\partial B_0$, we invoke \eqref{CA201} to ensure that, for $w_0\in\Gamma_0$, the vector $\tau\cdot\varphi'(w_0)$ is parallel to the unit tangential vector $\nu=\nu_1+i\nu_2$ of $\partial B_0$ at $\varphi(w_0)$, which in particular implies $\tau\cdot\varphi'(w_0)=|\varphi'(w_0)|\,\nu.$ 

Therefore the above equation transforms \eqref{CA440} into
\begin{equation}\label{CA442}
\begin{split}
&(\tau\,\partial_z+\bar{\tau}\,\partial_{\bar{w}})\,\partial_z H_\mathcal{C}(\zeta,\eta)\\
&=\varphi'(z)\cdot \left(\tau\cdot\varphi'(z)\cdot\partial_\zeta+\bar{\tau}\cdot\overline{\varphi'(w)}\cdot\partial_{\bar{\eta}}\right) \partial_\zeta H_\mathcal{C}(\zeta,\eta)+\tau\cdot\varphi''(z)\cdot\partial_\zeta H_\mathcal{C}(\zeta,\eta)\\
&=\varphi'(z)\cdot \left(\tau\cdot\varphi'(w_0)\cdot\partial_\zeta+\bar{\tau}\cdot\overline{\varphi'(w_0)}\cdot\partial_{\bar{\eta}}\right) \partial_\zeta H_\mathcal{C}(\zeta,\eta)+\mathcal{R}\\
&=\varphi'(z)\cdot\big|\varphi'(w_0)\big|\cdot\big(\nu\cdot\partial_\zeta+\bar{\nu}\cdot\partial_{\bar{\eta}}\big)\,\partial_\zeta H_\mathcal{C}(\zeta,\eta)+\mathcal{R},
\end{split}
\end{equation}
where the remaining term $\mathcal{R}$ is given by
\begin{equation*}
\begin{split}
\mathcal{R}&=\tau\cdot\big(\varphi'(z)-\varphi'(w_0)\big)\cdot\partial_\zeta^2 H_\mathcal{C}(\zeta,\eta)+\bar{\tau}\cdot\big(\overline{\varphi'(w)}-\overline{\varphi'(w_0)}\big)\cdot\partial_{\bar{\eta}} \partial_\zeta H_\mathcal{C}(\zeta,\eta)\\
&\quad+\tau\cdot\varphi''(z)\cdot\partial_\zeta H_\mathcal{C}(\zeta,\eta) \\
\end{split}
\end{equation*}
Now, we take advantage of Lemma \ref{T41} and Lemma \ref{LL45} to deduce that
\begin{equation}\label{CA436}
\begin{split}
\big|\nabla H_\mathcal{C}(\zeta,&\eta)\big|\leq C\big|\zeta-\eta/|\eta|^2\big|^{-1},~~\big|\nabla^2 H_\mathcal{C}(\zeta,\eta)\big|\leq C\big|\zeta-\eta/|\eta|^2\big|^{-2},\\ 
&\big|(\nu\cdot\partial_\zeta+\bar{\nu}\cdot\partial_{\bar{\eta}})\,\partial_\zeta H_\mathcal{C}(\zeta,\eta)\big|\leq C \big|\zeta-\eta/|\eta|\big|^{-1},
\end{split}
\end{equation}
where $\varphi(w_0)=\eta/|\eta|$ is exactly the projection of $\eta$ at $\partial B_0$.
Consequently, we apply Lemma \ref{L28} and \eqref{CA436} to declare that 
\begin{equation}\label{CA441}
\begin{split}
|\mathcal{R}|&\leq C\,\big(|z-w_0|+|w-w_0|\big)\,|\nabla^2 H_\mathcal{C}(\zeta,\eta)|+C\,|\nabla H_\mathcal{C}(\zeta,\eta)|\\
&\leq C\,\big(\big|\zeta-\eta/|\eta|\big|+\big|\eta-\eta/|\eta|\big|\big)\cdot|\nabla^2 H_\mathcal{C}(\zeta,\eta)|+C\,|\nabla H_\mathcal{C}(\zeta,\eta)|\\
&\leq C\,\big(\big|\zeta-\eta/|\eta|\big|+\big|\eta-\eta/|\eta|\big|\big)\cdot\big|\zeta-\eta/|\eta|^2\big|^{-2}+C\,\big|\zeta-\eta/|\eta|^2\big|^{-1}\\
&\leq C\,\big|\zeta-\eta/|\eta|\big|^{-1}
\leq C\,\big|z-w_0\big|^{-1},
\end{split}
\end{equation}
where we have made use of the fact
\begin{equation*}
\begin{split}
\left|\zeta-\frac{\eta}{|\eta|^2}\right|\geq\left|\frac{\eta}{|\eta|}-\frac{\eta}{|\eta|^2}\right|  \geq C \left|\eta-\frac{\eta}{|\eta|}\right|~~\mathrm{and}~~\left|\zeta-\frac{\eta}{|\eta|^2}\right|\geq \left|\zeta-\frac{\eta}{|\eta|}\right|,
\end{split}
\end{equation*}
since $\zeta,\eta\in B_0$ and $\mathrm{dist}(\eta,\partial B_0)\leq 1/4.$ Thus, we substitute \eqref{CA436} and \eqref{CA441} into \eqref{CA442} and utilize Lemma \ref{L28} to argue that
\begin{equation*} 
\begin{split}
&\left|(\tau\,\partial_z+\bar{\tau}\,\partial_{\bar{w}})\,\partial_z H_\mathcal{C}(\zeta,\eta)\right|\\
&\leq C\left|\big(\nu\cdot\partial_\zeta+\bar{\nu}\cdot\partial_{\bar{\eta}}\big)\,\partial_\zeta H_\mathcal{C}(\zeta,\eta)\right|+|\mathcal{R}|\\
&\leq C\big|z-w_0\big|^{-1},
\end{split}
\end{equation*}
which gives \eqref{C429} and finishes the proof.
\end{proof}

Let us apply real variables to conclude all properties of Green's functions on general $k$-connected domains obtained in Subsection \ref{SS41} and \ref{SS42}. 
\begin{lemma}\label{LL48}
Suppose that $\Omega$ is a $k$-connected domain satisfying \eqref{C101}, $H(x,y)$ is the harmonic function solved by \eqref{CC408}, and $N(x,y)$ is the Green's function defined by \eqref{CC407}. 
Then the following assertions are true:

1). If $x$ is away from the boundary and $\mathrm{dist}(x,\partial\Omega)>\mathbf{d}/4$, with $\mathbf{d}$ defined in \eqref{CA401}, then there is a constant $C$ depending only on $\Omega$ and $\varphi$, such that
\begin{equation}\label{CA451}
\big|\nabla H(x,y)\big|+\big|\nabla^2 H(x,y)\big|\leq C.
\end{equation}
If for some $j\in\{0,1,\cdots,k-1\}$, we have $\mathrm{dist}(x,\Gamma_j)\leq\mathbf{d}/4$, then there is a constant $C$ depending only on $\Omega$ and $\varphi$ such that
\begin{equation}\label{CA452}
|\nabla H(x,y)|\leq C|y-x_j|^{-1},~~|\nabla^2 H(x,y)|\leq C|y-x_j|^{-2}.
\end{equation}
where $x_j\triangleq \varphi^{-1}\big(b_j+ r_j(\varphi(x)-b_j)/|\varphi(x)-b_j|\big)$ is the projection of $x$ on $\Gamma_j$ induced by $\varphi$. Combining \eqref{CA451} and \eqref{CA452} leads to the corresponding estimates on $N(z,w)$,
\begin{equation}\label{CA450}
|\nabla N(x,y)|\leq C|y-x|^{-1},~~|\nabla^2 N(x,y)|\leq C|y-x|^{-2},~~~~\forall x,y\in\Omega.
\end{equation}

2). $($Cancellation properties$)$  
If $x$ is close to $\partial\Omega$ and for some $j\in\{0,1,\cdots,k-1\}$, we have $\mathrm{dist}(x,\Gamma_j)\leq\mathbf{d}/4$, then there is a constant $C$ depending only on $\Omega$ and $\varphi$, such that for $p=1,2$,
\begin{equation}\label{CA443}
\big|\sum_{s=1}^2(\tau_s\partial_{x_s}+
\tau_s\partial_{y_s})\partial_{y_p}N(x,y)\big|
\leq C \left|y-x_j\right|^{-1},
\end{equation}
where $x_j$ is the projection given above, and $\tau=(\tau_1,\tau_2)$ is the unit tangential vector of $\Gamma_j$ at $x_j$. Compared with \eqref{CA452} and \eqref{CA450}, the singularity of \eqref{CA443} is actually cut down by order one.
\end{lemma}
\begin{proof}
We mention that \eqref{CA451}--\eqref{CA450} are obtained by the same method as Lemma \ref{LL47}, together with the previous Lemma \ref{LL42} and Lemma \ref{T41} on $H_\mathcal{C}(\zeta,\eta)$ as
\begin{equation*}
\begin{split}
\big|&\nabla^s H_\mathcal{C}(\zeta,\eta_j)\big|\leq C\,|\zeta-\eta_j|^{-s},~~\mbox{ when $\eta$ is near $\partial B_j$ and $s\in\{1,2\}$},\\
&\quad\big|\nabla H_\mathcal{C}(\zeta,\eta)\big|+\big|\nabla^2 H_\mathcal{C}(\zeta,\eta)\big|\leq C,~~\mbox{when $\eta$ is away from $\partial\mathcal{C}$}.
\end{split}
\end{equation*}
Moreover, combining \eqref{CA420}, \eqref{CA423}, and \eqref{CA443} leads to \eqref{CA443}. The proof of the lemma is therefore completed.
\end{proof}

\subsection{The commutators on multi-connected domains}\label{SS43}
\quad At the present stage, we can establish a type of commutator arguments which help us extending the Lions' theory (shown in \eqref{CC501}) from original Cauchy problems to general $k$-connected domains $\Omega$ satisfying \eqref{C101}.

Let $(\rho,u)$ be a smooth solution to the system \eqref{11}--\eqref{15} on $\Omega\times(0,T)$. Recalling that in the case $\Omega=\mathbb{R}^2$, the viscous effective flux
$F=(2\mu+\lambda)\mathrm{div}u-(P-\hat{P})$ 
is directly solved by
\begin{equation}\label{CA445}
F=\frac{\mathrm{d}}{\mathrm{d}t}\Delta^{-1}\mathrm{div}(\rho u)-[u_i,R_{ij}](\rho u_j),~~\mathrm{with}\ R_{ij}=\partial_i\partial_j\Delta^{-1}=\partial_i\Delta^{-1}\partial_j.
\end{equation} 
We will derive the similar equality when $\Omega$ is a general $k$-connected domain. In fact, by taking $\mathrm{div}(\cdot)$ on \eqref{11}$_2$, we apply boundary conditions \eqref{15} to deduce that $F$ solves the Neumann problem,
\begin{equation}\label{C501}
\begin{cases}
\Delta F=\mathrm{div}(\rho \dot{u})~~\mathrm{in}\, \Omega,\\
\partial_n F=\rho\dot{u}\cdot n+\mu n^\bot\cdot\nabla(Ku\cdot n^\bot)~~\mathrm{on}\, \partial\Omega.
\end{cases}
\end{equation}
Let $N(x,y)$ be the Green's function on $\Omega$ given by
\begin{equation}\label{CA438}
\begin{split}
N(x,y)&=\Gamma(x,y)+H(x,y),~~\Gamma(x,y)=\frac{1}{2\pi}\log|x-y|,\\
&\begin{cases}
\Delta H(\cdot,y)=0~~\mathrm{in}\ \Omega,\\
\partial_nH(\cdot,y)=-\partial_n \Gamma(\cdot-y)+l^{-1}~~\mathrm{on}\ \partial\Omega, 
\end{cases}
\end{split}
\end{equation}
then we utilize the Green's function $N(x,y)$ to write out $F(x)$ explicitly.

\begin{lemma}\label{LL49}
Suppose that $F\in C([0,T];C^1(\overline{\Omega})\cap C^2(\Omega))$ solves the problem (\ref{C501}) and $N(x,y)$ is the Green's function given by \eqref{CA438}, then it holds that,
\begin{equation} \label{qp11}
\begin{split}
F(x)
&=\frac{\mathrm{d}}{\mathrm{d}t}\Delta^{-1}\mathrm{div}(\rho u)-\left[u_i,\,\partial_i\,\Delta^{-1}\,\partial_j\right](\rho u_j)+B(x)+R(x),~~\forall x\in\Omega,
\end{split}
\end{equation}
where we have defined that for any smooth function $f(x)$ on $\Omega$, 
$$\Delta^{-1}f(x)\triangleq\int_\Omega N(x,y)\cdot f(y)\,dy.$$

Consequently, it holds that
\begin{equation*} 
\Delta^{-1}\mathrm{div}(\rho u)= \int_\Omega N(x,y)\cdot\mathrm{div}(\rho u)(y)\,dy=-\int_\Omega\partial_{y_j} N(x, y)\cdot \rho u_j(y)\,dy,
\end{equation*}
since $u\cdot n=0$. Similarly the commutator is given by  
\begin{equation*} 
\left[u_i,\,\partial_i\,\Delta^{-1}\,\partial_j\right](\rho u_j)=-\int_\Omega\partial_{x_i}\partial_{y_j} N(x, y)\cdot\big(u_i(x)-u_i(y)\big)\cdot \rho u_j (y)\,dy.
\end{equation*}
In addition, the extra terms $B(x)$ and $R(x)$ are provided by
\begin{equation}\label{bd}
\begin{split}
&B(x)=\int_\Omega\left(\partial_{x_i}\partial_{y_j}+\partial_{y_i}\partial_{y_j}\right) N(x, y)\cdot \rho u_i u_j (y)\,dy,\\
R(x)&=l^{-1}\int_{\partial \Omega}F(y)\,dS_y
-\mu\int_{\partial \Omega}N(x, y)\cdot(n^\bot\cdot\nabla(Ku\cdot n^\bot))\,dS_y.
\end{split}
\end{equation}
\end{lemma}
\begin{remark}
Let us compare \eqref{CA445} with \eqref{qp11}, then it is quite clear that the obstacles due to the boundary $\partial\Omega$ are given by the extra terms $R(x)$ and $B(x)$.
\end{remark}
\begin{proof}
In view of \eqref{C501}, the standard elliptic theory (see \cite[Chapter 2]{GT}) implies that 
\begin{equation*} 
\begin{split}
F(x)&=\int N(x, y)\cdot\mathrm{div}(\rho\dot{u})(y)dy+\int_{\partial\Omega}\frac{\partial N}{\partial n}(x, y)\cdot F(y)\,dS_y\\
&\quad-\int_{\partial\Omega}N(x, y)\cdot
\left(\rho\dot{u}\cdot n+n^\bot\cdot\nabla(Ku\cdot n^\bot)\right)dS_y,
\end{split}
\end{equation*}
which along with integrating by parts and \eqref{CA438} leads to
\begin{equation}\label{C512}
\begin{split}
F(x) 
&=-\int_\Omega\partial_{y_j}{N}(x, y)\cdot \rho\dot{u}_j(y)\,dy+l^{-1}\int_{\partial\Omega} F(y)\,dS_y\\
&\quad-\mu\int_{\partial\Omega}{N}(x, y)\left(n^\bot\cdot\nabla(Ku\cdot n^\bot)\right)dS_y\\
&=-\int_\Omega\partial_{y_j} N(x, y)\cdot \rho\dot{u}_j(y)\,dy+R(x),
\end{split}
\end{equation}
where $R(x)$ is the lower order remaining term given in \eqref{bd}.
Then by taking advantage of \eqref{11}$_1$ and the boundary condition $u\cdot {n}\big|_{\partial  \Omega}=0$, we infer that
\begin{equation}\label{369}
\begin{split}
&-\int_\Omega\partial_{y_j} N(x, y)\cdot \rho\dot{u}_j(y)\,dy\\
&=-\int_\Omega\partial_{y_j} N(x, y)\cdot\big(\partial_t(\rho u_j)+\mathrm{div}(\rho u\otimes u_j)\big)dy\\
&=-\frac{\partial}{\partial t}\int_\Omega\partial_{y_j} N(x, y)\cdot \rho u_j (y)\,dy+\int_\Omega\partial_{y_i}\partial_{y_j} N(x, y)\cdot \rho u_iu_j (y)\,dy \\
&=-\frac{\mathrm{d}}{\mathrm{d}t}\int  \partial_{y_j}N(x, y)\cdot \rho u_j(y)\,dy+J(x)=\frac{\mathrm{d}}{\mathrm{d}t}\Delta^{-1}\mathrm{div}(\rho u)+J(x),
\end{split}
\end{equation}
where $J(x)$ satisfies that
\begin{equation}\label{CA446}
\begin{split}
J(x)
&=\int_\Omega u_i(x)\cdot\partial_{x_i}\partial_{y_j} N(x, y)\cdot \rho u_j (y)\,dy +\int_\Omega\partial_{y_i}\partial_{y_j} N(x, y)\cdot \rho u_iu_j (y)\,dy\\
&=\int_\Omega\partial_{x_i}\partial_{y_j} N(x, y)\cdot\big(u_i(x)-u_i(y)\big)\cdot \rho u_j (y)\,dy
+B(x)\\
&=-\left[u_i,\,\partial_i\,\Delta^{-1}\,\partial_j\right](\rho u_j)+B(x).
\end{split}
\end{equation} 
Substituting \eqref{369} and \eqref{CA446} into \eqref{C512} gives \eqref{qp11} and finishes the proof.
\end{proof}

Now, let us focus on \eqref{qp11}. The most difficult term is the boundary part $B(x)$, since in view of \eqref{CA450}, it holds that
$$|B(x)|\leq \int_\Omega|\nabla^2N(x,y)|\cdot\rho|u|^2(y)\,dy\leq C\int_\Omega\frac{\rho|u|^2}{|x-y|^2}\,dy.$$
Such estimate yields that the singularity of $B(x)$ is of order 2 and out of control in 2D. What saves us is that $B(x)$ falls into the scope of the cancellation structure developed in Subsection \ref{SS42}, which cuts down the singularity of $B(x)$ by order 1.

The next lemma makes detailed analysis on the inner part $\left[u_i,\,\partial_i\,\Delta^{-1}\,\partial_j\right](\rho u_j)$ and the boundary part $B(x)$.  
 
\begin{lemma}\label{qp08} 
Suppose that $\left[u_i,\,\partial_i\,\Delta^{-1}\,\partial_j\right](\rho u_j)$ and $B(x)$ are given by Lemma \ref{LL49}. In addition, let $\varphi:\overline{\Omega}\rightarrow\overline{\mathcal{C}}$ be a conformal mapping from $\Omega$ to some circular domain given by Lemma \ref{L27}, then there is a constant $C$ depending only on $\Omega$ and $\varphi$, such that the inner part is bounded by,
\begin{equation}\label{CA447}
\begin{split}
&\big|\left[u_i,\,\partial_i\,\Delta^{-1}\,\partial_j\right]\big(\rho u_j\big)(x)\big|
\leq C\int_\Omega\frac{|u(y)-u(x)|}{|y-x|^2}
\,\rho|u|(y)\,dy.\\
\end{split}
\end{equation} 
While the boundary part is controlled by
\begin{equation}\label{CA448}
\begin{split}
&|B(x)|
\leq C\,\sum_{j=0}^{k-1}\left(\int_\Omega\frac{|u(y)-u(x_j)|}{|y-x_j|^2}\,\rho|u| (y)\,dy+ \int_\Omega\frac{\rho|u|^2(y)}{|y-x_j|}\,dy\right).
\\
\end{split}
\end{equation} 
where $x_j\in\Gamma_j$
is the projection of $x$ to $\Gamma_j$ induced by $\varphi$ as in Lemma \ref{LL48}. Without loss of generality, we may assume $0$ is not in $\mathcal{C}$ to ensure $x_0$ is well defined.
\end{lemma}
\begin{proof}
According to \eqref{CA450}, we directly deduce that
\begin{equation*}
\begin{split}
\big|\left[u_i,\,\partial_i\,\Delta^{-1}\,\partial_j\right]\big(\rho u_j\big)(x)\big|
&=\left|\int_\Omega\partial_{x_i}\partial_{y_j} N(x, y)\cdot\big(u_i(x)-u_i(y)\big)\cdot \rho u_j (y)\,dy\right|\\
&\leq C\int_\Omega\frac{|u(x)-u(y)|}{|x-y|^2}\cdot\rho|u|(y)\,dy,
\end{split}
\end{equation*}
which provides \eqref{CA447}. 
Next we make use of Lemma \ref{LL48} to handle $R(x)$.  

If $x$ is away from the boundary and $\mathrm{dist}(x,\partial\Omega)>\mathbf{d}/4$, then we take advantage of \eqref{CA420} and \eqref{CA451} to declare that
\begin{equation}\label{CA453}
\begin{split}
|B(x)|&=\left|\int_\Omega\left(\partial_{x_i}\partial_{y_j}+\partial_{y_i}\partial_{y_j}\right) H(x, y)\cdot \rho u_i u_j (y)\,dy\right|\\
&\leq\int_\Omega\left|\left(\partial_{x_i}\partial_{y_j}+\partial_{y_i}\partial_{y_j}\right) H(x, y)\right|\cdot \rho |u|^2 (y)\,dy\\
&\leq C\int_\Omega\rho|u|^2(y)\,dy.
\end{split}
\end{equation}

If $x$ is close to $\partial\Omega$ and $\mathrm{dist}(x,\Gamma_s)\leq\mathbf{d}/4$ for some $s\in\{0,1,\cdots,k-1\}$, then we have
\begin{equation*}
\begin{split}
B(x)&=\int_\Omega\left(\partial_{x_i}\partial_{y_j}+\partial_{y_i}\partial_{y_j}\right) N(x, y)\cdot \rho u_i u_j (y)\,dy\\
&=\int_\Omega\left(\partial_{x_i}\partial_{y_j}+\partial_{y_i}\partial_{y_j}\right) N(x,y)\cdot\big(u_i(y)-u_i(x_s)\big)\cdot \rho u_j (y)\,dy\\
&\quad+\int_\Omega\left(\partial_{x_i}\partial_{y_j}+\partial_{y_i}\partial_{y_j}\right) N(x, y)\cdot u_i(x_s)\cdot \rho u_j (y)\,dy\\
&\triangleq B_{1}(x)+B_{2}(x),
\end{split}
\end{equation*}
where $x_s$ is the projection to $\Gamma_s$ defined in Lemma \ref{LL48}.
According to \eqref{CA420} and \eqref{CA452}, we declare that
\begin{equation}\label{CA454}
\begin{split}
|B_1(x)|&=\left|\int_\Omega\left(\partial_{x_i}\partial_{y_j}+\partial_{y_i}\partial_{y_j}\right) H(x,y)\cdot\big(u_i(y)-u_i(x_s)\big)\cdot \rho u_j (y)\,dy|\right|\\
&\leq\int_\Omega\big|(\partial_{x_i}\partial_{y_j}+\partial_{y_i}\partial_{y_j}) H(x,y)\big|\cdot\big|u_i(y)-u_i(x_s)\big|\cdot \rho|u|(y)\,dy\\
&\leq C\int\frac{|u(y)-u(x_s)|}{|y-x_s|^2}\cdot \rho |u| (y)\,dy.
\end{split}
\end{equation}
Moreover, the slip boundary conditions \eqref{15} ensures that $u(x_s)$ is parallel to the unit   tangential vector $\tau=(\tau_1,\tau_2)$ of $\Gamma_s$ at $x_s$, thus we obtain    $u(x_s)=\big(u(x_s)\cdot\tau\big)\,\tau$, which together with \eqref{CA452} and \eqref{CA443} leads to
\begin{equation}\label{CA455}
\begin{split}
|B_2(x)|&=\left|\int_\Omega \left(\partial_{x_i}\partial_{y_j}+\partial_{y_i}\partial_{y_j}\right) H(x, y)\cdot u_i(x_s)\cdot \rho u_j (y)\,dy\right|\\
&=\left|\int_\Omega \left(\partial_{x_i}\partial_{y_j}+\partial_{y_i}\partial_{y_j}\right) H(x, y)\cdot\big(u(x_s)\cdot\tau\big)\,\tau_i\cdot \rho u_j (y)\,dy\right|\\
&=\left|\int_\Omega \left(\partial_{x_i}\partial_{y_j}+\partial_{y_i}\partial_{y_j}\right) H(x, y)\cdot\big(\big(u(x_s)-u(y)\big)\cdot\tau\big)\tau_i\cdot \rho u_j (y)\,dy\right|\\
&\quad+\left|\int_\Omega \tau_i\left(\partial_{x_i}\partial_{y_j}+\partial_{y_i}\partial_{y_j}\right) H(x, y)\cdot(u(y)\cdot\tau)\cdot \rho u_j(y)\,dy\right|\\
&\leq C\int_\Omega|\nabla^2 H(x,y)|\cdot\big|u(y)-u(x_s)\big|\cdot \rho |u|(y)\,dy\\
&\quad+C\int_\Omega\big|\tau_i\left(\partial_{x_i}\partial_{y_j}+\partial_{y_i}\partial_{y_j}\right) N(x, y)\big|\cdot\rho|u|^2(y)\,dy\\
&\leq C\int_\Omega\frac{|u(y)-u(x_s)|}{|y-x_s|^2}\cdot\rho |u|(y)\,dy+C\int_\Omega\frac{\rho|u|^2(y)}{|y-x_s|}dy,
\end{split}
\end{equation}
where we have applied \eqref{CA420} in the last but two line.
Combining \eqref{CA453}, \eqref{CA454}, and \eqref{CA455} implies that 
\begin{equation*} 
|B(x)|\leq C\,\sum_{j=0}^{k-1}\left(\int_\Omega\frac{|u(y)-u(x_j)|}{|y-x_j|^2} \rho |u|(y)\,dy+\int_\Omega\frac{\rho|u|^2 (y)}{|y-x_j|}dy\,\right),
\end{equation*}
which is \eqref{CA448} and finishes the proof.
\end{proof}

\begin{remark} 
Let us make a comparison between commutator arguments on the bounded domain $\Omega$ and the global domain $\mathbb{R}^2$.
Note that the commutator 
$[u,R_{ij}](\rho u)$ in $\mathbb{R}^2$ can be written as a  singular integral: 
\begin{equation}\label{CA458}
[u,R_{ij}](\rho u)(x)\leq C\int_{\mathbb{R}^2}\frac{|u(y)-u(x)|}{|y-x|^2}\cdot\rho |u|(y)\,dy,~~R_{ij}=\partial_i\Delta^{-1}\partial_j.
\end{equation} 
In contrast, the commutator on the bounded domain $\Omega$ consists of two parts. The first part given by $\left[u_i,\,\partial_i\,\Delta^{-1}\,\partial_j\right](\rho u_j)$ corresponds to the inner domain, since
\begin{equation}\label{CA456}
\begin{split}
&\left[u_i,\,\partial_i\,\Delta^{-1}\,\partial_j\right](\rho u_j)(x)\leq C\int_{\Omega}\frac{|u(y)-u(x)|}{|y-x|^2}\cdot\rho |u|(y)\,dy,~~\mathrm{with}\ x\in\Omega.\\
\end{split}
\end{equation}
While the second part given by $B(x)$ corresponds to the boundary, since
\begin{equation}\label{CA457}
\begin{split}
&B(x)\leq C\sum_{j=0}^{k-1}\int_{\Omega}\frac{|u(y)-u(x_j)|}{|y-x_j|^2}\cdot\rho |u|(y)\,dy+\mbox{Remainders},~~\mathrm{with}\ x_j\in\partial\Omega.
\end{split}
\end{equation}
Clearly, the integrals given by \eqref{CA458}--\eqref{CA457} are quite similar and enjoy plenty of common properties. Let us call the left hand side of \eqref{CA456} the inner commutator and each summation term on the right hand side of \eqref{CA457} the boundary commutator.

Let us make a detailed discussion on  \eqref{CA456} and \eqref{CA457}. We mention that the inner commutator in \eqref{CA456} is   analogous to commutators in $\mathbb{R}^2$, since they are generated in the same  process, and $\Delta^{-1}$ all refer to the convolution with the fundamental solution.
Meanwhile, \eqref{CA457} measures the error between commutators on $\Omega$ and $\mathbb{R}^2$  due to the boundary $\partial\Omega$, precisely, each component of $\partial\Omega$ gives rise to an additional boundary commutator in \eqref{CA457} . In summary, the structure of commutators on $\Omega$ is rather clear: The inner part is formulated by the same method as in $\mathbb{R}^2$, and the
existence of boundary commutator reflects the differences between $\Omega$ and $\mathbb{R}^2$, while the number of boundary commutators describes the multi-connectedness of $\Omega$.

\end{remark}

\section{A priori estimates}\label{S5}
\quad In this section, we utilize the arguments in Section \ref{S3} and \ref{S4} to develop a priori estimates of the system \eqref{11}--\eqref{15}. Supposing that $\Omega$ is a $k$-connected bounded smooth domain satisfying \eqref{C101}, and the conditions in Theorem \ref{T1} always holds, we first quote the local well-posedness assertion (see \cite{weko_39341_1,Solonnikov1980}).
\begin{lemma} \label{u21}
Under the same conditions of Theorem \ref{T1}, we assume that the initial values are given by
\begin{equation*}
\rho_0\in H^2 ,\quad \inf_{x\in\Omega}\rho_0(x)>0,\quad u_0\in \tilde{H}^1 \cap H^2 ,\quad m_0=\rho_0u_0,   
\end{equation*}
then there exist a small time $T>0$ and a positive constant $C_0$ both depending on $ \Omega$, $\mu$, $\beta$, $\gamma$, $\|\rho_0\|_{H^2}$,  $\|u_0\|_{H^2},$ and $\|\rho_0^{-1}\|_{L^\infty}$, such that the system \eqref{11}--\eqref{15} has a unique strong solution $(\rho,  u)$ in $\Omega\times(0, T)$ satisfying
\begin{equation*} 
\begin{cases}
\rho\in C\big([0, T]; H^2 \big),~~ \rho_t\in C\big([0, T];H^1 \big), \\
u\in L^2\big(0, T;H^3 \big),~~u_t\in L^2\big(0, T;H^2 \big)\cap H^1\big(0,T;L^2 \big).
\end{cases}
\end{equation*}
\end{lemma}

Let $(\rho, u)$ be the strong solution given above, then we follow the outline in  \cite{FLL,huang2016existence} and extend the local solution to the global one.
The effective viscous flux $F$ and the vorticity $\omega$ are defined in the usual manners:
\begin{equation}\label{303}
 F\triangleq\big(2\mu+\lambda(\rho)\big)\mathrm{div}u-(P-\bar{P}),~~
\omega\triangleq\partial_2u_1-\partial_1u_2.
\end{equation}
Moreover, we introduce the notations
\begin{equation}\label{a}
\begin{split}
A^2(t)\triangleq \int_\Omega \big( \lambda(\rho )(\mathrm{div}u  )^2&+|\nabla u|^2+(\rho+1)^{\gamma-1}(\rho-1)^2 \big)dx+\int_{\partial \Omega}K|u|^2\,dS,\\
&\ B^2(t)\triangleq\int_\Omega\rho|\dot{u}|^2(x,t)\,dx,\\
&R_T\triangleq 1+\sup_{0\leq t\leq T}\|\rho(\cdot,t)\|_{L^\infty }.
\end{split}
\end{equation}

\subsection{A priori estimates (I): uniform estimates}
\quad The goal of this subsection is to establish the uniform upper bound of $\rho$ \eqref{md} and the exponential decay \eqref{zs}, which describe the large time behaviour of the solution $(\rho,u)$. 
To begin with, we first set up the basic energy estimates.
\begin{lemma}\label{3111}
There exists a positive constant $C$ depending only on $  \gamma$, $\mu$, $\rho_0$,  $u_0$, and $\Omega$, such that
\begin{equation}
\begin{split}
&\sup_{0\leq t\leq T}\int_\Omega(\rho|u|^2+\rho^\gamma)dx+\int_0^T\int_\Omega \left((2\mu+\lambda)(\mathrm{div}u  )^2+|\nabla u|^2 \right)dxdt\\
&+\int_0^T\int_{\partial \Omega}K|u|^2\,dSdt\leq C.  \label{300}
\end{split}
\end{equation}
\end{lemma}
\begin{proof}
Since $\Delta u=\nabla \mathrm{div}u+\nabla^\bot\omega$, we can rewrite \eqref{11}$_2$ as
\begin{equation}\label{ml}
\rho\dot{u}=\nabla((2\mu+\lambda(\rho)) \mathrm{div}u)-\nabla P+\mu\nabla^\bot\omega.  
\end{equation}
Multiplying \eqref{ml} by $u$ and integrating over $\Omega$, one deduces from \eqref{11}$_1$ and the boundary conditions \eqref{15} that
\begin{equation*}
\begin{split}
&\frac{d}{dt}\int_\Omega\left(\frac12\rho|u|^2+ \frac{\rho^\gamma}{\gamma-1}\right)\,dx\\
&+\int_\Omega\big((2\mu+\lambda)(\mathrm{div}u  )^2+\mu\omega^2\big)\,dx+\mu\int_{\partial \Omega}K|u|^2\,dS=0,
\end{split}
\end{equation*}
where we have applied the fact $u=(u\cdot n^\bot)\,n^\bot$ due to \eqref{15}. Then  
integrating the above equality over $(0, T)$ yields
\begin{equation*}
\begin{split}
&\sup_{0\leq t\leq T}\int_\Omega(\rho|u|^2+\rho^\gamma)\,dx+\int_0^T\int_\Omega\left((2\mu+\lambda)(\mathrm{div}u  )^2+\mu\omega^2 \right)\,dxdt\\
&+\int_0^T\int_{\partial \Omega}K|u|^2\,dSdt\leq C,  
\end{split}
\end{equation*}
which combined with Lemma \ref{le3.2} shows \eqref{300} directly and finishes the proof.
\end{proof}

Based on the energy estimates \eqref{300}, the  arguments in \cite{flw} lead to the uniform $L^p$-estimates of density, see also related topics in \cite{vaigant1995}. 
\begin{lemma}[\cite{flw}]\label{004}
Let $g_+\triangleq\max\{g,0\}$, then for any $p\geq 2$, there are constants $M\geq 1$ and $C$ depending only on $\mu$, $\alpha$, $\beta$, $\gamma$, $\rho_0$, $u_0$, and $\Omega$, such that
\begin{equation}\label{S327}
\sup_{0\leq t\leq T}\|\rho(\cdot,t)\|_{L^{p} }+\int_0^T\int_\Omega(\rho-M)_+^{p} \,dxdt\leq C,
\end{equation}
\begin{equation}\label{S3277}
\int_0^T\int_\Omega(\rho+1)^{\gamma-1}(\rho-\hat{\rho})^2 dxdt\leq C.
\end{equation}
\end{lemma}
 
Combining Lemmas  \ref{3111} and \ref{004}, we can establish the extra integrability of the momentum $\rho u$. Some modifications are demanded for the multi-connected domain $\Omega$.
\begin{lemma}\label{0001} 
There exists a constant $C $ depending on  $\mu$,  $\beta$, $\gamma$, $\rho_0$, $u_0$, and $\Omega$, such that
\begin{equation}\label{qp90}
\sup_{0\leq t\leq T}\int_\Omega \rho |u|^{2+\nu}\,dx\leq C ,
\end{equation} 
where we set
\begin{equation}\label{nu1q}
\nu\triangleq R_T^{-\beta/2}\nu_{0},
\end{equation}
for some suitably small generic constant $\nu_0\in (0,(\gamma-1)/(\gamma+1))$ depending only on $\mu$, $\gamma$, $\rho_0$, $u_0$, and $\Omega$.
\end{lemma}
\begin{proof}
First, we deduce from \eqref{300} and \eqref{S3277} that for $A(t)$ defined in \eqref{a},
\begin{equation}\label{w2}
\int_0^T A^2(t)\,dt\leq C.
\end{equation}
Following the proof of \cite[Lemma 3.7]{huang2016existence}, we multiply \eqref{34} by $ |u|^\nu u$ and integrate over $\Omega$ by parts to deduce that
\begin{equation}\label{w3}
\begin{split}
&\frac{d}{dt}\int_\Omega\rho |u|^{2+\nu}dx+ \int_\Omega|u|^\nu\big((2\mu+\lambda)(\mathrm{div}u)^2+\mu\omega^2\big)\,dx
+\mu\int_{\partial \Omega}K|u|^{2+\nu}dS\\
&\leq 4 \nu\int_\Omega(2\mu+\lambda)|\mathrm{div}u|\,|u|^{\nu}|\nabla u|\,dx+4\nu\mu\int_\Omega|u|^{\nu}|\omega|\,|\nabla u|\,dx\\
&\quad+C\int_\Omega|\rho^\gamma-\hat{\rho}^\gamma|\,|u|^\nu|\nabla u|\,dx\\
&\leq \frac{1}{2}\int_\Omega|u|^\nu\big((2\mu+\lambda)(\mathrm{div}u)^2+\mu\omega^2\big)\,dx+ C\,\nu_0^2\int_\Omega|u|^\nu|\nabla u|^2\,dx\\
&\quad+C\int_\Omega(\rho^{\gamma-1}+1)|\rho-\hat\rho|\,|u|^\nu|\nabla u|\,dx,\\
\end{split}
\end{equation}
where in the last line we have applied \eqref{nu1q}. In view of Lemma \ref{a004}, we declare that
\begin{equation}\label{C401}
C\,\nu_0^2\int_\Omega|u|^{\nu}|\nabla u|^2\,dx\leq  \frac{\mu}{4}\int_\Omega|u|^{\nu}\big((\mathrm{div}u)^2+\omega^2)\,dx+C\int_\Omega\rho|u|^{2+\nu}dx,
\end{equation}
by setting $\nu_0^2\leq\mu/(4CC_1)$ with $C_1$ given by Lemma \ref{a004}. In addition, 
with the help of \eqref{S327} and Sobolve's inequality, one gets for $p=4/(2-\nu)$ and $\theta=\nu/(2+\nu)$,
\begin{equation}\label{w4}
\begin{split}
&\int_\Omega\rho|u|^{2+\nu}\,dx \leq \|\rho|u|^{2+\nu}\|_{L^1 }^{\theta}\|\rho\|_{L^{p} }^{1-\theta}\|u\|_{L^4 }^2\leq  C\left(\|\rho|u|^{2+\nu}\|_{L^1 }+1\right)\cdot A^2.
\end{split}
\end{equation}
Moreover, for $p'$ satisfying $1/p'=(1-\nu)/2-1/(\gamma+1)$, it holds that
\begin{equation}\label{w1}
\begin{split}
&\int_\Omega(\rho^{\gamma-1}+1)|\rho-\hat\rho|\,|u|^\nu|\nabla u|\,dx\\
&\leq C\int_\Omega\left((\rho-M)_+^{\gamma-1}+1\right)|\rho-\hat\rho|\,|u|^\nu|\nabla u|\,dx\\
&\leq C\left(\int_\Omega(\rho-M)_+^{p'(\gamma-1)}dx+\int_\Omega\left(|\rho-\hat\rho|^{\gamma+1}+|\rho-\hat\rho|^\frac{2}{1-\nu}\right)dx+\int_\Omega|\nabla u|^2\,dx\right)\\
&\leq CA^2+C\int_\Omega(\rho-M)_+^{p'(\gamma-1)}dx,
\end{split}
\end{equation}
where $p'(\gamma-1)>2$ and the third line is due to H\"{o}lder's and Poincar\'{e}'s inequalities, while the last line is due to the following fact:
$$|\rho-\hat\rho|^{\frac{2}{1-\nu}}\leq(\rho+1)^{\frac{2\nu}{1-\nu}}(\rho-\hat\rho)^2\leq C(\rho+1)^{\gamma-1}(\rho-\hat\rho)^2,$$
since $\nu<\nu_0<(\gamma-1)/(\gamma+1)$.

Therefore, by submitting \eqref{C401}, \eqref{w4}, and \eqref{w1} into \eqref{w3}, we apply \eqref{w2} and Gr\"{o}nwall's inequality to obtain \eqref{qp90}. The proof of Lemma \ref{0001} is finished.
\end{proof}

Then let us deal with the $L^p$-estimates of $\nabla u$, which will be repeatedly used in later discussions, see also \cite[Proposition 3.5]{flw}. 
\begin{lemma}\label{005} 
For any $p>2 $ and $\varepsilon>0,$ we denote $\kappa=1-2/p$, then there exists some positive constant $C(p, \varepsilon)$ depending only on  $p$, $\varepsilon,$ $\mu$, $\gamma$, $\beta$,   and $\Omega$, such that
\begin{equation}\label{qp31}
\|\nabla u\|_{L^{p} }\leq C(p, \varepsilon)\,R_T^{{\kappa}/{2}+\varepsilon} (A+1)^{1-\kappa}  (A+B+1)^{\kappa}.
\end{equation}
In particular, when $\kappa\leq\varepsilon<(\gamma+1)^{-1}$ and $\gamma<2\beta$, we have
\begin{equation}\label{S413}
\|\nabla u\|_{L^{p} }\leq C(p, \varepsilon)\, R_T^{\varepsilon}\,A^{1-\kappa} (A+B+1)^{\kappa}.
\end{equation}
\end{lemma}
\begin{proof}
First, by the definition of $F$ in \eqref{303},  \eqref{ml} is equivalent to 
\begin{equation}\label{34}
\rho\dot{u}=\nabla F+\mu\nabla^\bot\omega,   
\end{equation} which together with  the boundary condition \eqref{15} yields that $\omega$ solves a Dirichlet boundary problem
\begin{equation}\label{ty}
\begin{cases}
  \mu\Delta\omega=\nabla^\bot\cdot(\rho\dot{u})& \mbox{ in } \Omega, \\
   \omega=-Ku\cdot n^\bot& \mbox{ on } \partial\Omega.  \end{cases}
\end{equation}
Then, standard $L^p$-estimates of elliptic system \eqref{ty} (see\cite{GT}, \cite[Lemma 4.27]{novotny2004introduction}) along with \eqref{34} imply that for any integer $k\geq 0$ and  $p\in(1, \infty)$,
\begin{equation}
\|\nabla F\|_{W^{k, p} }+\|\nabla\omega\|_{W^{k, p} }\leq C(p, k)\,(\|\rho\dot{u}\|_{W^{k, p} }+\|\nabla u\|_{W^{k,p} }),  \label{336}
\end{equation}
which together with \eqref{25}, \eqref{303},  \eqref{S327}, and \eqref{S3277} implies that
\begin{equation} \label{311}
\begin{split}
&\|F\|_{H^1 }+\|\omega\|_{H^1 }\\
&\leq C\left(\|\nabla F\|_{L^2 }+\|\nabla\omega\|_{L^2 }+A+|\hat F|\right)\\
&\le C\left(\|\rho\dot u\|_{L^2 }+A+\|\nabla u\|_{L^2 }\|\rho^\beta-\hat\rho^\beta\|_{L^2 }+\|\rho^\gamma-\hat\rho^\gamma\|_{L^1 }\right)\\
&\leq C\left(R_T^{1/2}B+A\right). 
\end{split}
\end{equation}
Thus we apply \eqref{g4} and \eqref{S327} to derive that for $p>2$, $\varepsilon>0$,  $q=2\big((1+\varepsilon)p-2\big)/p\varepsilon$,
\begin{equation*} 
\begin{split}
\|\nabla u\|_{L^{p} }
&\leq C(p)\big(\|\mathrm{div}u\|_{L^{p} }+\|\omega\|_{L^{p}  }+A \big) \\
&\leq C(p)\left(\left\|(2\mu+\lambda)^{-1}F \right\|_{L^{p} }+ \|\omega\|_{L^{p} }+A+1\right) \\ 
&\leq C(p)\left(\left\|(2\mu+\lambda)^{-1}F\right\|_{L^{2} }^{1-\kappa-\varepsilon}\|F\|_{L^{q} }^{\kappa +\varepsilon}+\|\omega\|_{L^{2}}^{1-\kappa}\|\omega\|_{H^1}^{\kappa }+A+1\right)
\\&\leq C(p, \varepsilon)\,(A+1)^{1-\kappa -\varepsilon}\|F\|_{L^{2} }^{\varepsilon} \|F\|_{H^1 }^{\kappa }+C(p)\left(A^{1-\kappa }\| \omega\|_{H^1}^{\kappa }+ A+1\right)\\
&\leq C(p, \varepsilon)\,R_T^{\frac{\beta\varepsilon}{2} }(A+1)^{1-\kappa }( \|F\|_{H^1}+\|\omega\|_{H^1})^{\kappa }+C(p)(A+1),
\end{split}
\end{equation*}  
which together with \eqref{311} gives \eqref{qp31}.

In particular, when $\kappa <\varepsilon<(\gamma+1)^{-1}$ and $\gamma<2\beta$, we derive that
\begin{equation}\label{S421}
\begin{split}
&\|F\|_{L^2}\le CR_T^\beta\|\mathrm{div} u\|_{L^2}+\|P-1\|_{L^2}\le CR_T^{\beta+\gamma}A,\\
\big\|(2\mu&+\lambda)^{-1}F\big\|_{L^2} \leq C\bigg(\|\mathrm{div}u\|_{L^2}+\left\|(2\mu+\lambda)^{-1}(P-\hat P)\right\|_{L^{2}} \bigg)\leq CA,
\end{split}
\end{equation}
where the last line is due to the fact $p(\gamma-\beta)-2<(\gamma-1)$ and
\begin{equation}\label{CC402}
\begin{split}
\left\|(2\mu+\lambda)^{-1}(P-\hat P)\right\|_{L^{p} }^p\leq C\int_\Omega\big(\rho+1\big)^{p(\gamma-\beta)-2}
\big(\rho-1\big)^2\,dx\leq CA^2.
\end{split}
\end{equation}
Consequently, by making use of \eqref{311}--\eqref{CC402}, we infer that
\begin{equation*}
\begin{split}
&\|\nabla u\|_{L^{p}}\\
&\leq C(p)\left(\left\|(2\mu+\lambda)^{-1}F \right\|_{L^{p} }+\left\|(2\mu+\lambda)^{-1}(P-\hat P)\right\|_{L^{p} }+\|\omega\|_{L^{p} }+A\right)\\ 
&\leq C(p)\left(\left\|(2\mu+\lambda)^{-1}F\right\|_{L^{2} }^{1-\kappa -\varepsilon}\|F\|_{L^{2} }^{\varepsilon} \|F\|_{H^1 }^{\kappa }+A^{1-\kappa }+ \|\omega\|_{L^{2}}^{1-\kappa }\|\omega\|_{H^1}^{\kappa } +A\right)\\
&\leq C(p,\varepsilon)\,A^{1-\kappa -\varepsilon}\|F\|_{L^{2} }^{\varepsilon} \|F\|_{H^1 }^{\kappa }+C(p)A^{1-\kappa}\left( \| \omega\|_{H^1}^{\kappa }+ A^\kappa+1\right)\\
&\leq C(p, \varepsilon)\,R_T^{(\beta+\gamma+1)\varepsilon} A^{1-\kappa }(A+B+1)^{\kappa }.
\end{split}
\end{equation*}
which is \eqref{S413} and finishes the proof of Lemma \ref{005}.
\end{proof}
 
The next lemma provides further energy estimates concerning the upper bound of $\mathrm{log}(e+\|\nabla u\|_{L^2})$. The proof remains the same as that of \cite[Proposition 3.6]{flw} for simply connected domains.
\begin{lemma}\label{LL56}
For any $\varepsilon\in(0, 1/2)$, there is a constant $C(\varepsilon)$ depending only on $\varepsilon$, $\mu$, $\beta$, $\gamma$, $\rho_0$, $u_0$, and $\Omega$, such that
\begin{equation}\label{qp40}
\sup_{0\leq t\leq T}\mathrm{log}\left(e +A^2(t)\right) +\int^T_0 \frac{B^2(t)}{1+A^2(t)}dt\leq C(\varepsilon)R_T^{1 +\varepsilon}.
\end{equation}
\end{lemma}
 
Now let us turn to the crux of a priori estimates. 
We will make use of the commutator arguments developed in Section \ref{S4} together with the   energy estimates in Lemma \ref{0001}--\ref{LL56} to provide proper controls on the effective viscous flux $F$.
\begin{lemma}
Let $F(x)$ be the effective viscous flux defined in \eqref{qp11} and $\varepsilon\in(0,1/2)$, then there is a constant $C$ depending only on $\varepsilon$, $\mu$, $\beta$, $\gamma$, $\rho_0$, $u_0$, and $\Omega$, such that for ${\beta'}\triangleq 1+\beta/{4}+2\varepsilon$ and $\tau,s>0$, it holds that
\begin{equation}\label{CA460}
\int_{\tau}^{\tau+s}-F(x(t),t)\, dt\leq C
\begin{cases}
R_T^{1+\varepsilon}s+R_T^{\beta'},~~&\mbox{if $\gamma<2\beta$,}\\
 R_T^{\beta'}\big(s+1\big),~~&\mbox{if $\gamma\geq 2\beta$.}
\end{cases}
\end{equation}
where $x(t)$ is the flow line determined by $x(t)'=u(x(t),t)$.
\end{lemma}
\begin{proof}

According to Lemma \ref{LL49}, it holds that
\begin{equation*}
\begin{split}
F(x)
&=\frac{\mathrm{d}}{\mathrm{d}t}\Delta^{-1}\mathrm{div}(\rho u)-\left[u_i,\,\partial_i\,\Delta^{-1}\,\partial_j\right](\rho u_j)+B(x)+R(x),~~\forall x\in\Omega,\\
\end{split}
\end{equation*}

\textit{Step 1.} We first handle $\Delta^{-1}\mathrm{div}(\rho u)$. According to the definition, we argue that
\begin{equation*}
\begin{split}
\left|\Delta^{-1}\mathrm{div}(\rho u)\right|=\left|\int_\Omega\partial_{y_i}N(x,y)\cdot \rho u_i(y)\, dy\right|\leq\int_\Omega|\nabla N(x,y)|\cdot\rho |u|(y)\,dy.
\end{split}
\end{equation*}
Thus with the help of \eqref{CA450} and \eqref{qp90}, we obtain that for $q=(2+\nu)/(1+\nu)<2$,
\begin{equation*}
\begin{split}
&\int_\Omega|\nabla N(x,y)|\cdot\rho |u|(y)\,dy\\
&\leq C\int_\Omega|x-y|^{-1}\cdot\rho|u|(y)\,dy\\
&\leq C\left(\int_\Omega|x-y|^{-q}dy\right)^{1/q}\left(\int_\Omega \rho^{2+\nu}|u|^{2+\nu}dy\right)^{\frac{1}{2+\nu}}\\
&\leq C\cdot\nu^{-1/q}R_T^{1/q}\left(\int_\Omega \rho |u|^{2+\nu}dy\right)^{\frac{1}{2+\nu}}\leq CR_T^{\beta'}.
\end{split}
\end{equation*}
Consequently, we check that
\begin{equation}\label{CA461}
\int_{\tau}^{\tau+s}\frac{d}{dt}\Delta^{-1}\mathrm{div}(\rho u)\big(x(t),t\big)\,dt\leq 2\sup_{0\leq t\leq T}\|\Delta^{-1}\mathrm{div}(\rho u)\|_{L^\infty}\leq CR_T^{\beta'}.
\end{equation}
  
\textit{Step 2.} Then we consider the commutator
$\left[u_i,\,\partial_i\,\Delta^{-1}\,\partial_j\right](\rho u_j)-B(x)$. Note that Lemma \ref{qp08} yields that
\begin{equation*}
\begin{split}
&\left|\left[u_i,\,\partial_i\,\Delta^{-1}\,\partial_j\right](\rho u_j)\right|+\big|B(x)\big|\\
&\leq
C\,\left(\sup_{x\in\overline{\Omega}} \int_\Omega\frac{|u(y)-u(x)|}{|y-x|^2}\,\rho|u| (y)\,dy+  \sup_{x\in\overline{\Omega}}\int_\Omega\frac{\rho|u|^2(y)}{|y-x|}\,dy\right).
\end{split}
\end{equation*}

Let us control the difference quotient in the above inequality. If we still denote that $\kappa=1-2/p\leq\varepsilon$, then the Morrey's inequality (Theorem 5 of \cite[Chapter 5]{2010Partial}) ensures that for any $x\in \overline{\Omega}, $
\begin{equation}\label{qp30}
\begin{split}
\int_\Omega \frac{|u(x)-u(y)|}{|x-y|^2} \rho|u| (y)\,dy &\leq C(p) \int_\Omega\frac{\|\nabla u\|_{L^p}}{|x-y|^{2-\kappa}} \cdot\rho|u| (y)\,dy.
\end{split}
\end{equation}
Taking $\delta>0$ and $\mathbf{Q}<\kappa/2$, which will be determined later, on the one hand,  we deduce  
\begin{equation}\label{S435}
\begin{split}
&\int_{|x-y|<2\delta}\frac{\|\nabla u\|_{L^p}}{|x-y|^{2-\kappa}} \cdot\rho|u| (y)\,dy\\
&\leq C \|\nabla u\|_{L^p}R_T\cdot
\left(\int_{|y| <2\delta}|y|^{(\kappa-2)/(1-\mathbf{Q})} dy\right)^{1-\mathbf{Q}}
\| u\|_{L^{1/\mathbf{Q}}}\\
&\leq C \|\nabla u\|_{L^p}R_T\cdot
\delta^{\kappa-2\mathbf{Q}} \left(\mathbf{Q}^{-1/2}\| u\|_{H^1}\right)\\
&\leq C \|\nabla u\|_{L^p}R_T\cdot \mathbf{Q}^{-1/2} \left(A_1
\delta^{\kappa-2\mathbf{Q}}\right).
\end{split}
\end{equation}
where in the third line, we have applied the Poincar\'{e}-type inequality
\begin{equation}\label{25}
\|u\|_{L^{1/\mathbf{Q}}}\leq C\mathbf{Q}^{-1/2}\|u\|^{2\mathbf{Q}}_{L^2}\|u\|^{1-2\mathbf{Q}}_{H^1 },~~\mbox{for $u\in H^1$ with $u\cdot n\big|_{\partial\Omega}=0$.}
\end{equation}
On the other hand, for $\nu=R_T^{-\frac{\beta}{2}}\nu_{0}$ as in \eqref{nu1q}, we use Lemma \ref{0001} to derive
\begin{equation}\label{S438}
\begin{split}
&\int_{|x-y| >\delta}\frac{\|\nabla u\|_{L^p}}{|x-y|^{2-\kappa}}\cdot\rho|u|(y)\,dy\\
&\leq
C\|\nabla u\|_{L^p}\cdot\left(\int_{|y| >\delta} |y| ^{\frac{2+\nu}{1+\nu}\cdot(\kappa-2)} dy\right)^{\frac{1+\nu}{2+\nu}}
\left(\int_\Omega\rho^{2+\nu}|u|^{2+\nu}\,dx\right)^{\frac{1}{2+\nu}}\\
&\leq C\|\nabla u\|_{L^p}R_T\cdot
\delta^{\kappa-\frac{2}{2+\nu}}.
\end{split}
\end{equation}
Then we choose $\delta>0$ to guarantee
$$\delta^{\kappa-\frac{2}{2+\nu}}=A_1^{1-\kappa},~~\mbox{which also implies that}~~A_1\delta^{\kappa-2\mathbf{Q}} =A_1^{1-\kappa},$$
by setting $
2\mathbf{Q}=\frac{\nu}{2+\nu}\cdot\frac{\kappa}{1-\kappa} \in \left(0,\kappa\right).$
Thus we collect \eqref{S435}--\eqref{S438} to infer that
\begin{equation}\label{S440}
\begin{split}
&\int_\Omega\frac{\|\nabla u\|_{L^p}}{|x-y|^{2-\kappa}} \cdot\rho|u| (y)\,dy\\
&\leq\left(\int_{|x-y|<2\delta}+\int_{|x-y|>\delta}\right)\frac{\|\nabla u\|_{L^p}}{|x-y|^{2-\kappa}} \cdot\rho|u| (y)\,dy \\
&\leq C\|\nabla u\|_{L^p}R_T\cdot \mathbf{Q}^{-\frac{1}{2}} A_1^{1-\kappa}\leq C(p)\,\|\nabla u\|_{L^p}R_T^{1+\frac{\beta}{4}}\cdot A_1^{1-\kappa} ,
\end{split}
\end{equation}
where in the last line we have used $\mathbf{Q}^{-1/2}\leq C(p)\nu^{-1/2}\leq C(p)R_T^{\beta/4} $ due to the definition of $\mathbf{Q}$. Note that by virtue of \eqref{qp31} and \eqref{S413}, we argue that
\begin{equation*}
\begin{cases}
\|\nabla u\|_{L^p}
\leq C\,R_T^{\varepsilon}\cdot A^{1-\kappa}(A+B+1)^{\kappa},~~&\mathrm{if}\ \gamma<2\beta,\\
\|\nabla u\|_{L^{p}}\leq C\,R_T^{\varepsilon}\cdot (A+1)^{1-\kappa}(A+B+1)^{\kappa},~~&\mathrm{if}\ \gamma\geq 2\beta,
\end{cases}
\end{equation*}
which together with \eqref{qp30} and \eqref{S440}  gives that
\begin{equation*}
\begin{split}
&\int_\Omega\frac{|u(x)-u(y)|}{|x-y|^2} \rho|u|(y)\, dy\leq C
\begin{cases}
1+R_T^{\beta'} A^2+\frac{B^2}{1+A^2},~~&\mathrm{if}\ \gamma<2\beta,\\
R_T^{\beta'}\,(1+A^2)+\frac{B^2}{1+A^2}~~&\mathrm{if}\ \gamma\geq 2\beta.
\end{cases}\\
\end{split}
\end{equation*}
with ${\beta'}= 1+\beta/{4}+2\varepsilon$ and $\kappa<\varepsilon/(4+\beta).$ Thus we integrate the above inequality over $[\tau,\tau+s]$ and apply \eqref{300} along with \eqref{qp40} to deduce that   
\begin{equation}\label{CA459}
\begin{split}
&\int_{\tau}^{\tau+s}\sup_{x\in\overline\Omega}\left(\int_\Omega\frac{|u(x)-u(y)|}{|x-y|^2}\rho|u|(y)\, dy\right) dt\leq
C\begin{cases}
s+R_T^{\beta'}
~~&\mbox{if $\gamma<2\beta$},\\
R_T^{\beta'}\big(s
+1\big)~~&\mbox{if $\gamma\geq 2\beta$}.
\end{cases}
\end{split}
\end{equation}
  
\textit{Step 3.} Finally we turn to the lower order term $R(x)$. Recalling that
$$R(x)=l^{-1}\int_{\partial \Omega}F(y)\,dS_y
-\mu\int_{\partial \Omega}N(x, y)\cdot(n^\bot\cdot\nabla(Ku\cdot n^\bot))\,dS_y. $$
In view of \eqref{311}, the first term is directly bounded by
\begin{equation}\label{CA462}
\left|\int_{\partial\Omega} F(y)\,dS_y\right|\,dS\leq \int_{\partial\Omega}\big|F(y)\big|\, dS_y\leq C\big\|F\big\|_{H^1}\leq CR_T^{1/2}\big(1+A+B\big).
\end{equation}
While for the second term, we set $p=4$ in \eqref{qp31} and  apply \eqref{CA450} to declare that
\begin{equation}\label{CA463}
\begin{split}
&\int_{\partial\Omega}N(x,y)\cdot\big(n^\bot\cdot\nabla(Ku\cdot n^\bot)\big)\, dS_y\\
&\leq C \int_{\Omega}\big(|N(x,y)|+|\nabla N(x,y)|\big)\cdot\big(|\nabla u|+|u|\big)\, dy,\\
&\leq C\big(\|N(x,\cdot)\|_{L^{4/3}}+\|\nabla N(x,\cdot)\|_{L^{4/3}}\big)\cdot\|\nabla u\|_{L^4}\\
&\leq CR_T^{1/2+\varepsilon}\big(1+A+B\big),
\end{split}
\end{equation}
where we have taken advantage of the next observation in \cite{caili01},
$$\left|\int_{\partial\Omega}u\cdot(n^\bot\cdot\nabla v)\,dS\right|\leq C\int_\Omega|\nabla u|\cdot|\nabla v|\,dx.$$
Consequently, we integrate \eqref{CA462}--\eqref{CA463} over $[\tau,\tau+s]$ and utilize \eqref{300} together with \eqref{qp40} to infer that
\begin{equation}\label{CA464}
\begin{split}
&\int_{\tau}^{\tau+s}\big|R\big(x(t),t\big)\big|\,dt\\
&\leq
C\int_{\tau}^{\tau+s}R_T^{1/2}\big(1+A+B\big)\,dt\\
&\leq C\int_{\tau}^{\tau+s}\big(R_T^{1+\varepsilon}(1+A^2)+\frac{B^2}{1+A^2}\big)\,dt\\
&\leq CR_T^{1+\varepsilon}(s+1).
\end{split}
\end{equation}
Combining \eqref{CA461}, \eqref{CA459}, and \eqref{CA464} leads to \eqref{CA460}. We therefore finish the proof.
\end{proof}

Now we can obtain the uniform upper bound of the density $\rho$.  
The proof is based on the next Zlotnik's inequality which can be found in \cite{zlt}.
\begin{lemma}\label{L26}
Let $f, g\in W^{1, 1}(0, T)$ and $h\in C(\mathbb{R})$ satisfy
\begin{equation*}
f'(t)=g'(t)+h\big(f(t)\big)\ \mathrm{in}\ [0, T],~~ \lim_{x\rightarrow\infty}h(x)=-\infty.
\end{equation*}
If for any $0\leq t_1<t_2\leq T$, there are positive constants $N_1$ and $N_2$ such that
\begin{equation*}
\begin{split}
g(t_2)&-g(t_1)\leq N_1(t_2-t_1)+N_2,
\end{split}
\end{equation*}
then it holds that
\begin{equation*}
\sup_{0\leq t\leq T}f(t)\leq \max{\{f(0),\,\xi\}}+ N_2,
\end{equation*}
where $\xi$ is a constant such that
$g(x)\leq-N_1$ for any $|x|\geq\xi$.
\end{lemma}

Lemma \ref{L26} provides the following uniform estimates.
\begin{lemma} 
Under the condition of Theorem \ref{T1}, there exists some positive constant $\mathbf{C}$ depending only on  $\mu,  \beta,$ $ \gamma$, $\rho_0,$ $u_0$, and $\Omega$, such that
\begin{equation}\label{3103}
\sup_{0\leq t\leq T}(\|\rho\|_{L^{\infty}}+\|u\|_{H^{1}})+\int_{0}^{T}\|\sqrt\rho\dot{u}\|_{L^2}^2 dt\leq \mathbf{C} .
\end{equation}
\end{lemma}
\begin{proof}
We introduce the function $ \theta(\rho)=\big(\beta^{-1}\rho^\beta+2\mu\log\rho\big),$ then in view of
\eqref{11}, \eqref{ou1r}, and \eqref{303}, it is valid that
\begin{equation}\label{gh}
\frac{\mathrm{d}}{\mathrm{d}t}\theta(\rho)+(P-1)
=-F.
\end{equation}
When $\gamma<2\beta$, integrating \eqref{gh} with respect to $t$ and applying the Zlotnik's inequality in Lemma \ref{L26} along with \eqref{CA460}, one obtain that
\begin{equation}\label{S533}
R_T^\beta\leq CR_T^{\max\{\beta',\,(1+\varepsilon)\beta/\gamma\}},
\end{equation}
with $\beta'=1+\beta/4+2\varepsilon.$
Let us set $\varepsilon<\min\{(3\beta-4)/8,\, \gamma-1\}$, then \eqref{S533} and the fact $\beta>4/3$, $\gamma>1$ implies that
\begin{equation*}
\sup_{0\leq t\leq T}\|\rho\|_{L^\infty}\leq \mathbf{C}.
\end{equation*}
When $\gamma\geq 2\beta$, we still integrate \eqref{gh} with respect to $t$ and invoke the Zlotnik's inequality together with \eqref{CA460} to derive that
\begin{equation}\label{CC414}
R_T^\beta\leq CR_T^{\max\{\beta\cdot\beta'/\gamma,\,\beta'\}}.
\end{equation}
Consequently, we set $\varepsilon<(3\beta-4)/8$, then \eqref{CC414} makes sure
\begin{equation*}
\sup_{0\leq t\leq T}\|\rho\|_{L^\infty}\leq\mathbf{C},
\end{equation*}
which together with \eqref{S533} and \eqref{qp40}  gives \eqref{3103}. The proof is therefore finished.
\end{proof}
  
At last, we mention that in view of Theorem \ref{L11}, when $K\geq0$ is not identically $0$ on $\partial\Omega$, the stationary system \eqref{C102}--\eqref{C103} is uniquely solved by the trivial solution $(\hat{\rho},0)$.
Thus according to the arguments given by \cite[Proposition 4.2 \& 4.3]{flw}, we declare that the solution will converge to the steady state $(\hat{\rho},0)$ in the exponential rate, which also gives \eqref{zs}.
  
\begin{lemma} 
Under the conditions of Theorem \ref{T1}, for any $ T\in[0,\infty)$ and $ p\in[1,\infty)$, there are constants $\mathbf{C}'$ and $\alpha$ depending only on $p$, $\mu$, $\beta$, $\gamma$,  $\rho_0$, $u_0$, and $\Omega$ such that
\begin{equation}\label{CC416}
\begin{split}
&\sup_{0\leq t\leq T} e^{\alpha t}\big(\|\rho-\hat\rho\|_{L^p}+\|\nabla u\|_{L^p}\big)\leq\mathbf{C}',\\
&\sup_{0\leq t\leq T}\sigma\|\sqrt\rho \dot u\|_{L^2}^2+ \int^T_0 \sigma\| \nabla\dot{u}\| _{L^2}^2dt \leq\mathbf{C}',
\end{split}
\end{equation}
where the weight is defined by $\sigma(t)\triangleq\min\{1,t\}$.
\end{lemma}

\subsection{A priori estimates (II):  higher order estimates}
\quad  This subsection concerns the non-uniform higher order estimates, and the proof is rather routine. Therefore we just sketch the main step, and the complete details can be found in \cite{flw,huang2016existence}.  

\begin{lemma}\label{qou10}
Under the conditions of Theorem \ref{T1}, for any $q > 2$,  there is a constant $\tilde{\mathbf{C}}$ depending only on $\Omega$,  $T$,  $q$,  $\mu$,  $\gamma$,  $\beta$,  $\|u_0\|_{H^1},$ and $\|\rho_0\|_{W^{1, q}}$, such that
\begin{equation}\label{434}\ba
&\sup_{0\leq t\leq T}\big(\|\rho\|_{W^{1, q}} +t\|u\|^2_{H^2}\big) +\int^T_0\big(\|\nabla^2 u\|_{L^q}^{1+\frac{1}{q}}+t\|\nabla^2 u\|_{L^q}^2+t\|u_t\|^2_{H^1}\big)dt \leq \tilde{\mathbf{C}}.  \ea
\end{equation}
In particular, if $\rho_0$ is strictly positive, we declare that
\begin{equation}\label{puw2}  
\sup_{0\leq t\leq T}\|\rho^{-1}(\cdot,t)\|_{L^\infty}\leq \tilde{\mathbf{C}}\,\|\rho_0^{-1}\|_{L^\infty}.
\end{equation}
\end{lemma}
\begin{proof}
We first consider the following elliptic system on $\Omega$,
\begin{equation*}
\begin{cases}
\Delta{u}=\nabla\mathrm{div}u+\nabla^\bot\omega,\ \mathrm{in}\ \Omega,\\
u\cdot n=0,\ \mathrm{curl}u=-Ku\cdot n^\bot\ \mathrm{on}\ \partial\Omega.
\end{cases}
\end{equation*}
The $L^p$ estimates (see \cite{ADN}) of the above system guarantee that, for $p\in(1,\infty)$,
\begin{equation}\label{CA465}
\|u\|_{W^{k,p}}\leq C\big(\|\mathrm{div}u\|_{W^{k-1,p}}+\|\mathrm{curl}u\|_{W^{k-1,p}}+\|u\|_{L^p}\big).
\end{equation}

Next let us apply the arguments in \cite{huang2016existence} and set
$\Phi_i\triangleq(2\mu + \lambda)\,\partial_i\rho$. In view of $\eqref{11}_1$, we argue that $\Phi_i$ satisfies
\begin{equation}
\partial_t\Phi_i + \mathrm{div}(u\Phi_i) + \rho\,\partial_iP = -(2\mu + \lambda)\nabla\rho\cdot \partial_iu -\rho\, \partial_iF.  \label{435}
\end{equation}
Multiplying \eqref{435} by $|\Phi|^{q-2}\Phi_i$ and integrating over $\Omega$,  we apply \eqref{3103} to obtain that
\begin{equation}
\frac{d}{dt}\|\Phi\|_{L^q}\leq C(1+\|\nabla u\|_{L^\infty})\,\|\Phi\|_{L^q}+ C\,\|\nabla F\|_{L^q}.  \label{436}
\end{equation}

Then the Beale-Kato-Majda type inequality (see  \cite{beale1984remarks,caili01}) provides that
\begin{equation}\notag
\|\nabla u\|_{L^\infty}\leq C(\|\mathrm{div}u\|_{L^\infty}+\|\mathrm{curl}u\|_{L^\infty})\,\mathrm{log}(e+\|\nabla^2u\|_{L^q})+C\|\nabla u\|_{L^2}+C,
\end{equation}
which together with \eqref{336}, \eqref{CC416}, and \eqref{CA465} implies that
\begin{equation}
\begin{split}
\|\nabla u\|_{L^\infty} 
\leq&  C\,(1+\|\rho\dot{u}\|_{L^q})\,\mathrm{log}(e+\|\nabla\rho\|_{L^q}+\|\rho\dot{u} \|_{L^q})+C.  \label{438}
\end{split}
\end{equation}
Substituting \eqref{438} into \eqref{436}, and utilizing \eqref{336} lead to
\begin{equation*}
\frac{d}{dt}\mathrm{log}(e+\|\Phi\|_{L^q})\leq  C\,(1+\|\rho\dot{u}\|_{L^q})\,\mathrm{log}(e+\|\Psi\|_{L^q})+ C\,\|\rho\dot{u}\|^{1+1/q}_{L^q}.
\end{equation*}
Thus we apply \eqref{3103}, \eqref{CC416}, and the Gr\"{o}nwall's inequality to infer that
\begin{equation}
\sup_{0\leq t\leq T}\| \rho\|_{W^{1,q}}\leq \tilde{\mathbf{C}}.  \label{440}
\end{equation}

Now the interpolation arguments along with  \eqref{336}, \eqref{CC416}, \eqref{CA465}, and \eqref{440} gives
\begin{equation*}
\sup_{0\leq t\leq T}t\|\nabla^2u\|^2_{L^2}+\int_0^T\left(\|\nabla^2u\|_{L^q}^{1+1/q}+t\|\nabla^2u\|_{L^q}^2+t\|u_t\|_{H^1}^2\right)dt\leq \tilde{\mathbf{C}}.  
\end{equation*}
which combined with \eqref{440} yields \eqref{434}. In addition, when $\rho_0$ is strictly away from vacuum, we obtain from \eqref{11}$_1$ that
\begin{equation*}
\frac{\mathrm{d}}{\mathrm{d}t}\log \rho =\mathrm{div}u.
\end{equation*}
Integrating the above equation along the flow line and utilizing \eqref{434} and \eqref{CA465} lead  to \eqref{puw2}.
The proof is therefore completed.
\end{proof}

\section{Proof of Theorem \ref{T1}}\label{S6}
\quad The argument is routine and we just sketch it, see \cite{vaigant1995,huang2016existence} for complete details.
 
\textit{Proof of Theorem \ref{T1}.}  
Let $(\rho_0, m_0)$ be the initial values of the system \eqref{11}--\eqref{15} satisfying \eqref{18}. Then the standard approximation theory  (see \cite{2010Partial}) ensures that we can find
a sequence of $C^\infty$ functions $(\tilde{\rho}_0^\delta,\,\tilde{u}_0^\delta)$ satisfying
\begin{equation*} 
\lim_{\delta\rightarrow 0}\big(\|\tilde{\rho}_0^\delta-\rho_0\|_{W^{1,q}}+\|\tilde{u}_0^\delta-u_0\|_{H^1}+\|\mathrm{curl}(\tilde{u}_0^\delta-u_0)\|_{H^{1/2}(\partial\Omega)}\big)=0.
\end{equation*}
Suppose that $u^\delta_0$ is the unique smooth solution to the following elliptic system,
\begin{equation}\label{52}
\begin{cases}
\Delta u^{\delta}_0=\Delta\tilde{u}^\delta_0   \mbox{ in } \Omega, \\
u^{\delta}_0\cdot n=0, \ \mathrm{curl}u^{\delta}_0=-Ku^{\delta}_0\cdot n^\bot \mbox{ on }\partial \Omega. \\
\end{cases}
\end{equation} 
We define $ {\rho}_0^\delta=\tilde\rho^\delta_0+\delta$ and $m^\delta_0=\rho^\delta_0 u^\delta_0.$    Then, the elliptic estimates (see \cite{ADN}) of \eqref{52} leads to
\begin{equation}\label{CA601}
\lim_{\delta\rightarrow 0}(\|\rho^\delta_0-\rho_0\|_{W^{1, q}}+\|u^\delta_0-u_0\|_{H^1}+\|\mathrm{curl}(u_0^\delta-u_0)\|_{H^{1/2}(\partial\Omega)})=0.
\end{equation}

Now we can utilize Lemmas \ref{u21}--\ref{qou10} to construct a unique global strong solution $(\rho^{\delta}, u^{\delta})$ to the system \eqref{11}--\eqref{15} with the initial value $(\rho_0^{\delta}, u_0^{\delta})$. In particular, Lemma \ref{qou10} ensures that
$(\rho,u)$ satisfy  \eqref{434} for any $T>0$ and the constant $\tilde{\mathbf{C}}$ is independent of $\delta$. Thus the standard compactness and uniqueness assertions (see\cite{germain2011weak,perepelitsa2006global,vaigant1995})  together with \eqref{CA601} implies that $(\rho^\delta,u^\delta)$ will converge to the unique global strong solution $(\rho, u)$ to the system \eqref{11}--\eqref{15} as $\delta\rightarrow 0$, with \eqref{19} valid. Meanwhile, \eqref{3103} and \eqref{CC416}   guarantee that $(\rho,u)$ also admits the large time behaviours \eqref{md} and \eqref{zs}. The proof of the theorem is therefore completed.
\thatsall

%\section*{Acknowledgements} The authors would like to thank Prof. Guocai Cai for his valuable discussions.  The research  is partially supported by the National Center for Mathematics and Interdisciplinary Sciences, CAS, NSFC Grant Nos.  11688101,  11525106, and 12071200,  and Double-Thousand Plan of Jiangxi Province (No. jxsq2019101008).


\begin {thebibliography} {99}

%\bibitem{adams2003sobolev} R.  A.  Adams and J.  J.  Fournier,   Sobolev spaces,   {Elsevier}   2003

\bibitem{AM}
G. Alessandrini, R. Magnanini, The Index of Isolated Critical Points and Solutions of Elliptic Equations in the Plane.
Annali della Scuola Normale Superiore di Pisa,
{\bf 4}(19) (1992), 
567-589.




\bibitem{ADN}
S. Agmon, A. Douglis, L. Nirenberg, Estimates near the boundary for solutions of elliptic partial differential equations satisfying general boundary conditions II.
Commun. Pure Appl. Math., \textbf{17}(1), 35-92 (1964).


\bibitem{ACJ}
C.J. Amick, 
On Leray’s problem of steady Navier-Stokes flow past a body in the plane.
Acta Math. 
{\bf 161} (1988), 
71–130.


\bibitem{Aramaki2014Lp}
J.  Aramaki,
$L^p$ theory for the div-curl system.
Int.  J.  Math.  Anal.,
{\bf 8}(6) (2014),
259-271.




\bibitem{beale1984remarks}J.  T.  Beale, T.  Kato, A.  Majda,
  Remarks on the breakdown of smooth solutions for the 3-D Euler equations.
  Commun.  Math.  Phys.,
  {\bf 94}(1) (1984),
  61-66.

\bibitem{B}
S. Bergman,
The kernel function and conformal mapping. 
American Mathematical Society, 1950.


\bibitem{caili01}G. C. Cai,  J. Li, Existence and exponential growth of  global classical solutions to the  compressible Navier-Stokes equations with slip boundary conditions in 3D bounded domains. Indiana University Mathematics Journal,
{\bf 72}(6) (2023), 2491-2546.

\bibitem{CM}
R. Coifman, Y. Meyer, On commutators of singular integrals.
Trans. Amer. Math. Soc.,
{\bf 212} (1975), 315-331.


% \bibitem{desjardins1997regularity} B.  Desjardins,   Regularity results for two-dimensional flows of multiphase viscous fluids,  {Arch.  Ration.  Mech.  Anal. },  {137}(2)(1997),  135-158.

\bibitem{2010Partial}L.  C.  Evans, 
\textit{Partial Differential Equations. Second edition.} Graduate Studies in Mathematics, 19. American Mathematical Society, Providence, RI, 2010.

\bibitem{FLL}
X. Fan, J. Li, J. Li, Global Existence of Strong and Weak Solutions to 2D Compressible
Navier-Stokes System in Bounded Domains with Large Data and Vacuum.
Arch. Rational Mech. Anal.,
{\bf 245} (2022), 239-278;

\bibitem{flw}
X. Fan, J. Li, X. Wang, Large-time behavior of the 2D compressible Navier-Stokes system in bounded domains with large data and vacuum. arXiv:2310.15520

\bibitem{feireisl2004dynamics}E.  Feireisl,   A. Novotn\'{y}, H.  Petzeltov\'{a}, On the existence of globally defined weak solutions to
the Navier-Stokes equations. J. Math. Fluid Mech., {\bf 3} (2001),   358-392.

\bibitem{fmd}  H. Frid, D. Marroquin,  J. F. C. Nariyoshi,
Global smooth solutions with large data for a system modeling aurora type phenomena in the 2-Torus.
SIAM J. Math. Anal.,  {\bf 53} (2021),  1122-1167.


\bibitem{Galdi}
G.P. Galdi,
\textit{An Introduction to the Mathematical Theory of the Navier-Stokes Equations. Steady-State Problems.} Springer, Berlin 2011


\bibitem{germain2011weak} P. Germain,
Weak--strong uniqueness for the isentropic compressible Navier-Stokes system.
J.  Math.  Fluid Mech.,
{\bf 13}  (2011),
137-146.


\bibitem{GT}D.  Gilbarg, N.  S.  Trudinger,
\textit{Elliptic partial differential equations of second order. Second edition.}
Springer-Verlag, 1983.

\bibitem{GW}
D. Gilbarg, H. F. Weinberger,
Asymptotic properties of Leray's  solution of the stationary two-dimensional Navier-Stokes equations. 
Russian Math. Surv. 
{\bf 29} (1974),
109–123 .

\bibitem{GOL}
G. M. Goluzin,
\textit{Geometric theory of Functions of a complex variable.}
American Mathematical Society, Providence, RI, 1969.


\bibitem{1995Global} D.  Hoff,
Global solutions of the Navier-Stokes equations for multidimensional compressible flow with discontinuous initial data.
{J. Differential Equations},
{\bf 120}, (1995),
215-254.

\bibitem{hoff2005} D.  Hoff,
Compressible flow in a half-space with Navier boundary conditions.
J.  Math.  Fluid Mech.,
{\bf 7}(3) (2005),
315-338.

\bibitem{huang2016existence} X. D.  Huang, J.  Li,
  Existence and blowup behavior of global strong solutions to the two-dimensional barotropic compressible Navier-Stokes system with vacuum and large initial data.
  J.  Math.  Pures Appl.,   {\bf 106}  (2016),  123-154.

 

\bibitem{hlx21}X. Huang, J. Li, Z. P. Xin, Global well-posedness of classical solutions with large oscillations and vacuum to the three-dimensional isentropic compressible Navier-Stokes equations. Commun. Pure Appl. Math., {\bf 65} (2012), 549-585.
 
\bibitem{KPR}
M. V. Korobkov, K. Pileckas, R. Russo,
Solution of Leray’s problem for stationary
Navier-Stokes equations in plane and
axially symmetric spatial domains,
Ann. of Math.,
{\bf 181} (2015), 
769–807.
 
 
\bibitem{Leray}
J. Leray,
\'{E}tude de diverses \'{e}quations int\'{e}grales non lin\'{e}aires et de quelques problèmes que pose l’hydrodynamique. 
J. Math. Pures Appl. 
{\bf 12} (1933),
1–82.

\bibitem{lx01} J. Li, Z. P. Xin, Global well-posedness and large time asymptotic
behavior of classical solutions to the compressible
Navier-Stokes equations with vacuum. Annals of PDE, {\bf 5}(7), (2019).

\bibitem{LY}
H. Liu, X. P. Yang, 
Critical points and level sets of Grushin-Harmonic functions in the plane,
J. Anal. Math.
{\bf 143}(2) (2021),
435-460.
 

\bibitem{1998Mathematical} P.  L.  Lions,
\textit{Mathematical topics in fluid mechanics, vol.  2, compressible models.} Oxford University Press, New York, 1998.



\bibitem{1980The} A.  Matsumura, T.  Nishida,
The initial value problem for the equations of motion of viscous and heat-conductive gases.
J. Mathematics of Kyoto University,
{\bf 20}  (1980),
67-104.

\bibitem{NA}
J. Nash,  
Le probl\`{e}me de Cauchy pour les \'{e}quations diff\'{e}rentielles d'un fluide g\'{e}n\'{e}ral. Bull. Soc. Math. France. 
\textbf{90} (1962), 487-497.

\bibitem{novotny2004introduction} A.  Novotn\'{y}, I. Stra\v{s}kraba,
{Introduction to the mathematical theory of compressible flow}.
{27},
{Oxford Lecture Ser. Math. Appl. Oxford Univ.
Press, Oxford, 2004}.



\bibitem{perepelitsa2006global} M.  Perepelitsa,
On the global existence of weak solutions for the Navier--Stokes equations of compressible fluid flows.
SIAM J.  Math.  Anal., 
{\bf 38} (2006), 1126-1153.


\bibitem{weko_39341_1}
R.  Salvi and I. Straskraba,
Global existence for viscous compressible fluids and their behavior as $t \to \infty$.
J.  Fac.  Sci.    Uni.  Tokyo.  Sect.  1A,  Math.,
{\bf 40}(1) (1993), 17-51.

\bibitem{ser1}J. Serrin, On the uniqueness of compressible fluid motion. Arch. Ration. Mech. Anal., 
{\bf 3} (1959), 271-288.
 


\bibitem{1972The} V.  A.  Solonnikov,
The Green's matrices for elliptic boundary value problems. II,
{Proceedings of the Steklov Institute of Mathematics},
{\bf 116} (1971),
187-226.


\bibitem{Solonnikov1980}
V.  A.  Solonnikov,
Solvability of the initial-boundary-value problem for the equations of motion of a viscous compressible fluid.
J.  Math.  Sci.,
{\bf 14} (1980),
1120-1133.


\bibitem{2007Complex}
E. M. Stein, R. Shakarchi, \textit{Complex Analysis.} Princeton University Press,  2003.

\bibitem{RT}
R. Temam,
\textit{Navier-Stokes Equations. Theory and numerical Analysis.}
American Mathematical Society, Providence (1984).

 
\bibitem{1961On}  S.  E.  Warschawski,
On differentiability at the boundary in conformal mapping.
{Proceedings of the American Mathematical Society},
{\bf 12}  (1961),
614-620.


\bibitem{VW}W. Von Wahl,
Estimating $u$ by div$u$ and curl$u$.
Math.  Methods  Appl.  Sci.,
{\bf 15}  (1992),
123-143.

\bibitem{vaigant1995}V.  A.  Vaigant, A.  V.  Kazhikhov,
On existence of global solutions to the two-dimensional Navier-Stokes equations for a compressible viscous fluid.
Sib.  Math.  J.,
{\bf 36}  (1995),
1108-1141.

\bibitem{zlt}
A. A. Zlotnik,
Uniform estimates and stabilization of symmetric solutions of a system of quasilinear equations.
Differ. Equ.,
\textbf{36}(5) (2000),
701-716.

\end{thebibliography}
\end{document}